\documentclass[times]{amsart}

\usepackage{amsmath,amsopn,amssymb,amsthm,mathtools,mathrsfs,yhmath,tensor,graphicx,geometry,stmaryrd,tabulary,booktabs,leftindex,dsfont}
\usepackage[euler]{textgreek}

\usepackage{xcolor}

	\definecolor{UCIB}{HTML}{0064A4}
	\definecolor{UCSDB}{HTML}{00629B}

\usepackage{accents}

\usepackage{upgreek,textgreek}
	
\usepackage{hyperref}
\hypersetup{backref, urlcolor = UCIB, citecolor=blue, colorlinks=true}

\newcommand{\bhyperlink}[3][black]{\hyperlink{#2}{\color{#1}{#3}}}%

\usepackage{ifthen}

\usepackage{braket}

\newcommand{\mybinom}[3][0.8]{\scalebox{#1}{$\dbinom{#2}{#3}$}}
	
\usepackage{lipsum} 

\usepackage{tikz}
\usetikzlibrary{calc}
\usetikzlibrary{intersections}
\usetikzlibrary{graphs,graphs.standard,quotes}
\usepackage{tikz-cd}

\usepackage{enumitem}

\usepackage{tikz}

\usepackage{pgfplots}
\pgfplotsset{compat=1.17}
\usetikzlibrary{quotes,angles}

\usepackage{booktabs}

\usepackage{nicematrix}

\usepackage{graphicx} 
\newcommand\bigsubset[1][1.19]{%
   \mathrel{\vcenter{\hbox{\scalebox{#1}{$\subset$}}}}}

\usepackage{xcolor}
\theoremstyle{plain}
\newtheorem{fthm}{Theorem}[section]
\newtheorem*{fthm*}{Theorem}
\newtheorem{flemma}{Lemma}[section]
\newtheorem*{flemma*}{Lemma}
\newtheorem{fprop}{Proposition}[section]
\newtheorem*{fprop*}{Proposition}

\newtheorem*{fcor*}{Corollary}

\theoremstyle{definition}
\newtheorem{fdefi}{Definition}[section]
\newtheorem*{fdefi*}{Definition}
\newtheorem{fexmp}{Example}[section]
\newtheorem*{fexmp*}{Example}

\theoremstyle{remark}
\newtheorem{frmk}{Remark}[section]
\newtheorem*{frmk*}{Remark}
\newtheorem{fconj}{Conjecture}[section]
\newtheorem*{fconj*}{Conjecture}

\newtheorem*{fclaim*}{Claim}

\newtheorem*{fquest*}{Question}

\DeclareMathOperator\dist{dist}
\DeclareMathOperator{\arccosh}{arccosh}
\DeclareMathOperator\res{res}
\DeclareMathOperator\disc{discr}

\newcommand{\Address}{{
  \bigskip
  \footnotesize

  Chao-Ming Lin, \textsc{Department of Mathematics, Ohio State University, OH}\par\nopagebreak
  \textit{E-mail address:} \href{lin.4579@osu.edu}{lin.4579@osu.edu}\par\nopagebreak
  \textit{Personal Website:} \href{https://chaominl.github.io}{https://chaominl.github.io}

}}

\newcommand{\Acknow}{{
  \bigskip
\textbf{Acknowledgements:} The author is grateful to Zhiqin~Lu and Xiangwen~Zhang for giving me enlightening help. The author would like to thank Tristan Collins for his interest in this work. The author would like to thank Chen-Yu~Chi for several helpful conversations and Hsin-Po Wang for some programming support.
The author would like to thank Biao Ma and Hao Fang for pointing out a mistake in the proof of the convexity theorem in the previous version, we give a new proof and fix this mistake.
}}

\begin{document}

\title{On the Convexity of General Inverse $\sigma_k$ Equations}
\author{Chao-Ming Lin}

\begin{abstract}
We prove that if a level set of a degree $n$ general inverse $\sigma_k$ equation $f(\lambda_1, \cdots, \lambda_n)$ $\coloneqq \lambda_1 \cdots \lambda_n - \sum_{k = 0}^{n-1} c_k \sigma_k(\lambda) = 0$ is contained in $q + \Gamma_n$ for some $q \in \mathbb{R}^n$, where $c_k$ are real numbers not necessary to be non-negative and $\Gamma_n$ is the positive orthant, then this level set is convex. As an application, this result justifies the convexity of the level set of all general inverse $\sigma_k$ type equations, for example, the Monge--Ampère equation, the Hessian equation, the J-equation, the deformed Hermitian--Yang--Mills equation, the special Lagrangian equation, etc. Moreover, we find a numerical condition to verify whether a level set of a general inverse $\sigma_k$ equation is contained in $q + \Gamma_n$ for some $q \in \mathbb{R}^n$, which is a way to determine the convexity of this level set.
\end{abstract}
\maketitle
\vspace{-0.6cm}

\section{Introduction}

Let $(M, \omega)$ be a compact connected Kähler manifold of complex dimension $n$ with a Kähler form $\omega$ and $[\chi_0] \in H^{1,1}(M; \mathbb{R})$, where $H^{1,1}(M; \mathbb{R})$ is the $(1, 1)$-Dolbeault cohomology group. The study of the solvability of the following equation is widely considered:
\begin{align*}
\label{eq:1.1}
\chi^n = c_{n-1} \mybinom[0.8]{n}{n-1} \chi^{n-1}  \wedge \omega^{1} + \cdots + c_{1} \mybinom[0.8]{n}{1} \chi^{1}  \wedge \omega^{n-1} + c_{0} \mybinom[0.8]{n}{0}   \omega^{n}    =  \sum_{k = 0}^{n-1} c_k \mybinom[0.8]{n}{k} \chi^{k}  \wedge \omega^{n-k}, \tag{1.1}
\end{align*}where $c_k$ are real functions on $M$ and $\chi \in [\chi_0]$ is a real smooth, closed $(1, 1)$-form. We call an equation having the same format as equation (\ref{eq:1.1}) a degree $n$ \emph{general inverse $\sigma_k$ type equation}. A general inverse $\sigma_k$ type equation (\ref{eq:1.1}) is very likely to be ill-posed, but some special combinations of the coefficients raise some famous equations. For example, by letting $[\chi_0]$ be a Kähler class, $c_k = 0$ for all $k \in \{1, \cdots, n-1\}$, and $c_0$ be a positive function, equation (\ref{eq:1.1}) becomes the complex Monge--Ampère equation in the Calabi conjecture \cite{calabi1954kahler, calabi1957kahler}, which was solved by Yau \cite{yau1978ricci}. Inspired by the study of the Hermitian--Yang--Mills connections by Donaldson \cite{donaldson1985anti} and Uhlenbeck--Yau \cite{uhlenbeck1986existence}, Donaldson \cite{donaldson1999moment} studied the J-equation using the moment map. The J-equation was studied extensively by Collins--Székelyhidi \cite{collins2017convergence}, Chen \cite{chen2000lower}, Song--Weinkove \cite{song2008convergence}, and the references therein. The J-equation can be obtained by letting $[\chi_0]$ be a Kähler class, $c_k = 0$ for all $k \in \{0, \cdots, n-2\}$, and $c_{n-1}$ be a positive constant. Lejmi--Székelyhidi \cite{lejmi2015j} conjectured that the existence of the solution to the J-equation is equivalent to a certain stability condition. Chen \cite{chen2021j} studied the solvability and the stability of the J-equation using a Nakai--Moishezon type criterion and proved this conjecture under a slightly stronger condition. The Nakai--Moishezon type criterion was inspired by the work of Demailly--Păun \cite{demailly2004numerical}. Song \cite{song2020nakai} extended the method of Chen \cite{chen2021j} and confirmed the conjecture by Lejmi--Székelyhidi \cite{lejmi2015j}. Chen \cite{chen2021j} also extended the J-equation slightly so that the real function $c_0$ can be slightly negative but satisfies the integral condition. The general inverse $\sigma_k$ equation was raised by Chen \cite{chen2000lower} and some special cases were treated by Collins--Székelyhidi \cite{collins2017convergence} and Fang--Lai--Ma \cite{fang2011class}. Collins--Székelyhidi \cite{collins2017convergence} considered the case when $c_k$, for $k \in \{0, \cdots, n-1\}$, are non-negative constants satisfying the integral condition. Collins--Székelyhidi proved that if there exists a $C$-subsolution (introduced by Székelyhidi \cite{szekelyhidi2018fully} and Guan \cite{guan2014second} and will be discussed later in Section~\ref{sec:2.2}), then this special case is solvable. Datar--Pingali \cite{datar2021numerical} extended the techniques in Chen \cite{chen2021j} and Song \cite{song2020nakai} to this special case of the general inverse $\sigma_k$ equation. Motivated by mirror symmetry in string theory, the {deformed Hermitian--Yang--Mills equation}, which will be abbreviated to the dHYM equation from now on, was discovered around the same time by Mariño--Minasian--Moore--Strominger \cite{marino2000nonlinear} and Leung--Yau--Zaslow \cite{leung2000special} using different points of view. The dHYM equation was initiated by Jacob--Yau \cite{jacob2017special} and can be formulated as follows:
\begin{align}
\label{eq:1.2}
\Im \bigl ( \omega + \sqrt{-1} \chi \bigr )^n &= \tan \bigl (   {\theta}   \bigr ) \cdot \Re \bigl ( \omega + \sqrt{-1} \chi \bigr )^n. \tag{1.2}
\end{align}Here $\Im$ and $\Re$ are the imaginary and real parts, respectively, and $ {\theta}$ is a topological constant determined by the cohomology classes $[\omega]$ and $[\chi_0]$. If the phase ${\theta}$ lies in $\bigl( {(n-2)} \pi/2, n\pi/2 \bigr)$, that is, if $\theta$ is a supercritical phase, Collins--Jacob--Yau \cite{collins20151} showed that if there exists a supercritical $C$-subsolution (which will be introduced and defined later in Section~\ref{sec:2.2}), then the dHYM equation is solvable. Collins--Jacob--Yau conjectured that the existence of the solution to the dHYM equation (\ref{eq:1.2}) is equivalent to a certain stability condition for all analytic subvarieties and confirmed this numerical conjecture for complex surfaces. Chen \cite{chen2021j} proved a Nakai--Moishezon type criterion for the dHYM equation when the phase is supercritical under a slightly stronger condition that these holomorphic intersection numbers have a uniform lower bound independent of analytic subvarieties. Chu--Lee--Takahashi \cite{chu2021nakai} improved the result by Chen \cite{chen2021j} without assuming a uniform lower bound for these holomorphic intersection numbers on projective manifolds. Collins--Jacob--Yau \cite{collins20151} also conjectured that when $\theta$ is supercritical, if there exists a $C$-subsolution, then the dHYM equation is solvable. Lin \cite{lin2022} confirmed this conjecture by Collins--Jacob--Yau \cite{collins20151} of the solvability of the dHYM equation when the complex dimension equals three or four.  We should emphasize that there are many significant works that have been done recently. The interested reader is referred to \cite{chu2022hypercritical, collins2020stability, collins2021moment, collins2018deformed, jacob2019weak, jacob2020deformed, joyce2011existence, lin2020, lu2022dirichlet, phong2017fu, phong2018anomaly, phong2019estimates, phong2021fu, pingali2019deformed, schlitzer2021deformed, schoen2003volume, wang2013singular} and the references therein. We also want to remark that equation (\ref{eq:1.1}) also plays an important role on reals. For example, Caffarelli--Nirenberg--Spruck \cite{caffarelli1985dirichlet}, Krylov \cite{krylov1993lectures, krylov1997fully}, and Trudinger \cite{trudinger1995dirichlet} on the study of the Dirichlet problem for the Hessian equations on various settings. Also, Guan--Zhang \cite{guan2019class} studied the solvability of a general class of curvature equations in convex geometry. \smallskip

Throughout all these works, the convexity of either the equation itself or the level set plays a crucial role. To be more precise, to get a priori estimates, we highly rely on convexity. If we write equation (\ref{eq:1.1}) in terms of the eigenvalues of the Hermitian endomorphism $\Lambda = \omega^{-1} \chi$ at a point, then we can rewrite equation (\ref{eq:1.1}) as
\begin{align*}
\label{eq:1.3}
\lambda_1 \cdots \lambda_n = \sum_{k = 0}^{n-1} c_k \sigma_k(\lambda_1, \cdots, \lambda_n) = \sum_{k = 0}^{n-1} c_k \sigma_k(\lambda) , \tag{1.3}
\end{align*}where $\lambda_i$ are the eigenvalues of $\Lambda$, $\sigma_k(\lambda_1, \cdots, \lambda_n)$ is the $k$-th elementary symmetric polynomial of $\{\lambda_1, \cdots, \lambda_n\}$, and we denote $\sigma_k(\lambda_1, \cdots, \lambda_n)$ by $\sigma_k(\lambda)$ for convenience. The following multivariate polynomial in $n$ variables $\{ \lambda_1, \cdots, \lambda_n\}$
\begin{align*}
\label{eq:1.4}
\lambda_1 \cdots \lambda_n - \sum_{k = 0}^{n-1} c_k \sigma_k(\lambda_1, \cdots, \lambda_n)  \tag{1.4}
\end{align*}is a special case of multilinear polynomials, that is, multivariate polynomials in which no variable occurs to a power of two or higher. We will call a multilinear polynomial having the same format as (\ref{eq:1.4}) a general inverse $\sigma_k$ type multilinear polynomial.\smallskip

For the classical example, the complex Monge--Ampère equation considered by Yau \cite{yau1978ricci}, at a fixed point, can be rewritten as
\begin{align*}
\label{eq:1.5}
\lambda_1 \cdots \lambda_n = c, \tag{1.5}
\end{align*}where $c$ is a positive value. If we write $h = c/(\lambda_1 \cdots \lambda_n)$, then the Hessian of $h$ will always be positive-definite when $(\lambda_1, \cdots, \lambda_n) \in \Gamma_n$, where $\Gamma_n$ is the positive orthant. This implies that equation (\ref{eq:1.5}) is strictly convex, hence the level set $\{ \lambda_1 \cdots \lambda_n - c = 0 \}$ is convex as well. In Fang--Lai--Ma \cite{fang2011class}, Fang--Lai--Ma considered some special combinations of non-negative constants $c_k$ for $k \in \{0, \cdots, n-1\}$ and proved the convexity of equation (\ref{eq:1.3}). Collins--Székelyhidi \cite{collins2017convergence} generalized their results and proved the convexity of equation (\ref{eq:1.3}) when $c_k$ are non-negative constants for $k \in \{0, \cdots, n-1\}$. For the special Lagrangian equation and the dHYM equation, Yuan \cite{yuan2006global} showed that if the phase is supercritical, then the level set of equation (\ref{eq:1.2}) is a smooth convex hypersurface even though the equation might not be convex. Collins--Jacob--Yau \cite{collins20151} obtained a priori estimates from the convexity of the level set of equation (\ref{eq:1.2}). Since we mainly focus on the general inverse $\sigma_k$ equations in this work and because of the space limitations, the interested reader is referred to \cite{caffarelli2000priori, caffarelli1985dirichlet, chen2021j, collins2017convergence, fang2013convergence, fang2011class, gilbarg2015elliptic, guan1999dirichlet, hou2010second, lejmi2015j, siu2012lectures, trudinger1995dirichlet} and the references therein.\smallskip

In this work, we will concentrate on the convexity of the level set of equation (\ref{eq:1.3}), which gives a huge hope on the solvability of equation (\ref{eq:1.3}). In \cite{lin2022}, when complex dimension equals three or four, the author gave some constraints on the coefficients of equation (\ref{eq:1.3}) and proved the strict convexity of the level set of equation (\ref{eq:1.3}). The author also proved the solvability by finding a nice path under these constraints connecting the supercritical dHYM equation (\ref{eq:1.2}) to the general inverse $\sigma_k$ equation with non-negative coefficients considered by Collins--Székelyhidi \cite{collins2017convergence} and Fang--Lai--Ma \cite{fang2011class} and by obtaining a priori estimates on this path.  \smallskip

Let us state some of our settings, definitions, and results now. First, we introduce the following stableness condition for general inverse $\sigma_k$ type multilinear polynomials.

\hypertarget{D:1.1}{\begin{fdefi}[\texorpdfstring{$\Upsilon$}{}-stableness]}
Let $f(\lambda) \coloneqq   \lambda_1 \cdots \lambda_n - \sum_{k = 0}^{n-1} c_k \sigma_k(\lambda)$ be a general inverse $\sigma_k$ type multilinear polynomial and $\Gamma_f^{n}$ be a connected component of $\{f(\lambda) > 0\}$. We say that this connected component $\Gamma^n_f$ of $f(\lambda)$ is $\Upsilon$-stable if
\begin{align*}
\Gamma^n_f \subseteq  q +  \Gamma_n \text{ for some } q \in \mathbb{R}^n, 
\end{align*}where $\Gamma_n$ is the positive orthant of $\mathbb{R}^n$. We say that this connected component $\Gamma^n_f$ is strictly $\Upsilon$-stable if it is $\Upsilon$-stable and the boundary $\partial \Gamma^n_f$ is contained in the $\Upsilon_1$-cone. 
\end{fdefi} The $\Upsilon_k$-cones will be defined later in Section~\ref{sec:2.2} for $k \in \{1, \cdots, n-1\}$. In particular, the $\Upsilon_1$-cone is the $C$-subsolution cone introduced by Székelyhidi \cite{szekelyhidi2018fully} and Guan \cite{guan2014second}. With the $\Upsilon$-stableness condition, in Section~\ref{sec:3}, we prove that the boundary $\partial \Gamma^n_f$ of $\Gamma^n_f$ will be convex if $\Gamma^n_f$ is strictly $\Upsilon$-stable. In addition, we show that a connected component of $\{f = 0\}$ is contained in $q + \Gamma_n$ for some $q \in \mathbb{R}^n$ is equivalent to $\Gamma^n_f$ of this connected component being strictly $\Upsilon$-stable. In this case, this connected component equals $\partial \Gamma^n_f$. We have the following main result.

\hypertarget{T:1.1}{\begin{fthm}[Convexity of the general inverse $\sigma_k$ equation]}
Consider the following general inverse $\sigma_k$ equation $f(\lambda) \coloneqq \sigma_n(\lambda) - \sum_{k = 0}^{n-1} c_k \sigma_k(\lambda) = 0$, where $\sigma_k$ is the $k$-th elementary symmetric polynomial and $c_k$ are real numbers not necessary to be non-negative. Let $\Gamma_f^{n}$ be a connected component of $\{f(\lambda) > 0\}$. If $\Gamma^n_f$ is strictly $\Upsilon$-stable, then the boundary $\partial \Gamma^n_f$ is convex. 
\end{fthm}

The following general inverse $\sigma_k$ type equations are all strictly $\Upsilon$-stable, we will verify some of them in Section~\ref{sec:4}.

\hypertarget{R:1.1}{\begin{frmk}}
The following general inverse $\sigma_k$ type equations are all strictly $\Upsilon$-stable:
\begin{itemize}[leftmargin = 2cm]
\item Complex Monge--Ampère equation.
\item J-equation. 
\item Hessian equation. 
\item Deformed Hermitian--Yang--Mills equation with supercritical phase.
\item Special Lagrangian equation with supercritical phase.
\item General inverse $\sigma_k$ equation with non-negative $c_k$ for $k \in \{0, \cdots, n-1\}$.
\end{itemize}
\end{frmk}

In practice, verifying the $\Upsilon$-stableness condition is not easy. Here, we introduce the following class of special univariate polynomials which plays an important role in determining the convexity of general inverse $\sigma_k$ equations. In Section~\ref{sec:2.1}, we will show more special properties of these special univariate polynomials. Now, we list some definitions and some interesting and important results.

\hypertarget{D:1.2}{\begin{fdefi}[Noetherian polynomial]}
We say a degree $n$ real univariate polynomial $p(x)$ is right-Noetherian if for all $k \in \{0, \cdots, n-2 \}$, there exists a real root of $p^{(k)}$ which is greater than or equal to the largest real root of $p^{(k+1)}$; left-Noetherian if for all $k \in \{0, \cdots, n-2 \}$, there exists a real root of $p^{(k)}$ which is less than or equal to the smallest real root of $p^{(k+1)}$. Here $p^{(k)}$ is the $k$-th derivative of $p$. We say a right-Noetherian polynomial $p(x)$ is strictly right-Noetherian if the largest real root of $p(x)$ is strictly greater than the largest real root of $p'(x)$.
\end{fdefi}

In Section~\ref{sec:2.2}, we will show that the right-Noetherianness condition is equivalent to the $\Upsilon$-stableness condition in the following sense. We get the following Positivstellensatz-type result generalizing the work in \cite{lin2022}. When the degree is small, we can explicitly write down the constraints using the resultants, see Section~\ref{sec:4} for more examples when the degree equals three or four.

\hypertarget{T:1.2}{\begin{fthm}[Positivstellensatz]}
A general inverse $\sigma_k$ type multilinear polynomial 
\begin{align*}
f(\lambda) = f(\lambda_1, \cdots, \lambda_n) = \lambda_1 \cdots \lambda_n - \sum_{k = 0}^{n-1} c_k \sigma_k(\lambda)
\end{align*}has a connected component $\Gamma^n_f$ which is $\Upsilon$-stable if and only if the diagonal restriction $r_f(x)$ of $f(\lambda)$, which is defined by the following  
\begin{align*}
r_f(x) \coloneqq f(x, \cdots, x) = x^n - \sum_{k=0}^{n-1} c_k \mybinom[0.8]{n}{k} x^k,
\end{align*} is right-Noetherian. Moreover, $\Gamma^n_f$ is strictly $\Upsilon$-stable if and only if $r_f$ is strictly right-Noetherian.
\end{fthm}

As an application of the \bhyperlink{T:1.2}{Positivstellensatz Theorem}, in Section~\ref{sec:4}, we will verify some general inverse $\sigma_k$ type equations including the general inverse $\sigma_k$ equation with non-negative coefficients and the dHYM equation. As a quick consequence of the \bhyperlink{T:1.2}{Positivstellensatz Theorem}, we can show that the level set of the following general inverse $\sigma_k$ equation is convex. This is also numerical checkable, which gives a large quantity of new convex sets.

\hypertarget{E:1.1}{\begin{fexmp}}
The following univariate polynomial $r_f(x) = x^5 - \sum_{k = 0}^{3} c_k \binom{5}{k}x^k$ with $c_3 = 19, c_2 = -64, c_1 = 9$, and $c_0 = -20$ is strictly right-Noetherian. This is checkable using any computer. By rounding off to the third decimal place, we have
\begin{align*}
x_0 \sim 11.632, \  x_1 \sim 9.306, \  x_2 \sim 6.909, \  x_3 \sim 4.359, \  x_4 =0.
\end{align*}Here, for $k \in \{0, \cdots, n-1\}$, we denote by $x_k$ the largest real root of the $k$-th derivative $r_f^{(k)}(x)$. This implies that the level set of the following general inverse $\sigma_k$ equation is convex 
\begin{align*}
f(\lambda) &= \lambda_1 \cdots \lambda_5 - \sum_{k = 0}^3 c_k \sigma_k(\lambda) = \lambda_1 \cdots \lambda_5 -   19 \sigma_3(\lambda)  + 64  \sigma_2(\lambda) - 9   \sigma_1(\lambda)   + 20 = 0.
\end{align*}
\end{fexmp}

If a general inverse $\sigma_k$ type multilinear polynomial has an $\Upsilon$-stable connected component, then, for convenience, we say this general inverse $\sigma_k$ type multilinear polynomial is $\Upsilon$-stable. In the following setting, we can also compare two $\Upsilon$-stable general inverse $\sigma_k$ type multilinear polynomials.

\hypertarget{D:1.3}{\begin{fdefi}[\texorpdfstring{$\Upsilon$}{}-dominance]}
Let $f(\lambda) \coloneqq   \lambda_1 \cdots \lambda_n - \sum_{k = 0}^{n-1} c_k \sigma_k(\lambda)$ and $g(\lambda) \coloneqq   \lambda_1 \cdots \lambda_n - \sum_{k = 0}^{n-1} d_k \sigma_k(\lambda)$ be two $\Upsilon$-stable general inverse $\sigma_k$ type multilinear polynomials. For $k \in \{0, \cdots, n-1\}$, we write $x_k$ the largest real root of the diagonal restriction $r_f^{(k)}$ of $f$ and $y_k$ the largest real root of the diagonal restriction $r_g^{(k)}$ of $g$. If $y_k \geq x_k$ for all $k \in \{0, \cdots, n-1\}$, then we say $g \gtrdot f$. 
\end{fdefi}

We then get another Positivstellensatz-type result. This result implies that for $\Upsilon$-stable general inverse $\sigma_k$ type multilinear polynomials, the $\Upsilon$-dominance is equivalent to the set inclusion.

\hypertarget{T:1.3}{\begin{fthm}[\texorpdfstring{$\Upsilon$}{}-dominance]}
Let $f(\lambda) \coloneqq   \lambda_1 \cdots \lambda_n - \sum_{k = 0}^{n-1} c_k \sigma_k(\lambda)$ and $g(\lambda) \coloneqq   \lambda_1 \cdots \lambda_n - \sum_{k = 0}^{n-1} d_k \sigma_k(\lambda)$ be two $\Upsilon$-stable general inverse $\sigma_k$ type multilinear polynomials. Then $g \gtrdot f$ if and only if $\Gamma^n_g \subset \Gamma^n_f$.
\end{fthm}

\hypertarget{E:1.2}{\begin{fexmp}}
The following univariate polynomial $r_g(x) = x^5 - \sum_{k = 0}^{3} d_k \binom{5}{k}x^k$ with $d_3 = 19, d_2 = 65, d_1 = -2$, and $d_0 = -24$ is strictly right-Noetherian with roots:
\begin{align*}
y_0 \sim 15.250, \  y_1 \sim 11.673, \  y_2 \sim 8.066, \  y_3 \sim 4.359, \  y_4 =0.
\end{align*}Here, for $k \in \{0, \cdots, 4\}$, we denote by $y_k$ the largest real root of the $k$-th derivative $r_g^{(k)}(x)$. 
We compare this $\Upsilon$-stable general inverse $\sigma_k$ type multilinear polynomial with the one in Example~\hyperlink{E:1.1}{1.1}. Since $y_0 > x_0$, $y_1 > x_1$, $y_2 > x_2$, $y_3 = x_3$, and $y_4 = x_4$, we have $g \gtrdot f$. By Theorem~\hyperlink{T:1.3}{1.3}, we get
\begin{align*}
&\kern-2em  \{ \lambda_1 \cdots \lambda_5 -   19 \sigma_3(\lambda)  - 65  \sigma_2(\lambda) + 2   \sigma_1(\lambda)   + 24 > 0 \}  \\
&\subset \{ \lambda_1 \cdots \lambda_5 -   19 \sigma_3(\lambda)  + 64  \sigma_2(\lambda) - 9   \sigma_1(\lambda)   + 20 > 0 \}.
\end{align*}
\end{fexmp}

In this work, we find a way to determine the convexity of any general inverse $\sigma_k$ equation $f(\lambda) = \lambda_1 \cdots \lambda_n - \sum_{k = 0}^{n-1} c_k \sigma_k(\lambda) = 0$ using either the strict $\Upsilon$-stableness condition or equivalently the strict right-Noetherianness condition. Our next goal is to determine the solvability of any general inverse $\sigma_k$ equation on a compact connected Kähler manifold satisfying either condition at every point on the manifold.\bigskip

Special univariate and multivariate polynomials are widely studied in many different fields. For example, for multivariate polynomials, the class of Lorentzian polynomials considered by Brändén--Huh \cite{branden2020lorentzian} and Huh--Matherne--Mészáros--Dizier \cite{huh2022logarithmic}, which contains the class of homogeneous stable polynomials including volume polynomials of convex bodies and projective varieties, is important in the study of matroid theory. In addition, the classes of strongly log-concave polynomials in Gurvits \cite{gurvits2009multivariate}, completely log-concave multivariate polynomials in Anari--Gharan--Vinzant \cite{anari2018log1}, Anari--Liu--Gharan--Vinzant \cite{anari2019log, anari2018log3}, and Anari--Liu--Gharan--Vinzant--Vuong \cite{anari2021log}, are also crucial in the matroid theory. In \cite{branden2020lorentzian}, Brändén--Huh proved that a homogeneous polynomial with non-negative coefficients is Lorentzian if and only if it is strongly log-concave.\smallskip

Here, we list one of many special properties of right-Noetherian polynomials. We will state the definitions explicitly and more properties in Section~\ref{sec:2.1}. We define the log-concavity ratio $\alpha_f(x)$ of an analytic univariate function $f(x)$ in Section~\ref{sec:2.1}. Roughly speaking, it is defined by
\begin{align*}
\alpha_f(x) \coloneqq \frac{f(x) \cdot f''(x)}{f'(x)^2}.
\end{align*}If $f(x) > 0$ and $\alpha_f(x) \leq 1$ on an open interval $I$, then $f$ will be logarithmically concave on $I$. We prove that this log-concavity ratio will be monotonic for right-Noetherian polynomials. In Figure~\ref{fig:1}, we plot the log-concavity ratio $\alpha_p(x)$ of $p(x) = x^5 - 19 \binom{5}{3}x^3 + 64 \binom{5}{2}x^2 - 9 \binom{5}{1}x^1 + 20 \binom{5}{0}x^0 = x^5 - 190x^3 + 640x^2 -45x + 20$, which is right-Noetherian by Example~\hyperlink{E:1.1}{1.1}.

\begin{figure}
\centering
\begin{tikzpicture}[xscale=1,yscale=1]
  	\draw[->] (10,0) -- (18.8,0) node[right] {$x$};
  	\draw[->] (10.5,-0.6) -- (10.5,1) ;
	  \draw[color=UCIB,domain=11:12,samples=100, smooth]    plot (\x,{    -0.0958521  +  0.793249*(\x - 11.5)  - 0.539957*(\x - 11.5)^2 + 0.327835*(\x - 11.5)^3 - 0.186714*(\x - 11.5)^4  + 0.102099*(\x - 11.5)^5   })       ;
	  \draw[color=UCIB,domain=11.8:13,samples=100, smooth]    plot (\x,{    0.36504  +  0.259051*(\x - 12.5)  - 0.120486*(\x - 12.5)^2 + 0.0501828*(\x - 12.5)^3 - 0.0196259*(\x - 12.5)^4  + 0.00737136*(\x - 12.5)^5    })       ;
	  \draw[color=UCIB,domain=12.8:14,samples=100, smooth]    plot (\x,{    0.539553  +  0.115623*(\x - 13.5)  - 0.0406717*(\x - 13.5)^2 + 0.012873*(\x - 13.5)^3 - 0.00383133*(\x - 13.5)^4  + 0.00109564*(\x - 13.5)^5    })       ;
	  \draw[color=UCIB,domain=13.8:15,samples=100, smooth]    plot (\x,{    0.624402  +  0.0616676*(\x - 14.5)  - 0.0173753*(\x - 14.5)^2 + 0.00442787*(\x - 14.5)^3 - 0.00106296*(\x - 14.5)^4  + 0.000245342*(\x - 14.5)^5    })       ;
	  \draw[color=UCIB,domain=14.8:16,samples=100, smooth]    plot (\x,{    0.67226  +  0.0369129*(\x - 15.5)  - 0.00864536*(\x - 15.5)^2 + 0.00184102*(\x - 15.5)^3 - 0.000370072*(\x - 15.5)^4  + 0.0000715804*(\x - 15.5)^5    })       ;
	  \draw[color=UCIB,domain=15.8:17,samples=100, smooth]    plot (\x,{    0.702058  +  0.0239564*(\x - 16.5)  - 0.00478709*(\x - 16.5)^2 + 0.000874198*(\x - 16.5)^3 - 0.000151033*(\x - 16.5)^4  + 0.0000251317*(\x - 16.5)^5    })       ;
	  \draw[color=UCIB,domain=16.8:18,samples=100, smooth]    plot (\x,{    0.721972  +  0.0165057*(\x - 17.5)  - 0.00286923*(\x - 17.5)^2 + 0.000458*(\x - 17.5)^3 - 0.0000693281*(\x - 17.5)^4  + 0.000010118*(\x - 17.5)^5    })       ;	  
  	\draw[red,dashed] (10.8,0.8) -- (18.6,0.8) node at (10.1,0.8) {$4/5$};
	\fill[black] (11.632,{0}) circle (0.08cm) node at  (12.3,{-0.4}) {$x_0 \sim 11.632$};
\end{tikzpicture}
\caption{$\alpha_{p}(x)$ of $p(x) = x^5 - 19 \binom{5}{3}x^3 + 64 \binom{5}{2}x^2 - 9 \binom{5}{1}x^1 + 20 \binom{5}{0}x^0$}
\label{fig:1}
\end{figure}

\hypertarget{T:1.4}{\begin{fthm}[Monotonicity of log-concavity ratio]}
Let $p(x)$ be a real univariate polynomial of degree $n$ which is right-Noetherian. Then the log-concavity ratio $\alpha_p(x)$ of $p$ is monotonically increasing on $(x_1, \infty)$ with value from $-\infty$ to $1 - 1/n$ if $x_0 > x_1$ and on $[x_0, \infty)$ with value from $1 - 1/m$ to $1 - 1/n$ if $x_0 = x_1$. Here, $x_0$ is the largest real root of $p$, $x_1$ is the largest real root of $p'$, and $m$ is the multiplicity of $p$ at $x_0$. In particular, if $p$ is right-Noetherian, then $p$ is always logarithmically concave when $x > x_0$.
\end{fthm}

As a quick consequence, we show that after a translation, the right-Noetherian polynomials will be strongly log-concave. The definitions are stated in Section~\ref{sec:2.1}.

\hypertarget{L:1.1}{\begin{flemma}}
Let $p(x)$ be a real univariate polynomial of degree $n$ which is right-Noetherian, then $p$ is strongly log-concave after the translation $x \mapsto x-x_0$. Here, $x_0$ is the largest real root of $p(x)$.
\end{flemma}

In Theorem~\hyperlink{T:1.1}{1.1}, we prove that the level set will be convex. But it is still open whether the level set will be strictly convex in the following sense. 

\hypertarget{C:1.1}{\begin{fconj}[Strict convexity of the general inverse $\sigma_k$ equation]}
Consider the following general inverse $\sigma_k$ equation $f(\lambda) \coloneqq \sigma_n(\lambda) - \sum_{k = 0}^{n-1} c_k \sigma_k(\lambda) = 0$. If the diagonal restriction $r_f(x)$ of $f(\lambda)$ is strictly right-Noetherian, then the level set $\{f = 0\}$ is strictly convex. That is, the $n-1 \times n-1$ Hessian matrix $\begin{pmatrix}
\frac{\partial^2 \lambda_n}{\partial \lambda_i \partial \lambda_j}
\end{pmatrix}_{i, j \in \{1, \cdots, n-1\}}$, where
\begin{align*}
\lambda_n = \frac{\sum_{k=0}^{n-1}c_k \sigma_k(\lambda_{;n})}{\lambda_1 \cdots \lambda_{n-1} - \sum_{k = 0}^{n-1} c_k \sigma_{k-1}(\lambda_{; n})},
\end{align*}is positive-definite on the level set $\{f = 0\}$.
\end{fconj}

When the degree equals three, Pingali \cite{pingali2019deformed} showed that the $2 \times 2$ Hessian matrix of $\lambda_3$ for the dHYM equation will be positive-definite. In \cite{lin2022}, the author showed that the Hessian matrix of the dHYM equation will be positive-definite on the level set when the degree equals four. Here, we prove the following result, which gives convincing evidence that Conjecture~\hyperlink{C:1.1}{1.1} is true.

\hypertarget{L:1.2}{\begin{flemma}}
Let $f(\lambda) \coloneqq  \lambda_1 \cdots \lambda_n - \sum_{k = 0}^{n-1} c_k \sigma_k(\lambda)$ be a general inverse $\sigma_k$ type multilinear polynomial. If the diagonal restriction $r_f(x)$ of $f$ is strictly right-Noetherian, then on the curve $\{\lambda_1 = \cdots = \lambda_{n-1}\}$ of the level set $\{ f  =  0 \}$ with $\lambda_1  > x_1$, the positive definiteness of the following $n-1 \times n-1$ Hessian matrix 
\begin{align*}
\Bigl(  \frac{\partial^2 \lambda_n}{\partial \lambda_i \partial \lambda_j}  \Bigr)_{i, j \in \{1, \cdots, n-1\}}
\end{align*}is equivalent to the monotonicity of log-concavity ratio of $r_f(x) = x^n - \sum_{k=0}^{n-1} c_k  \binom{n}{k} x^k$. Here, $x_1$ is the largest real root of $r'_f$.
\end{flemma}

The layout of this paper is as follows: in Section~\ref{sec:2}, we discuss some background materials. In Section~\ref{sec:2.1}, we introduce the class of right-Noetherian polynomials, which is related to the largest real roots of the derivatives of the polynomials. We show some special properties of the class of right-Noetherian polynomials. In Section~\ref{sec:2.2}, we consider some special semialgebraic sets in real algebraic geometry, which are defined by systems of inequalities of polynomials with real coefficients. To be more precise, we introduce the notion of $\Upsilon$-cones, which is an extension of the $C$-subsolution cone introduced by Székelyhidi \cite{szekelyhidi2018fully}. Roughly speaking, we consider the $C$-subsolution cone, the $C$-subsolution cone of the $C$-subsolution cone, the $C$-subsolution cone thereof, etc. We define the $\Upsilon$-stableness condition and show that the $\Upsilon$-stableness condition is equivalent to the right-Noetherianness condition for any general inverse $\sigma_k$ type multilinear polynomial. We also obtain a Positivstellensatz-type result over these semialgebraic sets and another Positivstellensatz-type result comparing two general inverse $\sigma_k$ type multilinear polynomials. In Section~\ref{sec:3}, we prove that if a level set of any general inverse $\sigma_k$ equation after translation is contained in the positive orthant, then this level is convex. In Section~\ref{sec:4}, we show that the convexities of some classical general inverse $\sigma_k$ type equations are immediate consequences of our main theorem. For example, we verify the convexity of any general inverse $\sigma_k$ equation with degree less than or equal to four, the convexity of the general inverse $\sigma_k$ equation with non-negative coefficients, and the convexity of the dHYM equation with supercritical phase.

\Acknow

\section{Preliminaries}
\label{sec:2}

\subsection{Definition and Properties of right-Noetherian Polynomials}
\label{sec:2.1}
In this subsection, we introduce the class of Noetherian polynomials, which will be used throughout this paper. The class of Noetherian polynomials has some special properties and will help us determine the convexity of the level set of any general inverse $\sigma_k$ type equation. We will see this in the later sections.

\hypertarget{D:2.1}{\begin{fdefi}[Noetherian polynomial]}
We say a degree $n$ real univariate polynomial $p(x)$ is right-Noetherian if for all $k \in \{0, \cdots, n-2 \}$, there exists a real root of $p^{(k)}$ which is greater than or equal to the largest real root of $p^{(k+1)}$; left-Noetherian if for all $k \in \{0, \cdots, n-2 \}$, there exists a real root of $p^{(k)}$ which is less than or equal to the smallest real root of $p^{(k+1)}$. Here $p^{(k)}$ is the $k$-th derivative of $p$. We say a right-Noetherian polynomial $p(x)$ is strictly right-Noetherian if the largest real root of $p(x)$ is strictly greater than the largest real root of $p'(x)$.
\end{fdefi}

\hypertarget{P:2.1}{\begin{fprop}}
Let $p(x)$ be a real univariate polynomial of degree $n$ which is right-Noetherian. Then for any $k \in \{0, \cdots, n-1\}$, there exists a unique (ignoring multiplicity) real root of $p^{(k)}(x)$ which is greater than or equal to the largest real root of $p^{(k+1)}(x)$. Moreover, this real root is the largest real root of $p^{(k)}(x)$. In particular, if we denote $x_k$ to be the largest real root of $p^{(k)}(x)$, then
\begin{align*}
x_0 \geq x_1 \geq \cdots \geq x_{n-1}.
\end{align*}
\end{fprop}

\begin{proof}
We prove this statement by mathematical induction on the degree $n$. When $n=1$, there is nothing to prove. When $n=2$, $p'(x)$ is a degree $1$ polynomial and the only root will be the midpoint of the roots of $p(x)$. By the definition of right-Noetherianness, there exists a real root of $p(x)$ which is greater than or equal to the largest real root of $p'(x)$. Thus $p(x)$ is real rooted and if we ignore the multiplicity, then there exists a unique real root of $p(x)$ which is greater than or equal to the largest real root of $p'(x)$. Moreover, this real root will be the largest real root of $p(x)$. Suppose the statement is true when $n = m-1$. When $n = m$, it suffices to check $p(x)$, the rest follows by mathematical induction. If there exists $x_0$ and $\tilde{x}_0$ with 
\begin{align*}
\tilde{x}_0 > x_0 \geq x_1,
\end{align*}where $x_1$ is the largest real root of $p'(x)$. For convenience, we assume that the polynomial $p(x)$ is monic, we write
\begin{align*}
p(x) = x^n + \sum_{k = 0}^{n-1} c_k x^k.
\end{align*}Then $p'(x) = nx^{n-1} + \sum_{k = 1}^{n-1} k c_k x^{k-1}$. Since $x_1$ is the largest real root of $p'(x)$, for any $x > x_1$, we have $p'(x) > 0$. By the fundamental theorem of calculus, we have
\begin{align*}
0 = p(\tilde{x}_0) > p(x_0) + \int_{x_0}^{\tilde{x}_0} p'(x) dx = 0  + \int_{x_0}^{\tilde{x}_0} p'(x) dx > 0,
\end{align*}which is a contradiction. This finishes the proof.
\end{proof}

We give a quick example of right-Noetherian polynomial, the right-Noetherianness condition is checkable by any computer using long division algorithm and Sturm's theorem.


\hypertarget{E:2.1}{\begin{fexmp}}
The following univariate polynomial $p(x) = x^5 - \sum_{k = 0}^{3} c_k \binom{5}{k}x^k$ with $c_3 = 19, c_2 = 65, c_1 = -2$, and $c_0 = -24$ is strictly right-Noetherian. This is checkable using any computer. By rounding off to the third decimal place, we have
\begin{align*}
x_0 \sim 15.250, \  x_1 \sim 11.673, \  x_2 \sim 8.066, \  x_3 \sim 4.359, \  x_4 =0.
\end{align*}Here, for $k \in \{0, \cdots, 4\}$, we denote by $x_k$ the largest real root of the $k$-th derivative $p^{(k)}(x)$. 
\end{fexmp}


\hypertarget{P:2.2}{\begin{fprop}}
Let $p(x)$ be a real univariate polynomial of degree $n$ which is real rooted, that is, all roots are real numbers, then $p(x)$ is both right-Noetherian and left-Noetherian.
\end{fprop}

\begin{proof}
This follows immediately by the Gauss--Lucas theorem. If $p(x)$ is real rooted, then the roots of $p'(x)$ will be contained in the convex hull of the set of roots of $p(x)$. So $p'(x)$ will also be real rooted, the rest follows directly by mathematical induction. This finishes the proof.
\end{proof}

\hypertarget{R:2.1}{\begin{frmk}}
A right-Noetherian polynomial might not be real rooted, a simple example will be $p(x) = x^3 -1$. Then we have $p'(x) = 3x^2$ and $p''(x) = 6x$. If we denote by $x_i$ the largest real root of $p^{(i)}(x)$, then $x_0 = 1$, $x_1 = 0$, and $x_2 = 0$. So $p(x)$ will be right-Noetherian due to $x_0 \geq x_1 \geq x_2$. But the roots of $p(x) = x^3 -1$ are: $1,  {(-1 + \sqrt{-3})}/{2}$, and $ {(-1 - \sqrt{-3})}/{2}$.
\end{frmk}

The log-concavity property of special univariate or multivariate polynomials were studied extensively by Brändén--Huh \cite{branden2020lorentzian}, Gurvits \cite{gurvits2009multivariate}, Anari--Gharan--Vinzant \cite{anari2018log1}, Anari--Liu--Gharan--Vinzant \cite{anari2019log, anari2018log3}, and Anari--Liu--Gharan--Vinzant--Vuong \cite{anari2021log}. For the class of right-Noetherian polynomials, we not only show that any right-Noetherian polynomial will be strongly log-concave after translation, but we also show that the ratio, which will be defined now, will be monotone.

\hypertarget{D:2.2}{\begin{fdefi}[Log-concavity ratio]}
Let $f \colon I \rightarrow \mathbb{R}$ be an analytic function, $I$ be an open interval in $\mathbb{R}$, and define $C_f \coloneqq \{ x \in I \colon f'(x) = 0\}$. For any point $x \in I$, we define the log-concavity ratio $\alpha_f(x)$ of $f(x)$ to be the following  
\begin{align*}
\label{eq:2.1}
\alpha_f(x) \coloneqq  \frac{f(x) \cdot f''(x)}{f'(x)^2} \tag{2.1} 
\end{align*}if $x \notin C_f$. If $x \in C_f$ is a limit point of $C_f$, then we define $\alpha_f(x)$ to be $0$, otherwise we define
\begin{align*}
\label{eq:2.2}
\alpha_f(x) \coloneqq \lim_{y \rightarrow x} \frac{f(y) \cdot f''(y)}{f'(y)^2}, \tag{2.2}
\end{align*}where we allow $\alpha_f(x) = \infty$ or $\alpha_f(x) = -\infty$.
\end{fdefi}

\hypertarget{R:2.2}{\begin{frmk}}
Let $f \colon I \rightarrow \mathbb{R}$ be an analytic function and $I$ be an open interval in $\mathbb{R}$, if $\alpha_f(x) \leq 1$ for all $x \in I$, then $f$ is logarithmically concave on $\{f > 0\}$. 
\end{frmk}

The following Proposition~\hyperlink{P:2.3}{2.3} shows that for any real univariate polynomial $p(x)$, $p(x)$ (or $-p(x)$) will eventually be logarithmically concave when $x$ is sufficiently large. 

\hypertarget{P:2.3}{\begin{fprop}}
Let $p$ be a real univariate polynomial of degree $n$, then
\begin{align*}
\lim_{x \rightarrow   \infty} \alpha_p(x) = 1- \frac{1}{n}.
\end{align*}In particular, there exists a $N > 0$ sufficiently large such that $p$ (or $-p$ depends on the sign of the leading coefficient) is logarithmically concave on $(N, \infty)$.
\end{fprop}

\begin{proof}For $x$ sufficiently large, if we write $p(x) = \sum_{k=0}^n c_k x^k$, then by equation (\ref{eq:2.1}), we have
\begin{align*}
\alpha_p(x) = \frac{p(x) \cdot p''(x)}{p'(x)^2} =   \frac{\sum_{k=0}^n c_k x^k \cdot \sum_{k=2}^n k(k-1)c_k x^{k-2}}{ \Bigl(  \sum_{k=1}^n k c_k x^{k-1}    \Bigr)^2}.
\end{align*}Since $p$ is a polynomial, by letting $x$ approach $\infty$, we can avoid critical points and get
\begin{align*}
\lim_{x \rightarrow \infty} \alpha_p (x) &= \lim_{x \rightarrow \infty}  \frac{\sum_{k=0}^n c_k x^k \cdot \sum_{k=2}^n k(k-1)c_k x^{k-2}}{ \Bigl(  \sum_{k=1}^n k c_k x^{k-1}    \Bigr)^2} =  \lim_{x \rightarrow \infty}  \frac{n(n-1) x^{2n-2} + O(x^{2n-3}) }{n^2 x^{2n - 2} + O(x^{2n-3}) } \\ 
&=  \lim_{x \rightarrow \infty}  \frac{n-1  + O(x^{-1}) }{n + O(x^{-1}) } = 1- \frac{1}{n}.
\end{align*}This finishes the proof.
\end{proof}

The derivative of the log-concavity ratio in Definition~\hyperlink{D:2.2}{2.2} of any real univariate polynomial will satisfy the following.

\hypertarget{L:2.1}{\begin{flemma}}
Let $p(x)$ be a real univariate polynomial of degree $n$. For $k \in \{1, \cdots, n-1\}$ and $x > x_k$, where $x_k$ is the largest real root of $p^{(k)}$, then $\alpha_{p^{(k-1)}}'(x) < 0$ when $2 > \alpha_{p^{(k)}} (x)$ and $\alpha_{p^{(k-1)}}(x) >    {1}/({2 - \alpha_{p^{(k)} } (x) })$. On the other hand, $\alpha_{p^{(k-1)}}'(x) > 0$ when $\alpha_{p^{(k)}} (x) \geq 2$ or $2 > \alpha_{p^{(k)}} (x)$ and $ {1}/({2 - \alpha_{p^{(k)} } (x) }) > \alpha_{p^{(k-1)}}(x)$.
\end{flemma}

\begin{proof}
By taking the derivative of $\alpha_{p^{(k-1)}}(x)$ with respect to $x$, we get
\begin{align*}
\label{eq:2.3}
\alpha_{p^{(k-1)}}'(x) &=  \frac{d}{dx} \frac{p^{(k-1)}(x) p^{(k+1)}(x)}{p^{(k)}(x)^2}  \tag{2.3} \\ 
&= \frac{p^{(k)}(x)^2 p^{(k+1)}(x) + p^{(k-1)}(x) p^{(k)}(x) p^{(k+2)}(x) - 2 p^{(k-1)}(x)p^{(k+1)}(x)^2 }{p^{(k)}(x)^3}.
\end{align*}

Then the numerator of equation (\ref{eq:2.3}) will give us
\begin{align*}
&\kern-1em p^{(k)}(x)^2 p^{(k+1)}(x) + p^{(k-1)}(x) p^{(k)}(x) p^{(k+2)}(x) - 2 p^{(k-1)}(x)p^{(k+1)}(x)^2 \\
&=  p^{(k)}(x)^2 p^{(k+1)}(x) + p^{(k-1)}(x) p^{(k+1)}(x)^2 \alpha_{p^{(k)}}(x)  - 2 p^{(k-1)}(x)p^{(k+1)}(x)^2 \\
&=  p^{(k)}(x)^2 p^{(k+1)}(x) + p^{(k-1)}(x) p^{(k+1)}(x)^2 \bigl( \alpha_{p^{(k)}}(x)  - 2 \bigr) \\
&=  p^{(k)}(x)^2 p^{(k+1)}(x) \Bigl( 1 + \bigl( \alpha_{p^{(k)}}(x)  - 2 \bigr) \alpha_{p^{(k-1)}}(x) \Bigr).
\end{align*}For convenience, we assume the leading coefficient is positive. When $2 > \alpha_{p^{(k)}}(x)$ and $\alpha_{p^{(k-1)}}(x) >    {1}/({2 - \alpha_{p^{(k)} } (x) })$, then since $x$ is greater than the largest real root of $p^{(k)}$, we get
\begin{align*}
\alpha_{p^{(k-1)}}'(x) = p^{(k+1)}(x) \frac{1 + \bigl( \alpha_{p^{(k)}}(x)  - 2 \bigr) \alpha_{p^{(k-1)}}(x)}{p^{(k)}(x) } < 0.
\end{align*}On the other hand, when $\alpha_{p^{(k)}} (x) \geq 2$ or $2 > \alpha_{p^{(k)}} (x)$ and $ {1}/({2 - \alpha_{p^{(k)} } (x) }) > \alpha_{p^{(k-1)}}(x)$, we get
\begin{align*}
\alpha_{p^{(k-1)}}'(x) = p^{(k+1)}(x) \frac{1 + \bigl( \alpha_{p^{(k)}}(x)  - 2 \bigr) \alpha_{p^{(k-1)}}(x)}{p^{(k)}(x) } > 0.
\end{align*}This finishes the proof. 
\end{proof}

Now, with all these preparations, we are able to prove the following important result for the class of right-Noetherian polynomials.

\hypertarget{T:2.1}{\begin{fthm}[Monotonicity of log-concavity ratio]}
Let $p(x)$ be a real univariate polynomial of degree $n$ which is right-Noetherian. Then the log-concavity ratio $\alpha_p(x)$ of $p(x)$ is monotonically increasing on $(x_1, \infty)$ with value from $-\infty$ to $1 - 1/n$ if $x_0 > x_1$ and on $[x_0, \infty)$ with value from $1 - 1/m$ to $1 - 1/n$ if $x_0 = x_1$. Here, $x_0$ is the largest real root of $p$, $x_1$ is the largest real root of $p'$, and $m$ is the multiplicity of $p$ at $x_0$. In particular, if $p$ is right-Noetherian, then $p$ is always logarithmically concave when $x > x_0$.
\end{fthm}

\begin{proof}
For convenience, we may assume that $p(x)$ is monic. By the definition of right-Noetherian, if we denote $x_i$ by the largest real root of $p^{(i)} (x)$, then by Proposition~\hyperlink{P:2.1}{2.1}, we have
\begin{align*}
x_0 \geq x_1 \geq x_2 \geq \cdots \geq x_{n-1}.
\end{align*}

We use mathematical induction on the degree of polynomial. For the base case, when $p$ is a degree $2$ polynomial, since $p$ has a real root $x_0$, we write
\begin{align*}
\label{eq:2.4}
p(x) = (x - x_0)(x- (2x_1 - x_0)) \tag{2.4}
\end{align*}with $x_0 \geq x_1$. If $x > x_1$, then by equation (\ref{eq:2.4}), we have
\begin{align*}
\alpha_p(x) = \frac{p(x) p''(x)}{p'(x)^2} =   \frac{ 2 (x-x_0)(x-(2x_1 - x_0)) }{4 (x-x_1)^2} =  \frac{1}{2} - \frac{(x_0-x_1)^2 }{2(x-x_1)^2}.
\end{align*}

So, for degree $2$ polynomial, $\alpha_p(x)$ is monotonically increasing from $0$ to $1/2$ from $x_0$ to $\infty$ if $x_0 = x_1$ and from $-\infty$ to $1/2$ from $x_1$ to $\infty$ if $x_0 > x_1$. Suppose the statement holds when the degree equals $n-1$. When the degree equals $n$, say the multiplicity of the largest real root $x_0$ equals $m \geq 1$. We have
\begin{align*}
\label{eq:2.5}
p(x) = (x-x_0)^m \cdot \sum_{k=m}^n \frac{p^{(k)}(x_0) }{k!} (x - x_0)^{k-m} = (x-x_0)^m \tilde{p}(x) \tag{2.5}
\end{align*}by using the Taylor series expansion of $p(x)$ at $x_0$ and set 
\begin{align*}
\tilde{p}(x) \coloneqq \sum_{k=m}^n \frac{p^{(k)}(x_0) }{k!} (x - x_0)^{k-m}.
\end{align*}So the first and the second derivative of (\ref{eq:2.5}) with respect to $x$ can be written as
\begin{align*}
\label{eq:2.6}
p'(x) &= (x-x_0)^{m-1} \bigl(  m \tilde{p}(x) + (x-x_0) \tilde{p}'(x)   \bigr); \tag{2.6} \\
\label{eq:2.7}
p''(x) &= (x-x_0)^{m-2} \bigl(  m(m-1) \tilde{p}(x) + 2m (x-x_0) \tilde{p}'(x) + (x-x_0)^2 \tilde{p}''(x)   \bigr). \tag{2.7}
\end{align*}If $m=1$, then similar to before, we get $\alpha_p (x_0) = 0$ and $\lim_{x \rightarrow x_1^+} \alpha_p(x) = -\infty$. If $m \geq 2$, then
\begin{align*}
\alpha_p(x_0) &=  \lim_{x \rightarrow x_0^+} \frac{p(x) p''(x)}{p'(x)^2}\\
& =  \lim_{x \rightarrow x_0^+} \frac{  (x-x_0)^m \tilde{p}(x) \cdot (x-x_0)^{m-2}  \bigl(  m(m-1) \tilde{p}(x) + 2m (x-x_0) \tilde{p}'(x) + (x-x_0)^2 \tilde{p}''(x)   \bigr) }{(x-x_0)^{2m-2}\bigl(m\tilde{p}(x) + (x-x_0) \tilde{p}'(x) \bigr)^2} \\
& =  \lim_{x \rightarrow x_0^+} \frac{    \tilde{p}(x)    \bigl(  m(m-1) \tilde{p}(x) + 2m (x-x_0) \tilde{p}'(x) + (x-x_0)^2 \tilde{p}''(x)   \bigr) }{  \bigl(m\tilde{p}(x) + (x-x_0) \tilde{p}'(x) \bigr)^2} = \frac{m-1}{m} = 1 - \frac{1}{m}.
\end{align*}

There are two cases to consider: $m =1$ or $m \geq 2$. For the case $m=1$, since $p'$ is again right-Noetherian, $\alpha_{p'}$ is monotonically increasing on $[x_1, \infty)$ with value from $1- 1/\tilde{m}$ to $1-1/(n-1)$, where $\tilde{m}$ is the multiplicity of $p'$ at $x_1$. When $x > x_1$, by Lemma~\hyperlink{L:2.1}{2.1}, $\alpha'_{p}(x) > 0$ when $2 > \alpha_{p'}(x)$ and $1/ (2- \alpha_{p'}(x)) > \alpha_p(x)$. We have $2 > 1-1/(n-1) > \alpha_{p'}(x)$ and 
\begin{align*}
\frac{1}{2 - \alpha_{p'}(x)} \geq \frac{1}{2-(1-1/\tilde{m})} = \frac{\tilde{m}}{\tilde{m}+1} > \lim_{x \rightarrow x_1^+}  \alpha_p(x) = -\infty.
\end{align*}If we consider the set $I \coloneqq \bigl \{  x \in (x_1, \infty) \colon 1/(2- \alpha_{p'}(x)) > \alpha_p(x)    \bigr\}$, then the set $I$ is not empty. If we can show that $I = (x_1, \infty)$, then we are done. $I$ will be open by the continuity of functions $\alpha_p$ and $\alpha_{p'}$. If $I \neq (x_1, \infty)$, then we can find a smallest $\tilde{x} \in (x_1, \infty)$ such that $1/(2- \alpha_{p'} ( {x})) = \alpha_p( {x})$. This is ensured because $p$ is a polynomial and
\begin{align*}
\frac{1}{2- \alpha_{p'} (x)} - \alpha_p(x) &= \frac{1}{2 - \alpha_{p'}(x)} \bigl(  1 - (2 - \alpha_{p'}(x)) \alpha_p(x)    \bigr) \\
&= \frac{1}{2 - \alpha_{p'}(x)} \Bigl(  1 - \Bigl(2  -  \frac{p'(x) p'''(x)}{p''(x)^2} \Bigr) \frac{p(x) p''(x)}{p'(x)^2}    \Bigr)  \\
&= \frac{1}{(2- \alpha_{p'}(x))p'(x)^2 p''(x)} \Bigl (  p'(x)^2 p''(x) + p(x) p'(x) p'''(x)   -2p(x) p''(x)^2      \Bigr).
\end{align*}The first term $\frac{1}{(2- \alpha_{p'}(x))p'(x)^2 p''(x)}$ cannot be zero if $x > x_1$ because $p$ is right-Noetherian and the second term $p'(x)^2 p''(x) + p(x)p'(x) p'''(x)   -2p(x) p''(x)^2 $ is just a polynomial so can only have finitely many zeros. Then by Lemma~\hyperlink{L:2.1}{2.1}, at the point $\tilde{x}$, we have
\begin{align*}
\label{eq:2.8}
\frac{d}{d x} \alpha_p (\tilde{x}) =  \frac{p''(x) \bigl ( 1 - (2 - \alpha_{p'}(x)) \alpha_p(x) \bigr  )}{p'(x)} = 0. \tag{2.8}
\end{align*}On the other hand, since $p'$ is right-Noetherian and by mathematical induction, we have 
\begin{align*}
\label{eq:2.9}
\frac{d}{dx} \Big|_{x = \tilde{x}} \frac{1}{2- \alpha_{p'}(x)} = \frac{  \alpha'_{p'}(\tilde{x})}{(2-\alpha_{p'}(\tilde{x}))^2} > 0. \tag{2.9}
\end{align*}We get a contradiction. Otherwise by standard calculus argument and equation (\ref{eq:2.8}), there exists a $\delta > 0$ sufficiently small such that if $|h| < \delta$, then we have
\begin{align*}
\label{eq:2.10}
-\frac{1}{3}  \frac{  \alpha'_{p'}(\tilde{x})}{(2-\alpha_{p'}(\tilde{x}))^2} <   \frac{\alpha_p(\tilde{x}) - \alpha_p(\tilde{x}-h)}{h}   < \frac{1}{3}  \frac{  \alpha'_{p'}(\tilde{x})}{(2-\alpha_{p'}(\tilde{x}))^2}. \tag{2.10}
\end{align*} Also, since $\bigl( \frac{1}{2 - \alpha_{p'}(x)} \bigr)' = \frac{\alpha'_{p'}(x)}{(2- \alpha_{p'}(x))^2}$, by inequality (\ref{eq:2.9}), for $\delta$ sufficiently small, we get
\begin{align*}
\label{eq:2.11}
\frac{2}{3}  \frac{  \alpha'_{p'}(\tilde{x})}{(2-\alpha_{p'}(\tilde{x}))^2} <  \frac{  \frac{1}{2 - \alpha_{p'}(\tilde{x})} -   \frac{1}{2 - \alpha_{p'}(\tilde{x}-h)}  }{h} < \frac{4}{3}  \frac{  \alpha'_{p'}(\tilde{x})}{(2-\alpha_{p'}(\tilde{x}))^2}. \tag{2.11}
\end{align*}Since $\tilde{x}$ is the smallest value such that $1/(2- \alpha_{p'} (x)) = \alpha_p(x)$ and $\lim_{x \rightarrow x_1^+} \alpha_p(x) = -\infty$, by the intermediate value theorem, we have $\tilde{x} - h \in I$ where $\delta > h > 0$. Hence, by inequalities (\ref{eq:2.10}) and (\ref{eq:2.11}), we obtain
\begin{align*}
\frac{1}{3}  \frac{  \alpha'_{p'}(\tilde{x})}{(2-\alpha_{p'}(\tilde{x}))^2} &>  \frac{\alpha_p(\tilde{x}) - \alpha_p(\tilde{x}-h)}{h}  =    \frac{  \frac{1}{2 - \alpha_{p'}(\tilde{x})} -   \alpha_p(\tilde{x}-h)   }{h}    \\
&>  \frac{  \frac{1}{2 - \alpha_{p'}(\tilde{x})} -   \frac{1}{2 - \alpha_{p'}(\tilde{x}-h)}  }{h} > \frac{2}{3}  \frac{  \alpha'_{p'}(\tilde{x})}{(2-\alpha_{p'}(\tilde{x}))^2}.
\end{align*}This is a contradiction because $\alpha'_{p'}(x) > 0$ by mathematical induction. So $I = (x_1, \infty)$. If the multiplicity of $x_0$ equals $1$, then $\alpha_p$ is increasing on $(x_1, \infty)$ with value from $-\infty$ to $1-1/n$.\smallskip

For the second case, if the multiplicity of $x_0$ is greater than or equal to $2$, then at $x_0$ we have $\alpha_p (x_0) = 1 -  {1}/{m}$ and $\alpha_{p'}(x_0) = 1 -  {1}/{(m-1)}$. This gives
\begin{align*}
\frac{1}{2- \alpha_{p'}(x_0)} = \frac{1}{2- (1-1/(m-1))}  = 1- \frac{1}{m} = \alpha_{p}(x_0). 
\end{align*}We need to do some local analysis near $x_0$. First, we have
\begin{align*}
\frac{1}{2- \alpha_{p'} (x)} - \alpha_p(x) &= \frac{p''(x)}{\bigl(2- \alpha_{p'}(x) \bigr)p'(x)^2 p''(x)^2} \Bigl (  p'(x)^2 p''(x) + p(x)p'(x) p'''(x)   -2p(x) p''(x)^2      \Bigr).
\end{align*}Same as before, we only need to consider the term $p'(x)^2 p''(x) + p(x)p'(x) p'''(x)   -2p(x) p''(x)^2$. By equations (\ref{eq:2.6}) and (\ref{eq:2.7}), we get
\begin{align*}
&\kern-1em p'(x)^2 p''(x) + p(x)p'(x) p'''(x)   -2p(x) p''(x)^2 \\ 
&= (x-x_0)^{3m-4} \Bigl(     \bigl(  m \tilde{p} + (x-x_0) \tilde{p}'   \bigr)^2 \bigl(  m(m-1) \tilde{p} + 2m (x-x_0) \tilde{p}' + (x-x_0)^2 \tilde{p}   \bigr) \\
&\kern8em + \tilde{p}   \bigl(  m \tilde{p} + (x-x_0) \tilde{p}'   \bigr) \bigl(  m(m-1)(m-2) \tilde{p} + 3m(m-1) (x-x_0) \tilde{p}' \\
&\kern23em + 3m(x-x_0)^2 \tilde{p}'' + (x-x_0)^3 \tilde{p}''' \bigr)\\
&\kern17em -2 \tilde{p} \bigl(  m(m-1) \tilde{p} + 2m (x-x_0) \tilde{p}' + (x-x_0)^2 \tilde{p}''   \bigr)^2 \Bigr) \\
&= 2m (x-x_0)^{3m-3}\tilde{p}^2 \tilde{p}' + 4m (x-x_0)^{3m-2} \bigl( \tilde{p}^2 \tilde{p}'' - \tilde{p} \tilde{p}'^2 \bigr) \\
&\kern2em + m (x- x_0)^{3m-1} \Bigl( 2 \tilde{p}'^3 + \tilde{p}^2 \tilde{p}''' - 3  \tilde{p} \tilde{p}' \tilde{p}''  \Bigr) + (x-x_0)^{3m} \Bigl(  \tilde{p}\tilde{p}'\tilde{p}''' + \tilde{p}'^2 \tilde{p}'' - 2 \tilde{p} \tilde{p}''^2  \Bigr).
\end{align*}When $x > x_0$ is sufficiently close to $x_0$, since $2m \tilde{p}(x)^2 \tilde{p}'(x) > 0$, we get
\begin{align*}
p'(x)^2 p''(x) + p(x)p'(x) p'''(x)   -2p(x) p''(x)^2 > 0.
\end{align*}Similarly, we define the set $I \coloneqq \bigl \{  x \in (x_0, \infty) \colon 1/(2- \alpha_{p'}(x)) > \alpha_p(x)    \bigr\}$ which is open and non-empty. Same as the previous argument, we get $I = (x_0, \infty)$, which implies that $\alpha_p$ is increasing on $[x_0, \infty)$ with value from $1-1/m$ to $1-1/n$. This finishes the proof.
\end{proof}

As an application, we immediately obtain that for a right-Noetherian polynomial $p(x)$, the roots of $p(x)$, the roots of $p'(x)$, and the root of $p''(x)$ will satisfy the following relation.

\hypertarget{P:2.4}{\begin{fprop}}
Let $p(x)$ be a real univariate polynomial of degree $n$ which is right-Noetherian. If we denote all the roots of $p(x)$ by $\alpha_1, \cdots, \alpha_n$, all the roots of $p'(x)$ by $\beta_1, \cdots, \beta_{n-1}$, and all the roots of $p''(x)$ by $\gamma_1, \cdots, \gamma_{n-2}$. If we write $x_k$ the largest real root of $p^{(k)}(x)$, then for $x > x_1$,
\begin{align*}
\frac{\prod_{i=1}^{n}(x-\alpha_i) \cdot \prod_{i=1}^{n-2} (x-\gamma_i)}{ \prod_{i=1}^{n-1} (x - \beta_i)^2}
\end{align*}is monotonically increasing to $1$ when $x$ approaches infinity. 
\end{fprop}

\hypertarget{P:2.5}{\begin{fprop}}
Let $p(x)$ be a real univariate polynomial of degree $n$ which is right-Noetherian, then 
\begin{align*}
1- \frac{2}{n} > \alpha_{p}(x) \alpha_{p'}(x)
\end{align*}for $x \geq x_1$, where $x_1$ is the largest real root of $p'(x)$. In particular, for $x \geq x_1$,
\begin{align*}
(n-2){p'(x) p''(x)}   \geq n  {p(x) p'''(x)}.
\end{align*}
\end{fprop}

\begin{proof}
By Theorem~\hyperlink{T:2.1}{2.1}, for $x >x_1$, we have
\begin{align*}
1- \frac{1}{n} > \alpha_p(x) > - \infty \quad \text{ and } \quad 1- \frac{1}{n-1} > \alpha_{p'}(x) \geq 0.
\end{align*}By multiplying them together, we always get the following upper bound:
\begin{align*}
1- \frac{2}{n} >  \alpha_{p}(x) \alpha_{p'}(x).
\end{align*}Moreover, for $x = x_1$, since $p$ is right-Noetherian, we have
\begin{align*}
\frac{p'(x_1) p''(x_1)}{n} - \frac{p(x_1) p'''(x_1)}{n-2} = - \frac{ p(x_1) p'''(x_1)}{n-2} \geq 0.
\end{align*}Also, for $x > x_1$, we have
\begin{align*}
\frac{p'(x) p''(x)}{n} - \frac{p(x) p'''(x)}{n-2} &= \frac{p'(x) p''(x)}{n-2} \Bigl(  \frac{n-2}{n} - \frac{p(x) p''(x) \cdot  p'(x) p'''(x)}{p'(x) ^2 \cdot p''(x)^2}  \Bigr) \\
&= \frac{p'(x) p''(x)}{n-2} \Bigl(  \frac{n-2}{n} - \alpha_{p}(x) \alpha_{p'}(x)  \Bigr) > 0.
\end{align*}This finishes the proof.
\end{proof}

For the class of right-Noetherian polynomials, we show that this class will be strongly log-concave after translation. We state the definitions here.

\hypertarget{D:2.3}{\begin{fdefi}[Strongly log-concave]}
Let $p(x_1, \cdots, x_n)$ be a multivariate polynomial, we say $p$ is strongly log-concave if any order partial derivative is either identically zero or log-concave on $\mathbb{R}^n_{>0}$.
\end{fdefi}

\hypertarget{L:2.2}{\begin{flemma}}
Let $p(x)$ be a real univariate polynomial of degree $n$ which is right-Noetherian, then $p$ is strongly log-concave after the translation $x \mapsto x-x_0$.
\end{flemma}

\begin{proof}
This follows directly by Theorem~\hyperlink{T:2.1}{2.1}.
\end{proof}

We prove the following deformation result for the class of right-Noetherian polynomials. The idea here is to deform the largest real root $x_0$ of $p$ to the next largest real root $x_m$ of $p^{(m)}$ for some $m \in \{1, \cdots, n-1\}$ and keep the $m$-th partial derivative same. Along this deformation, the right-Noetherianness is remained and this deformation will be monotone. In Figure~\ref{fig:2}, we plot the value of the log-concavity ratio along the deformation of $(x-5)^2(x^2 + 10x +51) = x^4 - 24 x^2 - 260 x + 1275$.

\begin{figure}
\centering
\begin{tikzpicture}[yscale=4]
	\def\theta{2*pi/3}
	\xdef\acc{0}
  	\draw[->] (-0.2,0) -- (8.5,0) node[right] {$x$};
  	\draw[->] (0,-0.04) -- (0,1) ;
	\foreach \t in {2.25, 2.5, 2.75, 3, 3.25, 3.5, 3.75, 4, 4.25, 4.5, 4.75, 5}
	{
		\pgfmathsetmacro\r{\t/10 }
		\pgfmathsetmacro\g{\t/10 }	
		\pgfmathsetmacro\b{0.6+\t/15}
  		\draw[color={rgb, 1:red, \r; green, \g; blue, \b},domain={\t }:8.4,samples=200]    	plot (\x, { 0.75 + 0.75*( -2*\x*(\t^3) - \t^4 -4*(\x^2)+16*\t*\x + 12*(\t^2) -48  )/((\x^2 +\t*\x + (\t^2 - 12))^(2) )    } )   ;
	}
	  \draw[color={rgb, 1:red, 0.2; green, 0.2; blue, 0.733},domain={2 }:8.4,samples=200]    	plot (\x,  {0.75 - 3/((\x+4)^2)} ) ;
   	\draw[red,dashed] (0,0.666) -- (2,0.666) node at (-0.4,0.65) {$2/3$};
  	\draw[red,dashed] (0,0.5) -- (5,0.5) node at (-0.4,0.5) {$1/2$};
  	\draw[red,dashed] (0,0.75) -- (8.5,0.75) node at (-0.4,0.77) {$3/4$};
\end{tikzpicture}
\caption{Value of $\alpha_{P}$ along the deformation of $(x-5)^2(x^2 + 10x +51)$}
\label{fig:2}
\end{figure}

\hypertarget{T:2.2}{\begin{fthm}}
Let $p(x)$ be a real univariate polynomial which is right-Noetherian with $x_0 \geq x_1 \geq \cdots \geq x_{n-1}$, where $x_i$ is the largest real root of $p^{(i)}(x)$. Consider the following deformation
\begin{align*}
P(x, y) \coloneqq p(x) - \sum_{k = 0}^{m-1} \frac{(x-y)^k}{k!} p^{(k)}(y) = \sum_{k = m}^n \frac{(x-y)^k}{k!} p^{(k)}(y),
\end{align*}where $y \in [x_m, x_0]$ and $m$ is the multiplicity of $x_0$. Then for $n-1 \geq m \geq 1$ and $x > y > x_m$, 
\begin{align*}
\frac{\partial}{\partial y} \alpha_{P}(x, y)  < 0.
\end{align*}Here, $\alpha_{P}(x, y)$ is defined by
\begin{align*}
\alpha_{P}(x, y) \coloneqq \frac{P(x, y) \frac{\partial^2}{\partial x^2} P(x, y)}{ \bigl(  \frac{\partial}{\partial x} P(x, y) \bigr)^2}.
\end{align*}
Notice that because $p(x)$ is right-Noetherian, $p^{(k)}(y) > 0$ for $y \in (x_m, x_0]$, $k \in \{m, \cdots, n\}$, and $p^{(m)}(x_m) = 0$. In particular, this deformation is a foliation that foliates the following set
\begin{align*}
\Bigl \{  (x, y) \in [x_0, \infty) \times \mathbb{R} \colon  \alpha_p(x)  \leq y \leq \alpha_{P(x, x_m)}(x)  \Bigr \}.
\end{align*}
\end{fthm}

\begin{proof}
The idea of $P(x, y)$ is to deform the largest real root $x_0$ of $p$ with multiplicity to the largest real root $x_m$ of $p^{(m)}$ and keep the $m$-th partial derivative concerning $x$ the same. Along this deformation, the right-Noetherianness is remained and the multiplicity of the largest real root will be remained $m$ but jump to at least $m+1$ when $y = x_m$. First, we have the following properties
\begin{align*}
\frac{\partial^l}{\partial x^l}P(x,y) = p^{(l)}(x) - \sum_{k=l}^{m-1} \frac{(x-y)^{k-l}}{(k-l)!} p^{(k)}(y);\quad  \frac{\partial^l}{\partial x^m} P(x,y) = p^{(m)}(x);  \quad \frac{\partial^l}{\partial x^l} P(y,y) \equiv 0
\end{align*}for $l \in \{0, \cdots, m-1\}$. With these, for $y \in \bigl[x_m, \  x_0 \bigr]$, if we treat $y$ as a fixed value, then
\begin{align*}
P(x,y) = (x-y)^m \cdot \Bigl(  \sum_{k = m}^{n} \frac{(x-y)^{k-m}}{k!}p^{(k)}(y)     \Bigr),
\end{align*}which implies that $y$ is the largest real root of $P(x, y)$ with multiplicity $m$ for $y \in \bigl[x_m, \  x_0 \bigr)$. If $y = x_m$, then since $p^{(m)}(x_m) = 0$, we get
\begin{align*}
P(x, x_m) &= (x - x_m)^{m+1} \cdot  \Bigl(  \sum_{k = m+1}^{n} \frac{(x-x_m)^{k-m}}{k!}p^{(k)}(x_m)     \Bigr).
\end{align*}The largest real root $x_m$ of $P(x, x_m)$ has multiplicity greater than or equal to $m+1$. For $l \in \{0, \cdots, m-1\}$, we have
\begin{align*}
\label{eq:2.12}
&\kern-1em \frac{\partial^{l+1} }{\partial y \partial x^l} P(x, y)  \tag{2.12}  \\
&= \frac{\partial}{\partial y} \Bigl (  p^{(l)}(x) - \sum_{k=l}^{m-1} \frac{(x-y)^{k-l}}{(k-l)!} p^{(k)}(y)    \Bigr) \\
&= \sum_{k=l+1}^{m-1} \frac{(x-y)^{k-l-1}}{(k-l-1)!} p^{(k)}(y)  - \sum_{k=l}^{m-1} \frac{(x-y)^{k-l}}{(k-l)!} p^{(k+1)}(y) = - \frac{(x-y)^{m-l-1}}{(m-l-1)!} p^{(m)}(y).
\end{align*}

Now, for convenience, we denote $\partial^l P /\partial x^l$ by $P^{(l)}(x, y)$. For $x > y$, by equation (\ref{eq:2.12}), we may compute 
\begingroup
\allowdisplaybreaks
\begin{align*}
\label{eq:2.13}
&\kern-1em\frac{\partial }{\partial y} \alpha_P (x, y) \tag{2.13} \\
&= \frac{\partial}{\partial y} \frac{P(x, y) P''(x, y)}{P'(x,y)^2} 
=   \frac{P' P'' \frac{\partial}{\partial y} P + P P' \frac{\partial}{\partial y} P'' - 2P P'' \frac{\partial}{\partial y}P' }{P'(x, y)^3} \\
&= \frac{ p^{(m)} (y) }{ P'(x, y)^3}  \Bigl( P' P'' \frac{(x-y)^{m-1}}{(m-1)!} + P   P' \frac{(x-y)^{m-3}}{(m-3)!}  - 2P P'' \frac{(x-y)^{m-2}}{(m-2)!} \Bigr).
\end{align*}
\endgroup
If we write $P(x, y) = (x-y)^m \cdot \tilde{P}(x, y)$ with $\tilde{P}(x, y) \coloneqq  \sum_{k = m}^{n}  {(x-y)^{k-m}}p^{(k)}(y)/{k!}$, then
\begin{align*}
\label{eq:2.14}
\frac{\partial}{\partial y} P(x, y) &= (x-y)^{m-1} \Bigl( m\tilde{P} + (x-y)\tilde{P}' \Bigr); \tag{2.14} \\
\label{eq:2.15}
\frac{\partial^2}{\partial y^2} P(x, y) &= (x-y)^{m-2} \Bigl( m(m-1)\tilde{P} + 2m(x-y)\tilde{P}' + (x-y)^2\tilde{P}'' \Bigr). \tag{2.15}
\end{align*}Hence, by combining equations (\ref{eq:2.13}), (\ref{eq:2.14}), and (\ref{eq:2.15}), we obtain
\begin{align*}
\label{eq:2.16}
&\kern-1em \frac{\partial}{\partial y} \alpha_P(x, y) \tag{2.16} \\
&= \frac{p^{(m)}(y)}{(m-1)!P'(x, y)^3} \cdot \Bigl(  -  {2(x-y)^{3m-3}}  \tilde{P}\tilde{P}'  +  {(x-y)^{3m-2}}  \bigl( (m-2)\tilde{P}\tilde{P}'' -2m\tilde{P}'^2 \bigr) \\
&\kern27.5em   -  {(x-y)^{3m-1}} \tilde{P}'\tilde{P}''  \Bigr).
\end{align*}When $m =1$ or $2$, every terms of (\ref{eq:2.16}) are negative, then we are done. When $m = n -1$, we have $\tilde{P}'' =  0$, the remaining terms of (\ref{eq:2.16}) are all negative. When $n-2 \geq m \geq 3$, we consider the following term:
\begin{align*}
\label{eq:2.17}
(m-2)\tilde{P}\tilde{P}'' - 2m \tilde{P}'^2 = \tilde{P}'^2 \bigl(   (m-2) \alpha_{\tilde{P}} - 2m \bigr). \tag{2.17}
\end{align*}When $x = y$, we get
\begin{align*}
\alpha_{\tilde{P}}(y, y) = \lim_{x \rightarrow y^+} \frac{\tilde{P}(x, y) \tilde{P}''(x, y)}{\tilde{P}'(x, y)^2} =  \frac{ \bigl( p^{(m)}(y)/m! \bigr) \cdot \bigl( 2 p^{(m+2)}(y)/(m+2)! \bigr)  }{\bigl( p^{(m+1)}(y)  /(m+1)!  \bigr)^2} = \frac{2(m+1)}{m+2} \alpha_{p^{(m)}}(y).
\end{align*}Similarly, we obtain
\begin{align*}
\label{eq:2.18}
\alpha_{\tilde{P}^{(k)}}(y, y) &=   \frac{ \bigl( k! \cdot p^{(m+k)}(y)/(m+k)! \bigr) \cdot \bigl( (k+2)! \cdot p^{(m+k+2)}(y)/(m+k+2)! \bigr)  }{\bigl( (k+1)! \cdot p^{(m+k+1)}(y)  /(m+k+1)!  \bigr)^2} \tag{2.18} \\
&=  \frac{(k+2)(m+k+1)}{(k+1)(m+k+2)} \alpha_{p^{(m+k)}}(y)
\end{align*}for $k \in \{0, \cdots, n-m-1\}$. By Theorem~\hyperlink{T:2.1}{2.1} and equation (\ref{eq:2.18}), we get
\begin{align*}
\label{eq:2.19}
 \frac{(k+2)(m+k+1)}{(k+1)(m+k+2)} \Bigl (1- \frac{1}{n-m-k} \Bigr) > \alpha_{\tilde{P}^{(k)}}(y, y) =  \frac{(k+2)(m+k+1)}{(k+1)(m+k+2)} \alpha_{p^{(m+k)}}(y) \geq 0 \tag{2.19}
\end{align*}for $k \in \{0, \cdots, n-m-1\}$. When $k = n-m-1$, $\tilde{P}^{(n-m-1)}$ is a polynomial of degree $1$, so we automatically have $\alpha_{\tilde{P}^{(n-m-1)}} \equiv 0$. We have $2 > \alpha_{\tilde{P}^{(n-m-1)}}$ and $1/(2- \alpha_{\tilde{P}^{(n-m-1)}}) = 1/2$. When $k = n-m-2$, by inequality (\ref{eq:2.19}), when $x = y$ we have
\begin{align*}
 \frac{(n-m)(n-1)}{n(n-m-1)} \Bigl (1- \frac{1}{2} \Bigr) = \frac{(n-1)(n-m)}{2n(n-m-1)}  > \alpha_{\tilde{P}^{(n-m-2)}}(y, y)   \geq 0.
\end{align*}On the other hand, we have  
\begin{align*}
 \frac{(n-1)(n-m)}{2n(n-m-1)} > \frac{1}{2-\alpha_{\tilde{P}^{(n-m-1)}}} = \frac{1}{2}.
\end{align*}By Lemma~\hyperlink{L:2.1}{2.1}, $\alpha_{\tilde{P}^{(n-m-2)}}(x ,y)$ will be decreasing if the value $\alpha_{\tilde{P}^{(n-m-2)}}(x ,y)$ exceeds $1/(2- \alpha_{\tilde{P}^{(n-m-1)}}(x, y)) = 1/2$. Combine these, we have an upper bound 
\begin{align*}
2 > \frac{(n-1)(n-m)}{2n(n-m-1)} = \max \Bigl\{   \frac{(n-1)(n-m)}{2n(n-m-1)}, \frac{1}{2}  \Bigr\}  > \alpha_{\tilde{P}^{(n-m-2)}}(x, y)
\end{align*}for $x \geq y$. In addition, for $x \geq y$, we get 
\begin{align*}
\label{eq:2.20}
\frac{1}{2 -  \frac{(n-1)(n-m)}{2n(n-m-1)}} \geq \frac{1}{2 - \alpha_{\tilde{P}^{(n-m-2)}}(x, y)}. \tag{2.20}
\end{align*}When $k = n-m-3$ and $x=y$, by inequality (\ref{eq:2.19}), we have 
\begin{align*}
 \frac{(k+2)(m+k+1)}{(k+1)(m+k+2)} \Bigl (1- \frac{1}{n-m-k} \Bigr) =  \frac{2(n-2)(n-m-1)}{3(n-1)(n-m-2)}   > \alpha_{\tilde{P}^{(k)}}(y, y)  \geq 0.
\end{align*}One can check that
\begin{align*}
 \frac{2(n-2)(n-m-1)}{3(n-1)(n-m-2)} > \frac{1}{2 -  \frac{(n-1)(n-m)}{2n(n-m-1)}}.
\end{align*}Similar to the previous argument, by Lemma~\hyperlink{L:2.1}{2.1}, we have the following upper bound
\begin{align*}
2 >  \frac{2(n-2)(n-m-1)}{3(n-1)(n-m-2)}  > \alpha_{\tilde{P}^{(n-m-3)}}(x, y)
\end{align*}for $x \geq y$. We claim that for any $k \in \{0, \cdots, n-m-2\}$, we always have
\begin{align*}
2 >  \frac{(k+2)(m+k+1)}{(k+1)(m+k+2)} \Bigl (1- \frac{1}{n-m-k} \Bigr) 
\end{align*}and
\begin{align*} 
\frac{(k+2)(k+m+1)}{(k+1)(k+m+2)} \Bigl (1 - \frac{1}{n - m - k} \Bigr ) \Bigl( 2-  \frac{(k+3)(k+m+2)}{(k+2)(k+m+3)} \Bigl (1 - \frac{1}{n-1-m-k} \Bigr )    \Bigr) > 1.
\end{align*}The proof of these claims should be straightforward. Same as the previous argument, we have the following upper bound for $x \geq y$,
\begin{align*}
\frac{(k+2)(k+m+1)}{(k+1)(k+m+2)} \Bigl (1 - \frac{1}{n-m-k} \Bigr ) &\geq \alpha_{\tilde{P}^{(k)}} (x, y)   \text{ for } k \in \{0, \cdots, n-m-2\}.
\end{align*}Thus, for $n-2 \geq m \geq 3$, by the above upper bound, the quantity (\ref{eq:2.17}) will satisfy
\begin{align*}
 (m-2) \alpha_{\tilde{P}} - 2m \leq \frac{2(m+1)(n-m-1)}{(m+2)(n-m)} (m-2) - 2m < 2(m-2) -2m < 0.
\end{align*}In conclusion, we have $\frac{\partial}{\partial y} \alpha_P (x, y) < 0$. This finishes the proof.
\end{proof}

\subsection{General Inverse \texorpdfstring{$\sigma_k$}{} Equations and \texorpdfstring{$\Upsilon$}{}-Cones}
\label{sec:2.2}
In this subsection, we introduce the notion of $\Upsilon$-cones, which is an extension of the $C$-subsolution cone introduced by Székelyhidi \cite{szekelyhidi2018fully} and Guan \cite{guan2014second}. The arguments in this subsection might be tedious because any set in this subsection might have more than one connected components, we need to specify which connected component we are considering. After all the arguments in this subsection, there will be no ambiguity, so we may assume the set is the connected component we are interested in. \smallskip

First, let us state some widely used notations, see Spruck \cite{spruck2005geometric} for more details. For an $n$-tuple numbers $\lambda = \{ \lambda_1, \cdots, \lambda_n\}$, for $k \in \{1, \cdots, n \}$, the $k$-th elementary symmetric polynomial $\sigma_k(\lambda)$ of $\lambda$ will be
\begin{align*}
\sigma_k(\lambda) \coloneqq \sum_{1 \leq i_1 < \cdots < i_k \leq n} \lambda_{i_1} \cdots \lambda_{i_k}.
\end{align*}We also define $\sigma_0(\lambda) \coloneqq 1$ for convenience. For $l \in \{1, \cdots, n \}$ and pairwise distinct indices $i_1, \cdots, i_l$, where $i_j \in \{1, \cdots, n\}$ for all $j \in \{1, \cdots, l\}$, we denote the set $\lambda - \{ \lambda_{i_1}, \cdots, \lambda_{i_l} \}$ by $\lambda_{; i_1, \cdots, i_l}$. Consider the following general inverse $\sigma_k$ type multilinear polynomial
\begin{align*}
\lambda_1 \cdots \lambda_n -   \sum_{k = 0}^{n-1} c_k \sigma_k(\lambda).
\end{align*}By doing the substitution $\mu =  \lambda - c_{n-1}$, we get a new general inverse $\sigma_k$ type multilinear polynomial:
\begin{align*}
\label{eq:2.21}
\lambda_1 \cdots \lambda_n - \sum_{k=0}^{n-1} c_k \sigma_k(\lambda) = 
\mu_1 \cdots \mu_n - \sum_{j=0}^{n-2} d_j \sigma_j(\mu),  \tag{2.21}
\end{align*}where $\mu_i = \lambda_i - c_{n-1}$. We have the following change of variables formula.

\hypertarget{L:2.3}{\begin{flemma}}
By doing the substitution $\mu_i = \lambda_i - c_{n-1}$ for all $i \in \{1, \cdots, n\}$, we have
\begin{align*}
\sigma_k(\lambda) = \sum_{j = 0}^{k}   c_{n-1}^{k-j} \binom{n-j}{k-j}  \sigma_j(\mu)
\end{align*}and the coefficients $d_j$ for $j \in \{1, \cdots, n-2\}$ will be
\begin{align*}
d_j =   \sum_{k = j}^{n-1} c_k  c_{n-1}^{k-j} \binom{n-j}{k-j} - c_{n-1}^{n-j}.
\end{align*}In addition, after substitution, the original general inverse $\sigma_k$ type multilinear polynomial becomes
\begin{align*}
\lambda_1 \cdots \lambda_n  - \sum_{k = 0}^{n-1} c_k \sigma_k(\lambda)  = \mu_1 \cdots \mu_n  - \sum_{j = 0}^{n-2} d_j \sigma_j(\mu)
\end{align*}and for all positive integer $l$ and $i_a \in \{ 1, \cdots, n   \}$ for all $a \in \{ 1, \cdots, l\}$, we have
\begin{align*}
\frac{\partial^l}{\partial \lambda_{i_1} \cdots \partial \lambda_{i_l}}  \Bigl ( \lambda_1 \cdots \lambda_n  - \sum_{k = 0}^{n-1} c_k \sigma_k(\lambda) \Bigr) = \frac{\partial^l}{\partial \mu_{i_1} \cdots \partial \mu_{i_l}}  \Bigl ( \mu_1 \cdots \mu_n  - \sum_{j = 0}^{n-2} d_j \sigma_j(\mu)  \Bigr).
\end{align*}
\end{flemma}

\begin{proof}First, by doing the substitution $\mu = \lambda - c_{n-1}$, we have $\lambda =  \mu + c_{n-1}$. Hence,
\begin{align*}
\sigma_k(\lambda) &= \sigma_k(\mu + c_{n-1}) = \sum_{j = 0}^{k}   c_{n-1}^{k-j} \binom{n-j}{k-j}  \sigma_j(\mu).
\end{align*}Thus, we get
\begingroup
\allowdisplaybreaks
\begin{align*}
\lambda_1 \cdots \lambda_n - \sum_{k=0}^{n-1} c_k \sigma_k(\lambda) &= 
\sum_{j = 0}^{n}   c_{n-1}^{n-j}    \sigma_j(\mu) - \sum_{k=0}^{n-1} c_k \sum_{j = 0}^{k}   c_{n-1}^{k-j} \binom{n-j}{k-j}  \sigma_j(\mu) \\
&= \mu_1 \cdots \mu_n + \sum_{j = 0}^{n-1}   c_{n-1}^{n-j}   \sigma_j(\mu) - \sum_{j=0}^{n-1}  \sum_{k = j}^{n-1} c_k  c_{n-1}^{k-j} \binom{n-j}{k-j}  \sigma_j(\mu) \\
&= \mu_1 \cdots \mu_n - \sum_{j = 0}^{n-2}  \Bigl(      \sum_{k = j}^{n-1} c_k  c_{n-1}^{k-j} \binom{n-j}{k-j}  - c_{n-1}^{n-j}    \Bigr) \sigma_j(\mu).
\end{align*}
\endgroup
So, for any $j \in \{1, \cdots, n-2\}$, we have
\begin{align*}
d_j = \sum_{k = j}^{n-1} c_k  c_{n-1}^{k-j} \binom{n-j}{k-j} - c_{n-1}^{n-j}.
\end{align*}
The rest follows by the change of variables formula. This finishes the proof.
\end{proof}

For convenience, by above Lemma~\hyperlink{L:2.3}{2.3}, we may assume that $c_{n-1} = 0$ by doing a translation. In most of the proofs in this subsection, we will do this substitution to simplify the proofs. We consider the following general inverse $\sigma_k$ type multilinear polynomial instead.
\begin{align*}
\label{eq:2.22}
f(\lambda) = f(\lambda_1, \cdots, \lambda_n) \coloneqq \lambda_1 \cdots \lambda_n    - \sum_{k = 0}^{n-2} c_k \sigma_k(\lambda). \tag{2.22}
\end{align*}

Now, we state the definition of $C$-subsolution here which was introduced by Székelyhidi \cite{szekelyhidi2018fully} and Guan \cite{guan2014second}. We will slightly adjust the settings in \cite{szekelyhidi2018fully} because in this work we mainly focus on the level set not a global section on the manifold $M$.

\hypertarget{D:2.4}{\begin{fdefi}[\texorpdfstring{$C$}{}-Subsolution. Székelyhidi \cite{szekelyhidi2018fully}, Guan \cite{guan2014second} and Trudinger \cite{trudinger1995dirichlet}]}Consider an equation $f(\lambda_1, \cdots, \lambda_n) = h$, where $f(\lambda_1, \cdots, \lambda_n)$ is a smooth symmetric function of variables $\{\lambda_1, \cdots, \lambda_n\}$. We assume that $f$ is defined in an open symmetric cone $\Gamma_f \subset \mathbb{R}^n$ satisfying $f > 0$, $\partial f/\partial \lambda_i > 0$ for all $i \in \{1, \cdots, n\}$ on $\Gamma_f$, and $\sup_{\partial \Gamma_f}  f <  h$. We say that $\mu = (\mu_1, \cdots, \mu_n) \in \mathbb{R}^n$ is a $C$-suboslution to the equation $f = h$ if the following set
\begin{align*}
\label{eq:2.23}
F^{h}(\mu)  \coloneqq \bigl \{  \lambda  \colon  {f}(\lambda) = h \  \text{ and } \  \lambda - \mu = (\lambda_1 - \mu_1, \cdots, \lambda_n - \mu_n) \in \Gamma_n  \bigr \} \tag{2.23}
\end{align*}is bounded. By collecting all the $C$-subsolutions, we call this collection the $C$-subsolution cone.
\end{fdefi}

\hypertarget{D:2.5}{\begin{fdefi}[Alternative definition of Definition~\hyperlink{D:2.4}{2.4}. Székelyhidi \cite{szekelyhidi2018fully} and Trudinger \cite{trudinger1995dirichlet}]} 
Suppose that $f$ is defined in an open symmetric cone $\Gamma_f \subset \mathbb{R}^n$ satisfying $f > 0$, $\partial f/\partial \lambda_i > 0$ for all $i \in \{1, \cdots, n\}$ on $\Gamma_f$, and $\sup_{\partial \Gamma_f}  f <  h$. Define 
\begin{align*}
\label{eq:2.24}
\Gamma_f^h \coloneqq \bigl \{ \lambda \in \Gamma_f \colon f(\lambda) > h  \bigr \}. \tag{2.24}
\end{align*}For $\mu \in \mathbb{R}^n$, set (\ref{eq:2.23}) $F^{h}(\mu)$ is bounded if and only if $\lim_{t \rightarrow \infty} f(\mu + t e_i) > h$ for all $i \in \{1, \cdots, n\}$, where $e_i$ is the $i$-th standard vector. We denote by ${\Gamma}_f^{n-1, h}$ the projection of $\Gamma_f^h$ onto $\mathbb{R}^{n-1}$ by dropping the last entry. Then for any $\mu' = (\mu_1, \cdots, \mu_{n-1} ) \in {\Gamma}_f^{n-1, h}$, define the function $f^{(n-1)}$ on ${\Gamma}_f^{n-1, h}$ by the following limit 
\begin{align*}
f^{(n-1)} (\mu_1, \cdots, \mu_{n-1}) &\coloneqq \lim_{\lambda_n \rightarrow \infty} f(\mu_1, \cdots, \mu_{n-1}, \lambda_n) \geq 0.
\end{align*}
First, the set $F^{h}(\mu)$ is bounded if and only if $f^{(n-1)} \bigl(\mu_{s(1)}, \cdots, \mu_{s(n-1)} \bigr) > h$ for every $s \in S_n$, where $S_n$ is the symmetric group. This is well-defined since $f$ is a symmetric function. We can show that for any $\mu \in \mathbb{R}^n$, $F^h(\mu)$ is bounded if and only if $\bigl(\mu_{s(1)}, \cdots, \mu_{s(n-1)}\bigr) \in \Gamma^{n-1, h}_f$ for every $s \in S_n$.
\end{fdefi}

Let $\Gamma_f^{n}$ be a connected component of $\{f(\lambda) > 0\}$, we are interested in whether there exists a connected component of $\{f(\lambda) > 0\}$ contained in the positive orthant $\Gamma_n$ after translation. Inspired by the work of Trudinger \cite{trudinger1995dirichlet} on the Dirichlet problem (over the reals) for equations of the eigenvalues of the Hessian, the results of Caffarelli--Nirenberg--Spruck \cite{caffarelli1985dirichlet}, and the results of Collins--Székelyhidi \cite{collins2017convergence}. In \cite{lin2022}, the author introduced the $\Upsilon$-cones to keep track of the information of the original equation as much as possible. We abstractly define the following sets.
\hypertarget{D:2.6}{\begin{fdefi}[\texorpdfstring{$\Upsilon$}{}-cones. Lin \cite{lin2022}]}
Let $f(\lambda) \coloneqq  \lambda_1 \cdots \lambda_n - \sum_{k = 0}^{n-1} c_k \sigma_k(\lambda)$ be a general inverse $\sigma_k$ type multilinear polynomial and $\Gamma_f^{n}$ be a connected component of $\{f(\lambda) > 0\}$, we denote by $\Gamma^{n-1}_f$ the projection of $\Gamma^{n}_f$ onto $\mathbb{R}^{n-1}$ by dropping the last entry. We define  
\begin{align*}
\Upsilon_1 \coloneqq \bigl \{  \mu \in \mathbb{R}^n  \colon    \bigl(\mu_{s(1)}, \cdots, \mu_{s(n-1)}\bigr) \in \Gamma^{n-1}_f,\quad \forall s \in S_n \bigr \},  
\end{align*}where $S_n$ is the symmetric group. For $n-1 \geq k \geq 2$, we define the following $\Upsilon$-cones
\begin{align*}
\Upsilon_k  \coloneqq \bigl \{  \mu \in \mathbb{R}^n  \colon    \bigl(\mu_{s(1)}, \cdots, \mu_{s(n-k)}\bigr) \in \Gamma^{n-k}_f,\quad \forall s \in S_n  \bigr \},  
\end{align*}where we define $ \Gamma^{n-k}_f$ inductively by the projection of $ \Gamma^{n+1-k}_f$ onto $\mathbb{R}^{n-k}$ by dropping the last entry.
\end{fdefi}


\hypertarget{D:2.7}{\begin{fdefi}[\texorpdfstring{$\Upsilon$}{}-stableness]}
Let $f(\lambda) \coloneqq   \lambda_1 \cdots \lambda_n - \sum_{k = 0}^{n-1} c_k \sigma_k(\lambda)$ be a general inverse $\sigma_k$ type multilinear polynomial and $\Gamma_f^{n}$ be a connected component of $\{f(\lambda) > 0\}$. We say that this connected component $\Gamma^n_f$ of $f(\lambda)$ is $\Upsilon$-stable if
\begin{align*}
\Gamma^n_f \subseteq  q +  \Gamma_n \text{ for some } q \in \mathbb{R}^n, 
\end{align*}where $\Gamma_n$ is the positive orthant of $\mathbb{R}^n$. We say that this connected component $\Gamma^n_f$ is strictly $\Upsilon$-stable if it is $\Upsilon$-stable and the boundary $\partial \Gamma^n_f$ is contained in the $\Upsilon_1$-cone. 
\end{fdefi}

\hypertarget{R:2.3}{\begin{frmk}}
Let $f(\lambda) \coloneqq   \lambda_1 \cdots \lambda_n - \sum_{k = 0}^{n-1} c_k \sigma_k(\lambda)$ be a general inverse $\sigma_k$ type multilinear polynomial and $\Gamma_f^{n}$ be a connected component of $\{f(\lambda) > 0\}$. We will show that if $\Gamma_f^{n}$ is strictly $\Upsilon$-stable, then the symmetric cone $\Gamma_f$ in Definition~\hyperlink{D:2.4}{2.4} will always be contained in the $\Upsilon_1$-cone. Normally, we consider the largest possible $\Gamma_f$, which is in fact the $\Upsilon_1$-cone. So the $\Upsilon_1$-cone is the same as the $C$-subsolution cone introduced by Székelyhidi \cite{szekelyhidi2018fully}.
\end{frmk}

\hypertarget{L:2.4}{\begin{flemma}}
Let $f(\lambda) \coloneqq  \lambda_1 \cdots \lambda_n - \sum_{k = 0}^{n-1} c_k \sigma_k(\lambda)$ be a general inverse $\sigma_k$ type multilinear polynomial and $\Gamma_f^{n}$ be a connected component of $\{f(\lambda) > 0\}$. If $\Gamma^n_f$ is $\Upsilon$-stable, say $\Gamma^n_f \subseteq  q +  \Gamma_n$ with $q = (q_1, \cdots, q_n)$, then 
\begin{align*}
\label{eq:2.25}
\Gamma^n_f \subseteq  \Upsilon_1 \subseteq \Upsilon_2  \subseteq \cdots \subseteq \Upsilon_{n-1} = (c_{n-1}, \cdots, c_{n-1}) + \Gamma_n, \tag{2.25}
\end{align*}$c_{n-1} \geq q_i$ for all $i \in \{1, \cdots, n\}$. For any $l \in \{1, \cdots, n-1\}$, we have $\Upsilon_l$ is open, connected, and
\begin{align*}
\label{eq:2.26}
\Upsilon_l  =      \bigcap_{1 \leq i_1 < \cdots < i_l \leq n} \Gamma^n_{ f_{i_1 \cdots i_l} } =  \bigcap_{1 \leq i_1 < \cdots < i_l \leq n} \Bigl \{   \sigma_{n-l} (\lambda_{; i_1, \cdots, i_l})   - \sum_{k=l}^{n-1} c_k \sigma_{k - l}(\lambda_{;i_1, \cdots, i_l})  > 0  \Bigr \}. \tag{2.26}
\end{align*}Here, we write $f_{i_1 \cdots i_l}$ as the $l$-th partial derivative $\frac{\partial^l f}{\partial \lambda_{i_1} \cdots \partial \lambda_{i_l}}$.
\end{flemma}

\hypertarget{R:2.4}{\begin{frmk}}
Notice that for the above Lemma~\hyperlink{L:2.4}{2.4}, we need to specify each connected component inductively on the subindices to avoid ambiguity. For example, when $l = n-2$, the set  
\begin{align*}
&\kern-1em \bigcap_{1 \leq i_1 < \cdots < i_{n-2} \leq n} \Bigl \{   \sigma_{2} (\lambda_{; i_1, \cdots, i_{n-2}})   - \sum_{k=n-2}^{n-1} c_k \sigma_{k - (n-2)}(\lambda_{;i_1, \cdots, i_{n-2}})  > 0  \Bigr \} \\
 &= \bigcap_{1 \leq i_1  < i_{2} \leq n} \bigl \{  \lambda_{i_1} \lambda_{i_2}  - c_{n-1} (\lambda_{i_1} + \lambda_{i_2}) -   c_{n-2}   > 0  \bigr \}
\end{align*}will have two connected components. We specify the one which is contained in the next $\Upsilon$-cone $\Upsilon_{n-1} = (c_{n-1}, \cdots, c_{n-1}) + \Gamma_n$. Similarly, we specify the connected component inductively til $\Upsilon_1$ by decreasing the subindices. But for notational convention, we abbreviate these expressions. 
\end{frmk}

The $\Upsilon$-stableness in fact gives us some constraints on the coefficients $\{c_k\}_{k = 0, \cdots, n-1}$. For example, if $\Gamma^n_f$ is $\Upsilon$-stable, then $c_{n-1}^2 + c_{n-2} \geq 0$. Otherwise, if $c_{n-1}^2 + c_{n-2} < 0$, then  
\begin{align*}
\bigcap_{1 \leq i_1  < i_{2} \leq n} \bigl \{  \lambda_{i_1} \lambda_{i_2}  - c_{n-1} (\lambda_{i_1} + \lambda_{i_2})   -   c_{n-2}   > 0  \bigr \}
\end{align*}is not contained in $\Upsilon_{n-1} = (c_{n-1}, \cdots, c_{n-1}) + \Gamma_n$, which violates the above Lemma~\hyperlink{L:2.4}{2.4}. Later on, we will prove that the $\Upsilon$-stableness condition is equivalent to the right-Noetherianness condition for the class of general inverse $\sigma_k$ type multilinear polynomials. So the constraints can be derived using the resultants and can be written explicitly when the degree is low, see \cite{lin2022}. We will show some examples in Section~\ref{sec:4} when the degree is less than or equal to four.

\begin{proof}[Proof of Lemma~\hyperlink{L:2.4}{2.4}]
We prove this by mathematical induction on the degree $n$. We also show that there exists a unique connected component that intersects and is entirely contained in the next $\Upsilon$-cone if $\Gamma^n_f$ is $\Upsilon$-stable. After translation, we may assume $c_{n-1} = 0$ for convenience and say $\Gamma^n_f$ is contained in $q + \Gamma_n$ with $q = (q_1, \cdots, q_n) \in \mathbb{R}^n$. When $n = 1$, $f = \lambda_1$, we immediately get
\begin{align*}
\Gamma^1_f = \{\lambda_1 > 0\} = \Gamma_1.
\end{align*}So $c_{1-1} = c_0 = 0 \geq q_1$, otherwise $\Gamma^1_f$ will not be contained in $q + \Gamma_1$. When $n = 2$, $f = \lambda_1 \lambda_2 - c_0$, several cases must be considered. If $c_0 < 0$, then $\{\lambda_1 \lambda_2 - c_0 > 0\}$ will not be contained in $q + \Gamma_2$ for any $q \in \mathbb{R}^2$, which contradicts the hypothesis that $\Gamma^2_f$ is $\Upsilon$-stable. If $c_0 = 0$, then $\{  \lambda_1 \lambda_2 > 0\}$ has two connected components. Since $\Gamma^2_f$ is $\Upsilon$-stable, $\Gamma^2_f$ is actually the positive orthant, which is one of the connected components of $\{  \lambda_1 \lambda_2 > 0\}$. So $c_{2-1} = c_1 = 0 \geq q_i$ for any $i \in \{1, 2\}$. By the definition of $\Upsilon$-cones, we get 
\begin{align*}
\Upsilon_1 \coloneq \bigl \{  \mu \in \mathbb{R}^2  \colon    \bigl(\mu_{s(1)} \bigr) \in \Gamma^{1}_f,\  \forall s \in S_2 \bigr \} = \bigl \{  \mu \in \mathbb{R}^2  \colon    \bigl(\mu_{s(1)} \bigr) \in \Gamma_{1},\  \forall s \in S_2 \bigr \}  = \Gamma_2.
\end{align*}Similarly, if $c_0 > 0$, then $\{  \lambda_1 \lambda_2  -c_0 > 0\}$ has two connected components. Since $\Gamma^2_f$ is $\Upsilon$-stable, $\Gamma^2_f$ is the one contained in the positive orthant $\Gamma_2$ and $c_{2-1} = c_1 = 0 \geq q_i$ for any $i \in \{1, 2\}$. By the definition of $\Upsilon$-cones, we get 
\begin{align*}
\Upsilon_1 \coloneq \bigl \{  \mu \in \mathbb{R}^2  \colon    \bigl(\mu_{s(1)} \bigr) \in \Gamma^{1}_f,\  \forall s \in S_2 \bigr \} = \bigl \{  \mu \in \mathbb{R}^2  \colon    \bigl(\mu_{s(1)} \bigr) \in \Gamma_{1},\  \forall s \in S_2 \bigr \}  = \Gamma_2.
\end{align*}In conclusion, when $n= 2$, $\Gamma^2_f$ is the unique connected component of $\{f(\lambda) > 0\}$ contained in
\begin{align*}
\Upsilon_1=  \Gamma_2 = \bigcap_{1 \leq i_1   \leq 2} \bigl \{   \sigma_{1} (\lambda_{; i_1})    > 0  \bigr \} = \{ \lambda_1 > 0 \} \cap  \{ \lambda_2 > 0 \}. 
\end{align*}Suppose that the statement is true when $n = m-1$. When $n = m$, we have
\begin{align*}
f(\lambda) = \lambda_1 \cdots \lambda_m - \sum_{k=0}^{m-2} c_k \sigma_k(\lambda).
\end{align*}Suppose that there exists $(\lambda_1, \cdots, \lambda_m) \in \Gamma_f^m$, such that
\begin{align*}
0 \geq \frac{\partial f}{\partial \lambda_i } (\lambda_1, \cdots, \lambda_m) = f_i(\lambda_1, \cdots, \lambda_m) = \sigma_{m-1}(\lambda_{;i}) - \sum_{k=1}^{m-2} c_k \sigma_{k-1} (\lambda_{; i})
\end{align*}for some $i \in \{1, \cdots, m\}$ and we denote $\partial f/\partial \lambda_i$ by $f_i$. By fixing other entries, for $\tilde{\lambda}_i \leq \lambda_i$, we get
\begin{align*}
&\kern-1em \tilde{\lambda}_i \bigl( \sigma_{m-1}(\lambda_{;i}) - \sum_{k=1}^{m-2} c_k \sigma_{k-1} (\lambda_{; i}) \bigr) - \sum_{k=0}^{m-2} c_k \sigma_k(\lambda_{;i})  \\
&\geq {\lambda}_i \bigl( \sigma_{m-1}(\lambda_{;i}) - \sum_{k=1}^{m-2} c_k \sigma_{k-1} (\lambda_{; i}) \bigr) - \sum_{k=0}^{m-2} c_k \sigma_k(\lambda_{;i}) = f(\lambda_1, \cdots, \lambda_m) > 0.
\end{align*}This implies that $(\lambda_1, \cdots, \lambda_{i-1}, \tilde{\lambda}_i, \lambda_{i+1}, \cdots, \lambda_m) \in \Gamma^m_f$ for all $\tilde{\lambda}_i \leq \lambda_i$ due to the assumption that $\Gamma^m_f$ is connected. By letting $\tilde{\lambda}_i < q_i$, we see that $\Gamma^m_f$ will not be contained in $q + \Gamma_m$, which leads to a contradiction. So for any $i \in \{1, \cdots,m\}$, $\Gamma^m_f$ will be contained in one of the connected component of  
\begin{align*}
\Bigl \{ f_i =  \sigma_{m-1}(\lambda_{;i}) - \sum_{k=1}^{m-2} c_k \sigma_{k-1} (\lambda_{; i}) > 0    \Bigr\}. 
\end{align*}For any $(\lambda_1, \cdots, \lambda_m) \in \Gamma^m_f$, since $\Gamma^m_f$ is contained in $\bigl \{  \sigma_{m-1}(\lambda_{;i}) - \sum_{k=1}^{m-2} c_k \sigma_{k-1} (\lambda_{; i}) > 0    \bigr\}$ for all $i \in \{1, \cdots, m\}$, we may construct a piecewise linear path connecting $(\lambda_1, \cdots, \lambda_m)$ to $(\lambda_{\text{max}}, \cdots, \lambda_{\text{max}})$, where $\lambda_{\text{max}} \coloneqq \max \{ \lambda_1, \cdots, \lambda_m\}$. Similarly, we can construct a piecewise linear path connecting $(\lambda_{s(1)}, \cdots, \lambda_{s(m)})$ to $(\lambda_{\text{max}}, \cdots, \lambda_{\text{max}})$ for all $s \in S_m$. For any $(\tilde{\lambda}_1, \cdots, \tilde{\lambda}_m)$ on these continuity paths, we always have
\begin{align*}
\tilde{\lambda}_1 \cdots \tilde{\lambda}_m - \sum_{k=0}^{m-2} c_k \sigma_k(\tilde{\lambda}) > 0.
\end{align*}This implies that $(\lambda_{s(1)}, \cdots, \lambda_{s(m)})$ are in the same connected component as $(\lambda_1, \cdots, \lambda_m)$ for all $s \in S_m$. Hence, $(\lambda_{s(1)}, \cdots, \lambda_{s(m)}) \in \Gamma^m_f$ for all $s \in S_m$. As a consequence, we similarly obtain that each connected component of $\bigl \{  \sigma_{m-1}(\lambda_{; i}) - \sum_{k=1}^{m-2} c_k \sigma_{k-1} (\lambda_{; i}) > 0    \bigr\}$ will be the same for all $i \in \{1, \cdots, m\}$ after change of variables. Next, we prove that this connected component of $\bigl \{  \sigma_{m-1}(\lambda_{;i}) - \sum_{k=1}^{m-2} c_k \sigma_{k-1} (\lambda_{; i}) > 0    \bigr\}$ is contained in $(q_1, \cdots, q_{i-1}, q_{i+1}, \cdots, q_m) + \Gamma_{m-1}$ by ignoring the Cartesian product $\mathbb{R}$ term. Without loss of generality, we only consider the case $i=1$. Let $(\tilde{\lambda}_1, \cdots, \tilde{\lambda}_m) \in \Gamma^m_f$ and consider the following section 
\begin{align*}
\lambda_1(\lambda_2, \cdots, \lambda_m) \coloneqq \frac{\sum_{k=0}^{m-2} c_k \sigma_k(\lambda_{; 1}) + f(\tilde{\lambda})}{\sigma_{m-1}(\lambda_{; 1}) - \sum_{k=1}^{m-2} c_k \sigma_{k-1} (\lambda_{; 1})},
\end{align*}where $f(\tilde{\lambda}) = \sigma_m(\tilde{\lambda}) - \sum_{k=0}^{m-2} c_k \sigma_k(\tilde{\lambda}) > 0$. Notice that this section $\lambda_1$ in $\Gamma^m_f$ is defined on this connected component of $\bigl\{\sigma_{m-1}(\lambda_{; 1}) - \sum_{k=1}^{m-2} c_k \sigma_{k-1} (\lambda_{; 1}) > 0 \bigr\}$, continuous, and 
\begin{align*}
\lambda_1(\tilde{\lambda}_2, \cdots, \tilde{\lambda}_m) = \frac{\sum_{k=0}^{m-2} c_k \sigma_k(\tilde{\lambda}_{; 1}) + f(\tilde{\lambda})}{\sigma_{m-1}(\tilde{\lambda}_{; 1}) - \sum_{k=1}^{m-2} c_k \sigma_{k-1} (\tilde{\lambda}_{; 1})} = \tilde{\lambda}_1.
\end{align*}Moreover, for any $(\lambda_2, \cdots, \lambda_m) \in \bigl \{  \sigma_{m-1}(\lambda_{;1}) - \sum_{k=1}^{m-2} c_k \sigma_{k-1} (\lambda_{; 1}) > 0    \bigr\}$ we get
\begin{align*}
\lambda_1(\lambda_2, \cdots, \lambda_m) \lambda_2 \cdots \lambda_m - \lambda_1(\lambda_2, \cdots, \lambda_m) \sum_{k=1}^{m-2} c_k \sigma_{k-1} c_k \sigma_{k-1}(\lambda_{;1})   - \sum_{k = 0}^{m-2} c_k \sigma_k(\lambda_{; 1}) = f(\tilde{\lambda}) > 0.
\end{align*}Thus, if this connected component of $\bigl \{  \sigma_{m-1}(\lambda_{;1}) - \sum_{k=1}^{m-2} c_k \sigma_{k-1} (\lambda_{; 1}) > 0    \bigr\}$ is not contained in $(q_2, \cdots,   q_m) + \Gamma_{m-1}$, then $\Gamma^m_f$ is not contained in $q + \Gamma_m$. This contradicts the hypothesis that $\Gamma^m_f$ is $\Upsilon$-stable. Similarly, for all $i \in \{1, \cdots, m\}$, this connected component of $\bigl \{  \sigma_{m-1}(\lambda_{;i}) - \sum_{k=1}^{m-2} c_k \sigma_{k-1} (\lambda_{; i}) > 0    \bigr\}$ is contained in $(q_1, \cdots, q_{i-1}, q_{i+1}, \cdots, q_m) + \Gamma_{m-1}$ by ignoring the Cartesian product $\mathbb{R}$ term. Hence, this connected component of $\bigl \{  \sigma_{m-1}(\lambda_{;i}) - \sum_{k=1}^{m-2} c_k \sigma_{k-1} (\lambda_{; i}) > 0    \bigr\}$ is $\Upsilon$-stable and will be the unique connected component contained in the $\Upsilon_1$-cone of $\Gamma^{m-1}_{f_i}$ by mathematical induction.\smallskip

 We now show that the $\Upsilon_1$-cone of $\Gamma^m_f$ is exactly
\begin{align*}
 \bigcap_{1 \leq i \leq m} \Bigl \{   \sigma_{m-1} (\lambda_{; i })   - \sum_{k=1}^{m-2} c_k \sigma_{k - 1}(\lambda_{;i})  > 0  \Bigr \}.
\end{align*}Then, the rest follows from mathematical induction. By the previous arguments, we know 
\begin{align*}
\Gamma^m_f \subseteq \bigcap_{1 \leq i \leq m} \Bigl \{   \sigma_{m-1} (\lambda_{; i })   - \sum_{k=1}^{m-2} c_k \sigma_{k - 1}(\lambda_{;i})  > 0  \Bigr \}.
\end{align*}For any $(\lambda_1, \cdots, \lambda_m) \in \Gamma^m_f$ and $i \in \{1, \cdots, m\}$, we have $\sigma_{m-1} (\lambda_{; i })   - \sum_{k=1}^{m-2} c_k \sigma_{k - 1}(\lambda_{;i})  > 0$. By the definition of $\Upsilon_1$, we obtain that $\Upsilon_1 \subseteq \bigcap_{1 \leq i \leq m} \bigl \{   \sigma_{m-1} (\lambda_{; i })   - \sum_{k=1}^{m-2} c_k \sigma_{k - 1}(\lambda_{;i})  > 0  \bigr \}$. On the other hand, for any $(\lambda_1, \cdots, \lambda_m) \in \bigcap_{1 \leq i \leq m} \bigl \{   \sigma_{m-1} (\lambda_{; i })   - \sum_{k=1}^{m-2} c_k \sigma_{k - 1}(\lambda_{;i})  > 0  \bigr \}$, we define the following continuous section
\begin{align*}
\lambda_i(\lambda_1, \cdots, \lambda_{i-1}, \lambda_{i+1}, \cdots, \lambda_m) \coloneqq \frac{\sum_{k=0}^{m-2} c_k \sigma_k(\lambda_{; i}) + f(\tilde{\lambda})}{\sigma_{m-1}(\lambda_{; i}) - \sum_{k=1}^{m-2} c_k \sigma_{k-1} (\lambda_{; i})},
\end{align*}where $\tilde{\lambda} \in \Gamma^m_f$. Similar to before, we can show that
\begin{align*}
\bigl(\lambda_1, \cdots, \lambda_{i-1}, \lambda_i(\lambda_1, \cdots, \lambda_{i-1}, \lambda_{i+1}, \cdots, \lambda_m), \lambda_{i+1}, \cdots, \lambda_m \bigr) \in \Gamma^m_f. 
\end{align*}This implies that $(\lambda_1, \cdots, \lambda_{i-1}, \lambda_{i+1}, \cdots, \lambda_m ) \in \Gamma^{m-1}_f$ for any $i \in \{1, \cdots, m\}$. Hence, by the definition of $\Upsilon_1$-cone, we get $(\lambda_1, \cdots, \lambda_m) \in \Upsilon_1$. In conclusion, we obtain
\begin{align*}
\Upsilon_1 = \bigcap_{1 \leq i \leq m} \Bigl \{   \sigma_{m-1} (\lambda_{; i })   - \sum_{k=1}^{m-2} c_k \sigma_{k - 1}(\lambda_{;i})  > 0  \Bigr \}.
\end{align*}By mathematical induction and (\ref{eq:2.26}), since $ \bigl \{   \sigma_{m-1} (\lambda_{; i })   - \sum_{k=1}^{m-2} c_k \sigma_{k - 1}(\lambda_{;i})  > 0  \bigr \}$ is $\Upsilon$-stable, 
\begin{align*}
\bigcap_{1 \leq i \leq m} \Bigl \{   \sigma_{m-1} (\lambda_{; i })   - \sum_{k=1}^{m-2} c_k \sigma_{k - 1}(\lambda_{;i})  > 0  \Bigr \} &\subseteq \bigcap_{1 \leq i < j \leq m} \Bigl \{   \sigma_{m-2} (\lambda_{; i, j})   - \sum_{k=2}^{m-2} c_k \sigma_{k - 2}(\lambda_{;i, j})  > 0  \Bigr \}  \\
&\subseteq  \cdots \subseteq \bigcap_{1 \leq i   \leq m} \Bigl \{   \lambda_i   > 0  \Bigr \} = \Gamma_m.
\end{align*}Last, we show that there exists a unique connected component of $\{ f(\lambda) > 0\}$ so that the intersection with $\Upsilon_1$ is not empty. Let $(\tilde{\lambda}_1, \cdots, \tilde{\lambda}_m) \in  \Upsilon_1 \backslash  {\Gamma}^m_f$ and 
\begin{align*}
\tilde{\lambda}_1 \cdots \tilde{\lambda}_m - \sum_{k=0}^{m-2} c_k \sigma_k(\tilde{\lambda}) > 0.
\end{align*}Similar to before, because $(\tilde{\lambda}_1, \cdots, \tilde{\lambda}_m) \in \Upsilon_1$, we can find a piecewise linear path connecting $(\tilde{\lambda}_1, \cdots, \tilde{\lambda}_m)$ to $(\lambda, \cdots, \lambda)$ for some $\lambda > 0$ large. In addition, along this path, any point will satisfy $f > 0$. We can connect any point in ${\Gamma}^m_f$ to $(\tilde{\lambda}_1, \cdots, \tilde{\lambda}_m)$, thus $(\tilde{\lambda}_1, \cdots, \tilde{\lambda}_m)$ will be in the same connected component, which is a contradiction. This finishes the proof.
\end{proof}

In the proof of Lemma~\hyperlink{L:2.4}{2.4}, we also obtain the following result, let us list this result here.

\hypertarget{L:2.5}{\begin{flemma}}
Let $f(\lambda) \coloneqq  \lambda_1 \cdots \lambda_n - \sum_{k = 0}^{n-1} c_k \sigma_k(\lambda)$ be a general inverse $\sigma_k$ type multilinear polynomial and $\Gamma_f^{n}$ be a connected component of $\{f(\lambda) > 0\}$. If $\Gamma^n_f$ is $\Upsilon$-stable, then for any $\lambda \in \Upsilon_l$,
\begin{align*}
\lambda + \overline{\Gamma}_n \subset \Upsilon_l
\end{align*}for $l \in \{0, 1, \cdots, n-1\}$. Here, we write $\Gamma^n_f = \Upsilon_0$ and $\overline{\Gamma}_n$ is the closure of $\Gamma_n$. As a consequence, $\Upsilon_l$ will be a connected open set for any $l \in \{0, 1, \cdots, n-1\}$ and as a set, 
\begin{align*}
\Upsilon_{l} = \bigcap_{1 \leq i_1 < \cdots < i_l \leq n} \Bigl \{   \sigma_{n-l} (\lambda_{; i_1, \cdots, i_l})   - \sum_{k=l}^{n-1} c_k \sigma_{k - l}(\lambda_{;i_1, \cdots, i_l})  > 0  \Bigr \} \cap \Upsilon_{l+1}.
\end{align*}In particular, for any $\lambda \in \bigl\{\lambda_1 \cdots \lambda_n - \sum_{k=0}^{n-2} c_k \sigma_k(\lambda) > 0 \bigr\} \cap \Upsilon_1$, we have $\lambda \in \Gamma^n_f$.
\end{flemma}

\hypertarget{R:2.5}{\begin{frmk}}
By Lemma~\hyperlink{L:2.4}{2.4}, the $\Upsilon$-cones are defined by systems of inequalities of polynomials, so they are semialgebraic sets in real algebraic geometry.
\end{frmk}

We consider the boundary of the $\Upsilon$-cones, they will have the following relations.

\hypertarget{L:2.6}{\begin{flemma}}
Let $f(\lambda) \coloneqq  \lambda_1 \cdots \lambda_n - \sum_{k = 0}^{n-1} c_k \sigma_k(\lambda)$ be a general inverse $\sigma_k$ type multilinear polynomial and $\Gamma_f^{n}$ be a connected component of $\{f(\lambda) > 0\}$. If $\Gamma^n_f$ is $\Upsilon$-stable and $\Upsilon_l \neq \bigl(c_{n-1}, \cdots, c_{n-1} \bigr) + \Gamma_n$ for some $l \in \{0, \cdots, n-2\}$, then either
\begin{align*}
\label{eq:2.27}
\partial \Upsilon_l  \cap  \partial \Upsilon_{l+1} = \emptyset \, \text{ or } \, \{(x_{l}, \cdots, x_{l})\}. \tag{2.27}
\end{align*}Here $x_{l}$ will be the largest real value satisfies both
\begin{align*}
x_{l}^{n-l}  - \sum_{k = l}^{n-1} c_k \mybinom[0.8]{n-l}{k-l}  x_{l}^{k-l} = 0  \quad \text{ and } \quad x_{l}^{n-l-1}  - \sum_{k = l+1}^{n-1} c_k \mybinom[0.8]{n-l-1}{k-l-1}  x_{l}^{k-l-1} = 0.
\end{align*}
Moreover, if $\Gamma^n_f$ is strictly $\Upsilon$-stable, then
\begin{align*}
\partial \Upsilon_0 \cap  \partial \Upsilon_{1} = \partial \Gamma^n_f \cap \partial \Upsilon_{1} = \emptyset.
\end{align*}
\end{flemma}

\begin{proof}We use mathematical induction to prove this, for convenience, we assume $c_{n-1} = 0$. First, when $n = 1$, there is nothing to prove. Second, when $n=2$, $f(\lambda) = \lambda_1 \lambda_2 - c_0$ with $c_0 \geq 0$. If $c_0 = 0$, then $\Upsilon_0 = \Gamma^2_f = \Gamma_2$ and $\Upsilon_1 = \Gamma_2$. If $c_0 > 0$, then $\Upsilon_0 = \Gamma^2_f = \{\lambda_1 \lambda_2 > c_0\}$ and $\Upsilon_1 = \Gamma_2$. The intersection $\partial \Upsilon_0  \cap  \partial \Upsilon_{1} = \emptyset$. Suppose the statement is true when $n = m-1$. Then, when $n = m$, we only need to prove the case that $\Upsilon_0 = \Gamma^m_f \neq \Gamma_m$, the rest follows directly by mathematical induction. If $\Upsilon_1 = \Gamma_m$, then $f = \lambda_1 \cdots \lambda_m - c_0$ with $c_0 > 0$ by Lemma~\hyperlink{L:2.4}{2.4}. For this case, $\partial \Upsilon_0  \cap  \partial \Upsilon_{1} = \emptyset$. We consider the case that $\Upsilon_1 \subsetneq \Gamma_m$, for any $(\lambda_1, \cdots, \lambda_m) \in \partial \Upsilon_0$, we have
\begin{align*}
0 = \lambda_1 \cdots \lambda_m   - \sum_{k=0}^{m-2} c_k \sigma_{k}(\lambda) = \lambda_1 \Bigl(   \lambda_2 \cdots \lambda_m   - \sum_{k=1}^{m-2} c_k \sigma_{k -1} (\lambda_{; 1})     \Bigr) - \sum_{k=0}^{m-2} c_k \sigma_k(\lambda_{;1}).
\end{align*}Due to Lemma~\hyperlink{L:2.4}{2.4}, $\Upsilon_0$ is contained in $\Upsilon_1$. This implies that
\begin{align*}
\sum_{k=0}^{m-2} c_k \sigma_k(\lambda_{;1}) = \lambda_1 \Bigl(   \lambda_2 \cdots \lambda_m   - \sum_{k=1}^{m-2} c_k \sigma_{k -1} (\lambda_{; 1})     \Bigr) \geq 0.
\end{align*}If $\partial \Upsilon_0  \cap  \partial \Upsilon_{1} \neq \emptyset$, we use the method of Lagrange multipliers to find the local extrema of $\sum_{k=0}^{m-2} c_k \sigma_k(\lambda_{;1})$ under the constraint $\lambda_2 \cdots \lambda_m   - \sum_{k=1}^{m-2} c_k \sigma_{k -1}(\lambda_{; 1})  = 0$. Let
\begin{align*}
\label{eq:2.28}
\mathcal{F}(\lambda_2, \cdots, \lambda_m, \mu) \coloneqq \sum_{k=0}^{m-2} c_k \sigma_k(\lambda_{;1}) - \mu \Bigl(   \lambda_2 \cdots \lambda_m   - \sum_{k=1}^{m-2} c_k \sigma_{k -1}(\lambda_{; 1})     \Bigr). \tag{2.28}
\end{align*}By taking the partial derivative of quantity (\ref{eq:2.28}) with respect to $\mu$ and $\lambda_i$, we have
\begin{align*}
\label{eq:2.29}
{\partial \mathcal{F}}/{\partial \mu} &= - \   \lambda_2 \cdots \lambda_m   + \sum_{k=1}^{m-2} c_k \sigma_{k -1}(\lambda_{; 1}); \tag{2.29} \\
\label{eq:2.30}
 {\partial \mathcal{F}}/{\partial \lambda_i} &= \sum_{k = 1}^{m-2} c_k \sigma_{k-1}(\lambda_{; 1, i}) - \mu \Bigl(   \lambda_2 \cdots \lambda_m /\lambda_i  - \sum_{k=2}^{m-2} c_k \sigma_{k -2}(\lambda_{; 1, i})     \Bigr), \tag{2.30}
\end{align*}for $i \in \{2, \cdots, m\}$. At those points with $\nabla \mathcal{F} = 0$, we subtract equation (\ref{eq:2.30}) by (\ref{eq:2.29}) and get
\begin{align*}
0 =  {\partial \mathcal{F}}/{\partial \lambda_i}  -  {\partial \mathcal{F}}/{\partial \mu} &= (\lambda_i - \mu ) \Bigl(   \lambda_2 \cdots \lambda_m /\lambda_i  - \sum_{k=2}^{m-2} c_k \sigma_{k -2}(\lambda_{; 1, i})     \Bigr)
\end{align*}for all $i \in \{2, \cdots, m\}$. By mathematical induction, $\partial \Upsilon_1 \cap \partial \Upsilon_2 = \emptyset$ or $\{(x_{1}, \cdots, x_{1})\}$. Here $x_{1}$ is the largest real value satisfies both 
\begin{align*}
x_{1}^{m-1} - \sum_{k=1}^{m-2} c_k \mybinom[0.8]{m-1}{k-1} x_{1}^{k-1} = 0 = x_{1}^{m-2} - \sum_{k=2}^{m-2} c_k \mybinom[0.8]{m-2}{k-2} x_{1}^{k-2}.
\end{align*}No matter which case, there exists only one local minimum $( {x_1}, \cdots,  {x_1})$, where
\begin{align*}
 x_1^{m-1} - \sum_{k=1}^{m-2} c_k \mybinom[0.8]{m-1}{k-1}  x_1^{k} = 0.
\end{align*}Since we assume $\partial \Upsilon_0  \cap  \partial \Upsilon_{1} \neq \emptyset$ and by above, there exists only one critical point. It is a global minimum, $\partial \Upsilon_0  \cap  \partial \Upsilon_{1} = \{ ( {x_1}, \cdots,  {x_1}) \}$, and $x_1$ also satisfies 
\begin{align*}
 x_1^{m} - \sum_{k=0}^{m-2} c_k \mybinom[0.8]{m}{k}  x_1^{k-1} = 0.
\end{align*}This finishes the proof.
\end{proof}

\hypertarget{P:2.6}{\begin{fprop}}
Let $f(\lambda) \coloneqq  \lambda_1 \cdots \lambda_n - \sum_{k = 0}^{n-1} c_k \sigma_k(\lambda)$ be a general inverse $\sigma_k$ type multilinear polynomial. For any $q \in \mathbb{R}^n$, there exists at most one connected component of $\{ f(\lambda) \neq 0\}$ which is contained in $q + \Gamma_n$ and this connected component will be a connected component of $\{ f(\lambda) > 0\}$.
\end{fprop}

\begin{proof}
We use mathematical induction to prove this, for convenience, we assume $c_{n-1} = 0$. First, when $n = 1$, there is nothing to prove. Second, when $n=2$, $f(\lambda) = \lambda_1 \lambda_2 - c_0$. If $c_0 < 0$, then no connected component of $\{ f(\lambda) \neq 0\}$ will be contained in $q + \Gamma_2$ for any $q \in \mathbb{R}^2$. If $c_0 \geq 0$, there exists one connected component of $\{ f(\lambda) \neq 0\}$ which will be contained in $\Gamma_2$. This connected component is a connected component of $\{ f(\lambda) > 0\}$. Suppose the statement is true when $n = m-1$. Then, when $n = m$, for any connected component of $\{ f(\lambda) \neq 0\}$ which is contained in $q + \Gamma_m$ for some $q \in \mathbb{R}^m$. If there exists a point $(\lambda_1, \cdots, \lambda_m)$ such that $f_i(\lambda_1, \cdots, \lambda_m) = 0$ for some $i \in \{1, \cdots, m\}$. Then for any $\tilde{\lambda}_i \leq \lambda_i$, we always have
\begin{align*}
f(\lambda_1, \cdots, \lambda_{i-1}, \tilde{\lambda}_i, \lambda_{i+1}, \cdots, \lambda_m) = f(\lambda_1, \cdots, \lambda_m).
\end{align*}This gives a contradiction. By induction and similar to previous proofs, we see that this connected component of $\{ f(\lambda) \neq 0\}$ will be contained in $\cap_{i \in \{1, \cdots, m\}} \{f_i > 0\}$. Here, by ignoring the Cartesian product $\mathbb{R}$ term, for convenience, we write $\{f_i > 0\}$ as the unique connected component of $\{f_i > 0\}$ which is contained in $(q_1, \cdots, q_{i-1}, q_{i+1}, \cdots, q_m) + \Gamma_{m-1}$. By the proof in Lemma~\hyperlink{L:2.4}{2.4}, this connected component of $\{ f(\lambda) \neq 0\}$ will be a connected component of $\{f > 0\}$ and is unique by Lemma~\hyperlink{L:2.5}{2.5}. This finishes the proof.
\end{proof}

\hypertarget{P:2.7}{\begin{fprop}}
Let $f(\lambda) \coloneqq  \lambda_1 \cdots \lambda_n - \sum_{k = 0}^{n-1} c_k \sigma_k(\lambda)$ be a general inverse $\sigma_k$ type multilinear polynomial and $\Gamma_f^{n}$ be a connected component of $\{f(\lambda) > 0\}$ which is $\Upsilon$-stable. Then for any $l \in \{0, \cdots, n-1\}$, the boundary $\partial \Upsilon_l$ of the $\Upsilon_l$-cone separates the ambient space $\mathbb{R}^n$ into two disjoint connected components.
\end{fprop}

\begin{proof}
For convenience, we assume that $c_{n-1} = 0$. By Lemma~\hyperlink{L:2.4}{2.4}, we have
\begin{align*}
\Gamma^n_f \subseteq  \Upsilon_1 \subseteq \Upsilon_2  \subseteq \cdots \subseteq \Upsilon_{n-1} =   \Gamma_n.
\end{align*}For any $l \in \{0, \cdots, n-1\}$, we consider the following two open sets $\Upsilon_l$ and $\mathbb{R}^m \backslash \overline{\Upsilon}_l$. By Lemma~\hyperlink{L:2.5}{2.5}, $\Upsilon_l$ is connected. Similar to the proof in Lemma~\hyperlink{L:2.4}{2.4}, we can also show that $\mathbb{R}^m \backslash \overline{\Upsilon}_l$ is connected. For any $(\lambda_1, \cdots, \lambda_n) \in \mathbb{R}^m \backslash \overline{\Upsilon}_l$, we have a straight path connecting $(\lambda_1, \cdots, \lambda_m)$ to $(\lambda_{\text{min}}, \lambda_2, \cdots, \lambda_{n})$, where $\lambda_{{\min}} \coloneqq \min \{ \lambda_1, \cdots, \lambda_m\}$. Suppose that there exists $(\tilde{\lambda}, \lambda_2, \cdots, \lambda_{n})$ on this path such that $(\tilde{\lambda}, \lambda_2, \cdots, \lambda_{n}) \in \Upsilon_l$. Then, by Lemma~\hyperlink{L:2.5}{2.5}, we get $(\lambda_1, \cdots, \lambda_n) \in \Upsilon_l$, which gives a contradiction. We can find a piecewise linear path connecting $(\lambda_1, \cdots, \lambda_m)$ to $(\lambda_{\min}, \cdots, \lambda_{\min})$ and for any point on this path, this point is in $\mathbb{R}^m \backslash \overline{\Upsilon}_l$. Hence, $\mathbb{R}^m \backslash \overline{\Upsilon}_l$ is open and connected. By standard point-set topology arguments, the boundary $\partial \Upsilon_l$ separates $\mathbb{R}^n$ into two disjoint connected components. This finishes the proof.
\end{proof}

\hypertarget{P:2.8}{\begin{fprop}}
Let $f(\lambda) \coloneqq  \lambda_1 \cdots \lambda_n - \sum_{k = 0}^{n-1} c_k \sigma_k(\lambda)$ be a general inverse $\sigma_k$ type multilinear polynomial and $\Gamma_f^{n}$ be a connected component of $\{f(\lambda) > 0\}$ which is (strictly) $\Upsilon$-stable. For any $l \in \{1, \cdots, n-1\}$, suppose $(\mu_{1}, \cdots, \mu_{l}) \in \Gamma^{l}_f$. Then for any $s \in S_n$, by fixing the $s(k)$-th entry $\lambda_{s(k)}$ equals $\mu_{k}$ for all $k \in \{1, \cdots, l\}$, and treat the rest as variables. $f/(\mu_1 \cdots \mu_l)$ is a degree $n-l$ general inverse $\sigma_k$ type multilinear polynomial and the cross section will be (strictly) $\Upsilon$-stable.
\end{fprop}

\begin{proof}
We use mathematical induction to prove this, for convenience, we assume $c_{n-1} = 0$. First, when $n = 1$, there is nothing to prove. Second, when $n=2$, $f(\lambda) = \lambda_1 \lambda_2 - c_0$ with $c_0 \geq 0$. We obtain $\Gamma^1_f = \Gamma_1 = \mathbb{R}_{>0}$ and for any $\mu_1 \in \Gamma^1_f$, we have $\mu_1 > 0$. For convenience, we fix $\lambda_1 = \mu_1$, then 
\begin{align*}
\frac{f}{\mu_1}(\mu_1, \lambda_2) = \lambda_2 - \frac{c_0}{\mu_1} 
\end{align*}is a degree $1$ general inverse $\sigma_k$ type multilinear polynomial. The cross section is
\begin{align*}
\bigl\{ \lambda_2 \colon \lambda_2 -  {c_0}/{\mu_1} > 0 \bigr\}
\end{align*}and will be $\Upsilon$-stable. Suppose the statement is true when $n = m-1$. Then, when $n = m$, it suffices to prove the statement by fixing a single entry, say $\mu_1 \in \Gamma^1_f = \Gamma_1 = \mathbb{R}_{>0}$ for convenience, the rest follows by induction. By symmetry, we fix $\lambda_1 = \mu_1$, then we get
\begin{align*}
\frac{f}{\mu_1}(\mu_1, \lambda_2, \cdots, \lambda_m) = \lambda_2 \cdots \lambda_m - \frac{c_{m-2}}{\mu_1} \sigma_{m-2}(\lambda_{; 1}) - \sum_{k = 0}^{m-3}  \Bigl( c_{k+1} + \frac{c_k}{\mu_1} \Bigr) \sigma_k(\lambda_{;1}).
\end{align*}By hypothesis, $\Gamma^m_f$ is $\Upsilon$-stable, so there exists $q = (q_1, \cdots, q_m) \in \mathbb{R}^m$ such that $\Gamma^m_f$ is contained in $q + \Gamma_m$. Thus, we can also verify that this connected component of
\begin{align*}
\{  {f}/{\mu_1} > 0 \} = \Bigl \{  \lambda_2 \cdots \lambda_m - \frac{c_{m-2}}{\mu_1} \sigma_{m-2}(\lambda_{; 1}) - \sum_{k = 0}^{m-3}  \Bigl( c_{k+1} + \frac{c_k}{\mu_1} \Bigr) \sigma_k(\lambda_{;1})  > 0 \Bigr \} 
\end{align*}is contained in $(q_2, \cdots, q_{m}) + \Gamma_{m-1}$ by ignoring the first entry $\mu_1$. So the cross section $\lambda_1 = \mu_1$ is $\Upsilon$-stable. For the strict $\Upsilon$-stableness, the proof is similar, so this finishes the proof.
\end{proof}

We prove the following Positivstellensatz-type result and a proposition to finish this subsection. In Section~\ref{sec:4}, we compute some examples when the degree is less than or equal to four using the resultants.

\hypertarget{T:2.3}{\begin{fthm}[Positivstellensatz]}
Let $f(\lambda) \coloneqq  \lambda_1 \cdots \lambda_n - \sum_{k = 0}^{n-1} c_k \sigma_k(\lambda)$ be a general inverse $\sigma_k$ type multilinear polynomial. There exists a connected component $\Gamma^n_f$ of $\{ f(\lambda) > 0\}$ which is $\Upsilon$-stable if and only if the diagonal restriction $r_f(x)$ of $f(\lambda)$, which is defined by the following  
\begin{align*}
\label{eq:2.31}
r_f(x) \coloneqq f(x, \cdots, x) = x^n - \sum_{k=0}^{n-1} c_k \mybinom[0.8]{n}{k} x^k, \tag{2.31}
\end{align*} is right-Noetherian. Moreover, $\Gamma^n_f$ is strictly $\Upsilon$-stable if and only if $r_f$ is strictly right-Noetherian.
\end{fthm}

\begin{proof}
For convenience, we assume $c_{n-1} = 0$. If $\Gamma^n_f$ is $\Upsilon$-stable, by Lemma~\hyperlink{L:2.4}{2.4}, then we have    
\begin{align*}
\Gamma^n_f \subseteq \Upsilon_1 \subseteq \Upsilon_2 \subseteq \cdots \subseteq \Upsilon_{n-1} = \Gamma_n.
\end{align*}For $\lambda > 0$ sufficiently large, by Lemma~\hyperlink{L:2.5}{2.5}, we have $(\lambda, \cdots, \lambda) \in \Gamma^n_f$. By decreasing the value of $\lambda$, since $\Gamma^n_f$ is contained in $\Gamma_n$, there exists a largest ${x}_0 \geq 0$ such that
\begin{align*}
f( x_0) = f(x_0, \cdots, x_0) = r_f(x_0)= x_0^n - \sum_{k = 0}^{n-2} c_k \mybinom[0.8]{n}{k} x_0^k = 0.
\end{align*}Otherwise we will get a contradiction. Similarly, there exists a largest $x_1 \geq 0$ such that 
\begin{align*}
f_n (x_1) =  f_n (x_1, \cdots, x_1) = \frac{1}{n} \frac{d}{dx} \Big|_{x = x_1} r_f(x) =  x_1^{n-1} - \sum_{k=1}^{n-2} c_k \mybinom[0.8]{n-1}{k-1} x_1^{k-1} = 0.
\end{align*}If $x_1 > x_0$, then $f(x_1) = f(x_1, \cdots, x_1) > 0$. So, we have $(x_1, \cdots, x_1) \in \Gamma^n_f \subseteq \Upsilon_1$, which implies that $f_n (x_1) > 0$. This contradicts $f_n (x_1) = 0$. If we inductively let $x_i$ be the largest real root of $r_f^{(i)}(x)$ for $i \in \{0, \cdots, n-1\}$, then similarly we obtain the following right-Noetherianness of $r_f(x)$:
\begin{align*}
x_0 \geq x_1 \geq \cdots \geq x_{n-1} = 0.
\end{align*}

On the other hand, for convenience, we assume that $c_{n-1} = 0$. If the diagonal restriction $r_f(x) = x^n - \sum_{k=0}^{n-2} c_k \binom{n}{k} x^k$ is right-Noetherian, then we use mathematical induction on the degree $n$. When $n = 1$, this is immediately true. When $n=2$, we have $r_f(x) = x^2 - c_0$, $c_0 \geq 0$ due to the assumption that $r_f$ is right-Noetherian. For $f(\lambda_1, \lambda_2) = \lambda_1 \lambda_2 - c_0$, there exists a connected component $\Gamma^2_f$ of $\{ f(\lambda) > 0\}$ which is $\Upsilon$-stable for $c_0 \geq 0$. Suppose the statement is true when $n = m-1$. When $n = m$, for any $i \in \{1, \cdots, m\}$ we get
\begin{align*}
r_{ f_i } (x) = x^{n-1} - \sum_{k=1}^{n-2} c_k \mybinom[0.8]{n-1}{k-1} x^{k-1} = \frac{1}{n} \frac{d}{dx} r_f(x).
\end{align*}This implies that $r_{ f_i}$ is still right-Noetherian. By mathematical induction, there exists a connected component $\Gamma^{m-1}_{f_i}$ of $\{f_i > 0\}$ which is $\Upsilon$-stable for all $i \in \{1, \cdots, m\}$. As a consequence, by Lemma~\hyperlink{L:2.4}{2.4}, $\Gamma^{m-1}_{f_i}$ is contained in $\Gamma_{m-1}$. If $\Gamma^{m-1}_{f_i} = \Gamma_{m-1}$ for some $i \in \{1, \cdots, m\}$, then by Lemma~\hyperlink{L:2.4}{2.4}, we can use mathematical induction and get
\begin{align*}
f(\lambda) = \lambda_1 \cdots \lambda_m - c_0.
\end{align*}We have $c_0 \geq 0$, which is guaranteed by the right-Noetherianness of $r_f$. Let $\Gamma^m_f$ be the connected component of $\{ f(\lambda) > 0\}$ which is contained in $\Gamma_m$, then $\Gamma^m_f$ will be $\Upsilon$-stable.\smallskip

Now, let $\Gamma^m_f$ be the open connected component of $\{ f(\lambda) > 0\}$ containing the ray $\{(x, \cdots, x) \in \mathbb{R}^m  \colon x > x_0 \}$. For the case $\Gamma^{m-1}_{f_i} \subsetneq \Gamma_{m-1}$, suppose that $\Gamma^m_f$ is not contained in $\Gamma^{m}_{f_i}$ for some $i \in \{1, \cdots, m\}$. Here $\Gamma^m_{f_i} = \Gamma^{m-1}_{f_i} \times \mathbb{R}$. By Proposition~\hyperlink{P:2.7}{2.7}, since $\Gamma^{m-1}_{f_i}$ is $\Upsilon$-stable, the boundary $\partial \Gamma^{m-1}_{f_i}$ separates $\mathbb{R}^{m-1}$ into two disjoint connected components. Thus, $\Gamma^{m}_{f_i}$ separates $\mathbb{R}^{m}$ into two disjoint connected components. Since any connected set in $\mathbb{R}^m$ is also path connected, there exists $(\tilde{\lambda}_1, \cdots, \tilde{\lambda}_m) \in \Gamma^m_f \cap \partial  \Gamma^{m}_{f_i}$. That is,
\begin{align*}
f(\tilde{\lambda}) = \tilde{\lambda}_1 \cdots \tilde{\lambda}_m - \sum_{k=0}^{m-2} c_k \sigma_k(\tilde{\lambda}) > 0  \quad \text{ and } \quad
\tilde{\lambda}_1 \cdots \tilde{\lambda}_m /\tilde{\lambda}_i - \sum_{k=1}^{m-2} c_k \sigma_{k-1} (\tilde{\lambda}_{; i}) = 0
\end{align*}for some $i \in \{1, \cdots, m\}$. Say $i = 1$ for convenience. At this point $(\tilde{\lambda}_1, \cdots, \tilde{\lambda}_m)$, we get
\begin{align*}
\label{eq:2.32}
\sum_{k = 0}^{m-2} c_k \sigma_k(\tilde{\lambda}_{;1}) = -f(\tilde{\lambda}) < 0. \tag{2.32}
\end{align*}

Similar to the proof in previous Lemma~\hyperlink{L:2.6}{2.6}, we use the method of Lagrange multipliers to find the local extrema of $\sum_{k = 0}^{m-2} c_k \sigma_k(\lambda_{;1})$ under the constraint $\lambda_2 \cdots \lambda_m - \sum_{k=1}^{m-2}c_k \sigma_{k-1}(\lambda_{;1}) = 0$. There exists only one local extremum at $(x_1, \cdots x_1)$, where $x_1$ is the largest real root of $r_{f_1}(x)$. Since $\Gamma^{m-1}_{f_i}$ is $\Upsilon$-stable, by Lemma~\hyperlink{L:2.6}{2.6}, we can treat $\lambda_{m}$ as a function in terms of $\lambda_2, \cdots, \lambda_{m-1}$ and smooth when $(\lambda_2, \cdots, \lambda_m) \neq (x_1, \cdots, x_1)$. By taking the derivative of the quantity $\sum_{k = 0}^{m-2} c_k \sigma_k( {\lambda}_{;1})$ and the equation $f_1 = 0$ with respect to $i \in \{2, \cdots, m-1\}$, we get
\begin{align*}
\label{eq:2.33}
\frac{\partial}{\partial \lambda_i} \sum_{k = 0}^{m-2} c_k \sigma_k( {\lambda}_{;1}) &=  \sum_{k = 1}^{m-2} c_k \sigma_{k-1}( {\lambda}_{;1, i}) + \frac{\partial \lambda_{m}}{\partial \lambda_i}  \sum_{k = 1}^{m-2} c_k \sigma_{k-1}( {\lambda}_{;1, m}) \tag{2.33} \\
&= \lambda_i \Bigl( \lambda_2 \cdots \lambda_m/\lambda_i - \sum_{k=2}^{m-2} c_k \sigma_{k-2}(\lambda_{; 1, i}) \Bigr) \\
&\kern2em+  \lambda_m \frac{\partial \lambda_m}{\partial \lambda_i} \Bigl( \lambda_2 \cdots \lambda_{m-1} - \sum_{k=2}^{m-2} c_k \sigma_{k-2}(\lambda_{; 1, m}) \Bigr) \\
&= \lambda_i f_{1i} + \lambda_m \frac{\partial \lambda_m}{\partial \lambda_i} f_{1m}.
\end{align*}Also, on $\lambda_2 \cdots \lambda_m - \sum_{k=1}^{m-2}c_k \sigma_{k-1}(\lambda_{;1}) = 0$, we have
\begin{align*}
\label{eq:2.34}
0 = \frac{\partial}{\partial \lambda_i} f_1 =  \frac{\partial}{\partial \lambda_i} \Bigl( \lambda_2 \cdots \lambda_m - \sum_{k=1}^{m-2}c_k \sigma_{k-1}(\lambda_{;1}) \Bigr) = f_{1i} + \frac{\partial \lambda_m}{\partial \lambda_i} f_{1m}. \tag{2.34}
\end{align*}By combining equations (\ref{eq:2.33}) and (\ref{eq:2.34}), we get
\begin{align*}
\label{eq:2.35}
\frac{\partial}{\partial \lambda_i} \sum_{k = 0}^{m-2} c_k \sigma_k( {\lambda}_{;1}) = \lambda_i f_{1i} - \lambda_m f_{1i} = (\lambda_i - \lambda_m)  f_{1i}. \tag{2.35}
\end{align*}By setting $\lambda_i \geq  \lambda_m$ for any $i \in \{2, \cdots, m-1\}$ and since $\Gamma^{m-1}_{f_i}$ is $\Upsilon$-stable, quantity (\ref{eq:2.35}) is always positive on the level set $\{f_1 = 0\}$. By (\ref{eq:2.35}), for any $(\lambda_2, \cdots, \lambda_{m-1}, \lambda_{m})$ on $\{f_1 = 0\}$, we have
\begin{align*}
\label{eq:2.36}
\sum_{k = 0}^{m-2} c_k \sigma_k( {\lambda}_{;1}) \geq   \sum_{k=0}^{m-2}c_k \mybinom[0.8]{m-1}{k} x_1^k. \tag{2.36}
\end{align*}

Since $x_1$ is a root of $r_{f_1}$, we obtain,
\begin{align*}
\label{eq:2.37}
  r_f(x_1) 
= x_1 \Bigl ( x_1^{m-1} -  \sum_{k=1}^{m-2} c_k \mybinom[0.8]{m-1}{k-1} x_1^{k-1}      \Bigr)  - \sum_{k=0}^{m-2}c_k \mybinom[0.8]{m-1}{k} x_1^k = - \sum_{k=0}^{m-2}c_k \mybinom[0.8]{m-1}{k} x_1^k.  \tag{2.37} 
\end{align*}By combining inequalities (\ref{eq:2.32}), (\ref{eq:2.36}), and (\ref{eq:2.37}), we get

\begin{align*}
0 > -f(\tilde{\lambda}) = \sum_{k = 0}^{m-2} c_k \sigma_k(\tilde{\lambda}_{;1}) \geq \sum_{k=0}^{m-2}c_k \mybinom[0.8]{m-1}{k} x_1^k = -r_f(x_1) \geq -r_f(x_0) =0.
\end{align*}Here $x_0$ is the largest real root of $r_f$. This gives a contradiction, in conclusion, $\Gamma^m_f$ is contained in $\Gamma^{m}_{f_i}$ for all $i \in \{1, \cdots, m\}$. This implies that $\Gamma^m_f$ is $\Upsilon$-stable. Similarly, we can show that $\Gamma^n_f$ is strictly $\Upsilon$-stable if and only if $r_f$ is strictly right-Noetherian. This finishes the proof.
\end{proof}

\hypertarget{P:2.9}{\begin{fprop}}
Let $f(\lambda) \coloneqq  \lambda_1 \cdots \lambda_n - \sum_{k = 0}^{n-1} c_k \sigma_k(\lambda)$ be a general inverse $\sigma_k$ type multilinear polynomial. If a level set of $\{f(\lambda) = 0\}$ is contained in $q + \Gamma_n$ for some $q \in \mathbb{R}^n$, then this level set is unique and there exists a unique connected component of $\{f(\lambda) > 0\}$, which is strictly $\Upsilon$-stable and the boundary will be this level set.
\end{fprop}

\begin{proof}
We use mathematical induction to prove this, for convenience, we assume $c_{n-1} = 0$. First when $n = 1$, there is nothing to prove. Second, when $n=2$, $f(\lambda) = \lambda_1 \lambda_2 - c_0$. If $c_0 \leq 0$, then no level set will be contained in $q + \Gamma_2$ for any $q \in \mathbb{R}^2$. Suppose the statement is true when $n = m-1$. Then, when $n = m$, if there exists a point $(\lambda_1, \cdots, \lambda_n)$ on this level set $\{f=0\}$ such that $f_i(\lambda_1, \cdots, \lambda_n) = 0$ for some $i \in \{1, \cdots, n\}$, then for $\tilde{\lambda}_i \leq \lambda_i$, we always have
\begin{align*}
f(\lambda_1, \cdots, \lambda_{i-1}, \tilde{\lambda}_i, \lambda_{i+1}, \cdots, \lambda_n) = f(\lambda_1, \cdots, \lambda_n) = 0.
\end{align*}By letting $\tilde{\lambda}_i$ approach $-\infty$, this gives a contradiction. By Proposition~\hyperlink{P:2.6}{2.6}, this level set of $\{f = 0\}$ will be contained in $\cap_{i \in \{1, \cdots, m\}} \{f_i > 0\}$. Hence, this level set will be the following graph 
\begin{align*}
\lambda_m = \frac{\sum_{k = 0}^{m-2} c_k \sigma_k(\lambda_{; m})}{\lambda_1 \cdots \lambda_{m-1} - \sum_{k=0}^{m-2} c_k \sigma_{k-1}(\lambda_{; m})}
\end{align*}over $\{f_m > 0\}$. We define $\Gamma^m_f$ by
\begin{align*}
\Gamma^m_f \coloneqq \Bigl \{ \lambda \colon  (\lambda_1, \cdots, \lambda_{m-1}) \in \{ f_m > 0\} \  \text{ and } \  \lambda_m > \frac{\sum_{k = 0}^{m-2} c_k \sigma_k(\lambda_{; m})}{  f_m} \Bigr\}.
\end{align*}We have $\Gamma^m_f$ is open and connected. If $\Gamma^m_f$ is not a connected component of $\{ f(\lambda) > 0\}$, say there exists $(\tilde{\lambda}_1, \cdots, \tilde{\lambda}_m) \notin \Gamma^m_f$ in this connected component. Then it suffices to check the case that $(\tilde{\lambda}_1, \cdots, \tilde{\lambda}_{m-1}) \notin \{f_m > 0\}$. Since connected set in $\mathbb{R}^m$ is also path connected and by induction, there exists $(\hat{\lambda}_1, \cdots, \hat{\lambda}_m)$ such that
\begin{align*}
f(\hat{\lambda}_1, \cdots, \hat{\lambda}_{m-1}, \hat{\lambda}_m) = - \sum_{k = 0}^{m-2} c_k \sigma_k (\hat{\lambda}_{;m}) >  0  \  \text{ and } \  f_m(\hat{\lambda}_1, \cdots, \hat{\lambda}_{m-1} ) = 0.
\end{align*}The rest follows by the proof in Theorem~\hyperlink{T:2.3}{2.3}, we use the method of Lagrange multipliers to get a contradiction. So, $\Gamma^m_f$ is an open connected component of $\{f(\lambda) > 0\}$ which is strictly $\Upsilon$-stable and the boundary $\partial \Gamma^m_f$ will be this level set. Similarly, a connected component of $\{ f(\lambda) > 0 \}$ satisfies these properties will be unique. This finishes the proof.
\end{proof}

%
%
%
%
%

\hypertarget{D:2.8}{\begin{fdefi}[\texorpdfstring{$\Upsilon$}{}-dominance]}
Let $f(\lambda) \coloneqq   \lambda_1 \cdots \lambda_n - \sum_{k = 0}^{n-1} c_k \sigma_k(\lambda)$ and $g(\lambda) \coloneqq   \lambda_1 \cdots \lambda_n - \sum_{k = 0}^{n-1} d_k \sigma_k(\lambda)$ be two $\Upsilon$-stable general inverse $\sigma_k$ type multilinear polynomials. For $i \in \{0, \cdots, n-1\}$, we write $x_i$ the largest real root of the diagonal restriction $r_f^{(i)}$ of $f$ and $y_i$ the largest real root of the diagonal restriction $r_g^{(i)}$ of $g$. If $y_i \geq x_i$ for all $i \in \{0, \cdots, n-1\}$, then we say $g \gtrdot f$. 
\end{fdefi}

\hypertarget{T:2.4}{\begin{fthm}[\texorpdfstring{$\Upsilon$}{}-dominance]}
Let $f(\lambda) \coloneqq   \lambda_1 \cdots \lambda_n - \sum_{k = 0}^{n-1} c_k \sigma_k(\lambda)$ and $g(\lambda) \coloneqq   \lambda_1 \cdots \lambda_n - \sum_{k = 0}^{n-1} d_k \sigma_k(\lambda)$ be two $\Upsilon$-stable general inverse $\sigma_k$ type multilinear polynomials. Then $g \gtrdot f$ if and only if $\Gamma^n_g \subset \Gamma^n_f$.
\end{fthm}

\begin{proof}
We use mathematical induction to prove this, for convenience, we assume $d_{n-1} = 0$. First when $n = 1$, the proof should be straightforward. Second, when $n=2$, $g(\lambda) = \lambda_1 \lambda_2  -  d_0$ and $f(\lambda) = \lambda_1 \lambda_2 - c_1(\lambda_1 + \lambda_2) -  c_0$. If $g \gtrdot f$, then we have
\begin{align*}
0 \geq c_1 \quad \text{ and } \quad  \sqrt{  d_0} \geq  c_1 + \sqrt{c_1^2 + c_0}.
\end{align*}For this case, if $\Gamma^2_g \not\subset \Gamma^2_f$, then there exists $(\tilde{\lambda}_1, \tilde{\lambda}_2)$ such that
\begin{align*}
\tilde{\lambda}_1 \tilde{\lambda}_2  -  d_0 > 0 \quad \text{ and } \quad \tilde{\lambda}_1 \tilde{\lambda}_2 - c_1( \tilde{\lambda}_1 + \tilde{\lambda}_2) -  c_0 \leq 0.
\end{align*}This is a contradiction, since
\begin{align*}
0 &< \tilde{\lambda}_1 \tilde{\lambda}_2  -  d_0  \leq c_1( \tilde{\lambda}_1 + \tilde{\lambda}_2)  + c_0 - d_0   \leq  2c_1 \sqrt{ \tilde{\lambda}_1 \tilde{\lambda}_2} + c_0 - d_0 < 2c_1 \sqrt{d_0} + c_0 - d_0 \\
&= - (\sqrt{d_0} -c_1)^2 + c_0 + c_1^2 \leq 0.
\end{align*}On the other hand, if $\Gamma^2_g \subset \Gamma^2_f$, then one can verify that $g \gtrdot f$. Suppose the statement is true when $n = m-1$. Then, when $n = m$, if $\Gamma^m_g \subset \Gamma^m_f$, by denoting the largest real root of $r_f^{(k)}$ by $x_k$ and the largest real root of $r_g^{(k)}$ by $y_k$, we immediately get $y_0 \geq x_0$. The rest follows from mathematical induction, thus $g \gtrdot f$. On the other hand, if $g \gtrdot f$, suppose $\Gamma^m_g \not\subset \Gamma^m_f$, there exists $(\tilde{\lambda}_1, \cdots, \tilde{\lambda}_m)$ such that
\begin{align*}
g(\tilde{\lambda}) = \tilde{\lambda}_1 \cdots \tilde{\lambda}_m - \sum_{k = 0}^{m-2} d_k \sigma_k (\tilde{\lambda}) > 0 \quad \text{ and } \quad f(\tilde{\lambda}) =   \tilde{\lambda}_1 \cdots \tilde{\lambda}_m - \sum_{k=0}^{m-1} c_k \sigma_k(\tilde{\lambda}) \leq 0.
\end{align*}
In addition, under the constraint $g = g(\tilde{\lambda}) > 0$, we consider the partial derivative of $f$ with respect to $\lambda_i$ for $i \in \{1, \cdots, m-1\}$. Under the constraint $g = g(\tilde{\lambda}) > 0$, we obtain
\begin{align*}
\label{eq:2.38}
\frac{\partial}{\partial \lambda_i} f = f_i + \frac{\partial \lambda_m}{\partial \lambda_i} f_m = f_i - \frac{g_i}{g_m}f_m = \frac{1}{g_m} \bigl(   f_i g_m  - g_i f_m  \bigr) = \frac{\lambda_m - \lambda_i}{g_m} \bigl(   g_m f_{im} - g_{im} f_m \bigr). \tag{2.38}
\end{align*}For the quantity $g_m f_{im} - f_{im} g_m$ in equation (\ref{eq:2.38}), we have
\begin{align*}
\frac{\partial}{\partial \lambda_i} \bigl(  g_m f_{im} - g_{im} f_m \bigr) = g_{im} f_{im} - g_{im} f_{im} = 0.
\end{align*}So the quantity $g_m f_{im} - g_{im} f_m$ is independent of the value of $\lambda_i$ and $\lambda_m$. By Theorem~\hyperlink{T:2.3}{2.3} and Proposition~\hyperlink{P:2.9}{2.9}, since $r_g$ is right-Noetherian, we have $g = g(\tilde{\lambda})$ is contained in the $\Upsilon_1$-cone $\Upsilon_1^g$ of $g - g(\tilde{\lambda})$. By fixing the values of $\tilde{\lambda}_1, \cdots, \tilde{\lambda}_{i-1}, \tilde{\lambda}_{i+1}, \cdots, \tilde{\lambda}_m$ and decreasing the value of the $i$-th entry, it will intersect with $\Upsilon_1^g$ in particular $\{ g_m = 0 \}$. At this intersection point, the quantity $g_m f_{i m} - g_{i m} f_m$ will be $g_m f_{im} - g_{im} f_m = - g_{im} f_m \leq 0$. The last inequality is due to mathematical induction, $g_m \gtrdot f_m$ if and only if $\Gamma^{m-1}_{g_m} \subset \Gamma^{m-1}_{f_m}$. Let $\lambda_m$ be the smallest value between $\{\lambda_1, \cdots, \lambda_m\}$, equation (\ref{eq:2.38}) will satisfy
\begin{align*}
\frac{\partial}{\partial \lambda_i} f  = \frac{\lambda_m - \lambda_i}{g_m} (g_m f_{im} - g_{im} f_m) \geq 0.
\end{align*}If we write $\tilde{y}_0$ the largest real root of $g - g(\tilde{\lambda})$, then we obtain
\begin{align*}
0 \geq f(\tilde{\lambda}_1, \cdots, \tilde{\lambda}_m) \geq f(\tilde{y}_0, \cdots, \tilde{y}_0) > f(y_0, \cdots, y_0) \geq f(x_0, \cdots, x_0) = 0.
\end{align*}This is a contradiction, hence we finish the proof.
\end{proof}

Here, we skip the proof of the following Lemma. By using mathematical induction, the proof should be straightforward.

\hypertarget{L:2.7}{\begin{flemma}}
Let $f(\lambda) \coloneqq   \lambda_1 \cdots \lambda_n - \sum_{k = 0}^{n-1} c_k \sigma_k(\lambda)$ and $g(\lambda) \coloneqq   \lambda_1 \cdots \lambda_n - \sum_{k = 0}^{n-1} d_k \sigma_k(\lambda)$ be two $\Upsilon$-stable general inverse $\sigma_k$ type multilinear polynomials. If $g \gtrdot f$, then for any $(\tilde{\lambda}_1, \cdots, \tilde{\lambda}_n) \in \{ g > 0 \}$, we have $f(\tilde{\lambda}_1, \cdots, \tilde{\lambda}_n) \geq g(\tilde{\lambda}_1, \cdots, \tilde{\lambda}_n)$. 
\end{flemma}

By considering the difference of two $\Upsilon$-stable general inverse $\sigma_k$ type multilinear polynomials with one $\Upsilon$-dominant another, we get the following Positivstellensatz-type result. We hope we can find a nicer statement in the future, comparing an $\Upsilon$-stable general inverse $\sigma_k$ multilinear polynomial with a $S_n$-invariant multilinear polynomial of smaller degree.

\hypertarget{L:2.8}{\begin{flemma}}
Let $f(\lambda) \coloneqq   \lambda_1 \cdots \lambda_n - \sum_{k = 0}^{n-1} c_k \sigma_k(\lambda)$ and $g(\lambda) \coloneqq   \lambda_1 \cdots \lambda_n - \sum_{k = 0}^{n-1} d_k \sigma_k(\lambda)$ be two $\Upsilon$-stable general inverse $\sigma_k$ type multilinear polynomials. If $g \gtrdot f$, then 
\begin{align*}
\Gamma^n_g \subset {\{ f - g = \sum_{k = 0}^{n-1} (d_k - c_k) \sigma_k(\lambda)  \geq 0\}}.
\end{align*}
\end{flemma}

\begin{proof}
The proof follows from Theorem~\hyperlink{T:2.4}{2.4} and Lemma~\hyperlink{L:2.7}{2.7}.
\end{proof}

Note that similar to before, we might need to specify a connected component of $\{ g - f   > 0\}$. A simple application of Lemma~\hyperlink{L:2.8}{2.8} will be the inequality of arithmetic and geometric means.



\subsection{New Materials to Prove the Convexity Theorem}
\label{sec:2.3}
We state an algebra result in \cite{lin2023s} that will be used when proving the convexity of the level set. We can collect all strictly $\Upsilon$-stable multilinear polynomials and get the following space. 

\hypertarget{D:2.9}{\begin{fdefi}[L.~\cite{lin2023s}] 
Consider the following set $\tilde{\mathscr{C}}_n \subset \mathbb{R}^{n-1}$, which is defined by
\begin{align*}
\tilde{\mathscr{C}}_n \coloneqq \Bigl \{ (c_{n-2},  \cdots, c_{1}, c_{0}) \in \mathbb{R}^{n-1} \colon  \lambda_1 \cdots \lambda_n - \sum_{k = 0}^{n-2} c_k \sigma_k(\lambda) \text{ is strictly } \Upsilon\text{-stable}   \Bigr \}.
\end{align*}Here, we let $\tilde{\mathscr{C}}_n$ be a topological space using the subspace topology induced from the standard Euclidean topology of the Euclidean space.
\end{fdefi}}

By Theorem~\hyperlink{T:2.3}{2.3}, strictly $\Upsilon$-stable general inverse $\sigma_k$ multilinear polynomials correspond to strictly right-Noetherian polynomials. Let $c \in \tilde{\mathscr{C}}_n$, if we denote $x_l(c)$ the largest real root of the $l$-th derivative of $x^n - \sum_{k=0}^{n-2} c_k \binom{n}{k} x^k$ for $l \in \{0, \cdots, n-1\}$, then we have
\begin{align*}
x_0(c) > x_1(c) \geq x_2(c) \geq \cdots \geq x_{n-2}(c) \geq x_{n-1}(c) = 0.
\end{align*}
So, any $c \in \tilde{\mathscr{C}}_n$ gives us an $(n-1)$-tuple $(x_{n-2}(c), \cdots, x_1(c), x_0(c))$ in the following polyhedron.

\hypertarget{D:2.10}{\begin{fdefi}[L.~\cite{lin2023s}]
Let $\tilde{\mathcal{X}}_n$ be the following polyhedron in $\mathbb{R}^{n-1}$:
\begin{align*}
\tilde{\mathcal{X}}_n \coloneqq \bigl \{  (x_{n-2}, \cdots, x_1, x_0) \colon    x_0 > x_1 \geq x_2 \geq \cdots \geq x_{n-2} \geq 0     \bigr\} \subsetneq \mathbb{R}^{n-1}.
\end{align*}Here, we let $\tilde{\mathcal{X}}_n$ be a topological space using the subspace topology induced from the standard Euclidean topology of the Euclidean space.
\end{fdefi}}

Naturally, we consider the following maps.
Let $\varphi \colon \tilde{\mathscr{C}}_n \rightarrow \mathbb{R}^{n-1}$ be a map defined by
\begin{align*}
\varphi(c) \coloneqq (x_{n-2}(c), \cdots, x_1(c), x_0(c)),
\end{align*}where $x_l(c)$ is the largest real root of the $l$-th derivative of $x^n - \sum_{k=0}^{n-2} c_k \binom{n}{k} x^k$ for $l \in \{0, \cdots, n-2\}$. 
Let $\psi \colon \tilde{\mathcal{X}}_n \rightarrow \mathbb{R}^{n-1}$ be a map defined by
\begin{align*}
\psi(x_{n-2}, \cdots, x_1, x_0) \coloneqq (d_{n-2}, \cdots, d_1, d_0),
\end{align*}where $d_l$ for $l \in \{0, 1, \cdots, n-2\}$ is defined recursively by
\begin{align*}
d_l \coloneqq x_l^{n-l} - \sum_{k = l +1}^{n-2} d_k \mybinom[0.8]{n-l}{k-l} x_l^{k-l}  
\end{align*}from $n-2$ back to $0$.

In \cite{lin2023s}, the author showed that the spaces $\tilde{\mathscr{C}}_n$ and $\tilde{\mathcal{X}}_n$ are homeomorphic.

\hypertarget{L:2.9}{\begin{flemma}[L.~\cite{lin2023s}]
The map $\varphi \colon \tilde{\mathscr{C}}_n   \rightarrow \tilde{\mathcal{X}}_n$ is a homeomorphism with inverse $\psi \colon \tilde{\mathcal{X}}_n \rightarrow \tilde{\mathscr{C}}_n$.
\end{flemma}}

\hypertarget{R:2.6}{\begin{frmk}
In fact, $\varphi$ can be extended to a map from $\overline{\tilde{\mathscr{C}}_n}$, the closure of $\tilde{\mathscr{C}}_n$, to $\overline{\tilde{\mathcal{X}}_n}$, the closure of $\tilde{\mathcal{X}}_n$. Moreover, similar to the proof of Lemma~\hyperlink{L:2.9}{2.9}, the extension is still a homeomorphism.
\end{frmk}}

The following are some new results that will also be used when proving the convexity of the level set.

\hypertarget{L:2.10}{\begin{flemma}
Let $c \in \overline{\tilde{\mathscr{C}}_{n}}$ and $\varphi  (c) = (x_{n-2}, \cdots, x_1, x_0)$, then for $l \in \{0, \cdots, n-2\}$, 
\begin{align*}
\sum_{k=l}^{n-2} c_k  \mybinom[0.8]{n-l-1}{k - l } x_{l+1}^{k-l} \geq 0
\end{align*}and equality holds if and only if $x_l = x_{l+1}$. In particular, if $c \in {\tilde{\mathscr{C}}_{n}}$, then 
$\sum_{k=0}^{n-2} c_k  \binom{n-1}{k  } x_{1}^{k} > 0$.
\end{flemma}}

\begin{proof}
For $l \in \{0, \cdots, n-2\}$, by plugging in $x_{l+1}$ to $x^{n-l} - \sum_{k = l}^{n-2} c_k \binom{n-l}{k-l} x^{k-l}$, we get
\begingroup
\allowdisplaybreaks
\begin{align*}
0 &\geq  x_{l+1}^{n-l} - \sum_{k = l}^{n-2} c_k \mybinom[0.8]{n-l}{k-l} x_{l+1}^{k-l} 
=   - \sum_{k=l}^{n-2} c_k  \mybinom[0.8]{n-l-1}{k - l } x_{l+1}^{k-l}  
\end{align*}
\endgroup
and equality holds if and only if $x_l = x_{l+1}$.
\end{proof}

\hypertarget{L:2.11}{\begin{flemma}
Let $c \in \overline{\tilde{\mathscr{C}}_n}$ and $\varphi(c) = (x_{n-2}, \cdots, x_1, x_0)$. For any $\mu > 0$, let $\tilde{x}_l(\mu)$ be the largest real root of the $l$-th derivative of $\mu(x^{n-1} - \sum_{k=1}^{n-2} c_k \binom{n-1}{k-1} x^{k-1} ) - \sum_{k=0}^{n-2} c_k \binom{n-1}{k} x^k$, then we have $\tilde{x}_l(\mu) \geq x_{l+1}$ for any $l \in \{0, \cdots, n-2\}$ and equality happens only if $x_l = x_{l+1}$. 
\end{flemma}}
\begin{proof}
Notice that by Proposition~\hyperlink{P:2.8}{2.8}, $\mu  (\lambda_2 \cdots \lambda_{m} - \sum_{k=1}^{m-2} c_k \sigma_{k-1}(\lambda_{; 1})  ) - \sum_{k=0}^{m-2} c_k \sigma_{k}(\lambda_{; 1})$ is $\Upsilon$-stable, so $\tilde{x}_0(\mu) \geq \cdots \geq \tilde{x}_{n-2}(\mu)$. For any $l \in \{0, \cdots, n-2\}$, by plugging $x_{l+1}$ to the polynomial $\mu (x^{n-l-1} - \sum_{k=l+1}^{n-2} c_k \binom{n-l-1}{k-l-1} x^{k-l-1} ) - \sum_{k=l}^{n-2} c_k \binom{n-l-1}{k-l} x^{k-l}$, we get
\begin{align*}
 \mu \Bigl(x_{l+1}^{n-l-1} - \sum_{k=l+1}^{n-2} c_k \mybinom[0.8]{n-l-1}{k-l-1} x_{l+1}^{k-l-1} \Bigr) - \sum_{k=l}^{n-2} c_k \mybinom[0.8]{n-l-1}{k-l} x_{l+1}^{k-l} 
&= - \sum_{k=l}^{n-2} c_k \mybinom[0.8]{n-l-1}{k-l} x_{l+1}^{k-l} \leq 0.
\end{align*}Here, the last inequality is due to Lemma~\hyperlink{L:2.10}{2.10} and equality holds if and only if $x_l = x_{l+1}$. Hence, $\tilde{x}_l(\mu) \geq x_{l+1}$ and equality happens only if $x_l = x_{l+1}$. This finishes the proof.
\end{proof}

\hypertarget{L:2.12}{\begin{flemma}
Let $c \in \overline{\tilde{\mathscr{C}}_n}$ and $\varphi(c) = (x_{n-2}, \cdots, x_1, x_0)$. For any $\mu >  \nu >  0$, we have $\tilde{x}_l(\nu) \geq \tilde{x}_l(\mu)$ for any $l \in \{0, \cdots, n-2\}$ and equality happens only if $x_l = x_{l+1}$. In particular, if $c \in {\tilde{\mathscr{C}}_n}$, then $\tilde{x}_0(\nu) > \tilde{x}_0(\mu)$.
\end{flemma}}

\begin{proof}
For any $l \in \{0, \cdots, n-2\}$, $\tilde{x}_l(\mu)$ satisfies
\begin{align*}
\mu \Bigl(\tilde{x}_l^{n-l-1}(\mu) - \sum_{k=l+1}^{n-2} c_k \mybinom[0.8]{n-l-1}{k-l-1} \tilde{x}_l^{k-l-1}(\mu)  \Bigr ) - \sum_{k=l}^{n-2} c_k \mybinom[0.8]{n-l-1}{k-l} \tilde{x}^{k-l}(\mu) = 0.
\end{align*}
By plugging in $\tilde{x}_l$ to $\nu (x^{n-l-1} - \sum_{k=l+1}^{n-2} c_k \binom{n-l-1}{k-l-1} x^{k-l-1} ) - \sum_{k=l}^{n-2} c_k \binom{n-l-1}{k-l} x^{k-l}$, we have
\begin{align*}
&\kern-2em \nu \Bigl(\tilde{x}_l^{n-l-1}(\mu) - \sum_{k=l+1}^{n-2} c_k \mybinom[0.8]{n-l-1}{k-l-1} \tilde{x}_l^{k-l-1}(\mu)  \Bigr ) - \sum_{k=l}^{n-2} c_k \mybinom[0.8]{n-l-1}{k-l} \tilde{x}^{k-l}(\mu) \\
&= (\nu - \mu) \Bigl(\tilde{x}_l^{n-l-1}(\mu) - \sum_{k=l+1}^{n-2} c_k \mybinom[0.8]{n-l-1}{k-l-1} \tilde{x}_l^{k-l-1}(\mu)  \Bigr ) \leq 0.
\end{align*}
The last inequality is due to Lemma~\hyperlink{L:2.11}{2.11}, we obtain $\tilde{x}_l(\mu) \geq x_{l+1}$ and $x_{l+1}$ is the largest real root of $x^{n-l-1} - \sum_{k = l+1}^{n-2} c_k \binom{n-l-1}{k-l-1}x^{k-l-1}$. Notice that $\tilde{x}_l(\mu) = x_{l+1}$ only if $x_l = x_{l+1}$. In particular, if $c \in \tilde{\mathscr{C}}_n$, then $x_0 > x_1$. By Lemma~\hyperlink{L:2.11}{2.11}, we have $\tilde{x}_0(\mu) > x_1$ which implies that $\tilde{x}_0^{n-1}(\mu) - \sum_{k=1}^{n-2} c_k \binom{n-1}{k-1} \tilde{x}_0^{k-1}(\mu) > 0$ and $\tilde{x}_0(\nu) > \tilde{x}_0(\mu)$. This finishes the proof. 
\end{proof}

\section{Convexity of General Inverse \texorpdfstring{$\sigma_k$}{} Equations}
\label{sec:3}
In this section, let $f(\lambda) \coloneqq  \lambda_1 \cdots \lambda_n - \sum_{k = 0}^{n-1} c_k \sigma_k(\lambda)$ be a general inverse $\sigma_k$ type multilinear polynomial. If $\{f(\lambda) = 0\}$ is contained in $q + \Gamma_n$ for some $q \in \mathbb{R}^n$, then we use a classical way to prove the convexity of this level set $\{f(\lambda) = 0\}$. \smallskip

By doing the substitution (\ref{eq:2.21}), we may assume $c_{n-1} = 0$ and consider the following general inverse $\sigma_k$ equation
\begin{align*}
f(\lambda) = f(\lambda_1, \cdots, \lambda_n) = \lambda_1 \cdots \lambda_n - \sum_{k=0}^{n-2} c_k \sigma_k(\lambda) = 0.
\end{align*}There are two ways to compute the convexity, first, if we write 
\begin{align*}
\label{eq:3.1}
h = \frac{ \sum_{k=0}^{n-2} c_k \sigma_k(\lambda) }{\lambda_1 \cdots \lambda_n }, \tag{3.1}
\end{align*}then we have the following.

\hypertarget{L:3.1}{\begin{flemma}}
If the following $n-1 \times n-1$ Hermitian matrix is positive semi-definite
\begin{align*}
\label{eq:3.2}
\Bigl( h_{ij} + h_{nn} \frac{h_i h_j}{h_n^2} - h_{in} \frac{h_j}{h_n} - h_{jn} \frac{h_i}{h_n}  
\Bigr)_{i, j \in \{1, \cdots, n-1\}}, \tag{3.2}
\end{align*}then the level set $\{ h = c \}$ is convex. Here, $h_i \coloneqq \partial h / \partial \lambda_i$ and $h_{ij} \coloneqq \partial^2h / \partial \lambda_i \partial \lambda_j$.
\end{flemma}

\begin{proof}
Let $V = (V_1, \cdots, V_n) \in T_{ {\lambda}} \bigl\{ h = c   \bigr\}$ be a tangent vector, which gives, $\sum_i h_i V_i =0$. Then, to get convexity, which is equivalent to the following quantity  
\begin{align*}
\label{eq:3.3}
\sum_{i, j} h_{ij} V_i  {V}_{\bar{j}}  \tag{3.3}
\end{align*}is non-negative. Since $V$ is a tangent vector, we can write $V_n = -  \sum_{i=1}^{n-1} h_i V_i/{h_n}$. By plugging in quantity (\ref{eq:3.3}), we obtain
\begin{align*}
\label{eq:3.4}
\sum_{i, j} h_{ij} V_i V_{\bar{j}} &= \sum_{i = 1}^{n-1} \Bigl ( h_{ii} + h_{nn} \frac{h_i^2}{h_n^2} - 2h_{in} \frac{h_i}{h_n}      \Bigr) |V_i|^2 \tag{3.4} \\
&\kern2em   + \sum_{1 \leq i < j \leq n-1}  \Bigl (h_{ij} + h_{nn}  \frac{h_i h_j}{h_n^2} - h_{in} \frac{h_j}{h_n} - h_{jn} \frac{h_i}{h_n}   \Bigr  ) (V_i V_{\bar{j}} + V_{\bar{i}} V_j ).
\end{align*}So, if the following $n-1 \times n-1$ Hermitian matrix is positive semi-definite
\begin{align*}
\Bigl( h_{ij} + h_{nn} \frac{h_i h_j}{h_n^2} - h_{in} \frac{h_j}{h_n} - h_{jn} \frac{h_i}{h_n}  
\Bigr)_{i, j \in \{1, \cdots, n-1\}},
\end{align*}then the quantity (\ref{eq:3.4}) is non-negative. This implies that the level set $\{ h = c \}$ is convex.
\end{proof}

If we write 
\begin{align*}
\label{eq:3.5}
\lambda_n = \frac{\sum_{k = 0}^{n-2} c_k \sigma_k(\lambda_{; n})}{\lambda_1 \cdots \lambda_{n-1} - \sum_{k=1}^{n-2} c_k \sigma_{k-1}(\lambda_{; n})}, \tag{3.5}
\end{align*}then the Hessian of $\lambda_n$ is related to $n-1 \times n-1$ Hermitian matrix (\ref{eq:3.3}) as follows.

\hypertarget{L:3.2}{\begin{flemma}}
Let $h =  {\sum_{k = 0}^{n-2} c_k \sigma_k(\lambda)}/{\sigma_n(\lambda)}$, then we have
\begin{align*}
h_i = -\frac{\sum_{k= 0}^{n-2} c_k \sigma_k(\lambda_{;i})}{\lambda_i \sigma_n(\lambda) };\quad   h_{ij} = \frac{\sum_{k =0}^{n-2} c_k \sigma_k(\lambda_{; i, j})}{ \lambda_i \lambda_j \sigma_n(\lambda) }(1 + \delta_{ij}),
\end{align*}where we denote by $h_i \coloneqq \partial h/ \partial \lambda_i$ and $h_{ij} \coloneqq \partial^2 h/ \partial \lambda_i \partial \lambda_j$. Moreover, we have
\begin{align*}
\label{eq:3.6}
&\kern-1em h_{ij} + h_{nn} \frac{h_i h_j}{h_n^2} - h_{in} \frac{h_j}{h_n} - h_{jn} \frac{h_i}{h_n}   \tag{3.6} \\
&= \frac{ ( C - \lambda_i C_{1; i}      )    \bigl(   C  - \lambda_j \lambda_n  C_{2; j, n}  \bigr) + ( C - \lambda_j C_{1; j}      )   \bigl(   C - \lambda_i \lambda_n   C_{2; i, n}  \bigr)   }{  \lambda_i \lambda_j \sigma_n(\lambda)   (  C -   \lambda_n C_{1;  n}      )  } \\
&\kern2em - \frac{     ( C - \lambda_n C_{1; n}      )   \bigl(  C - \lambda_i \lambda_j  C_{2; i, j}  \bigr)  }{  \lambda_i \lambda_j  \sigma_n(\lambda)  (  C -   \lambda_n C_{1;  n}      )  }(1- \delta_{ij}) \\
&= \frac{C_{0; n}}{\lambda_n \sigma_n(\lambda) } \frac{\partial^2}{\partial \lambda_i \partial \lambda_j} \lambda_n.
\end{align*}Here, we denote 
\begin{align*}
C  &\coloneqq \sum_{k = 0}^{n-2} c_k \sigma_k(\lambda) = h \sigma_n(\lambda) = h \lambda_1 \cdots \lambda_n; \,
&C_{1; i}& \coloneqq \sum_{k = 1}^{n-2} c_k \sigma_{k-1}( \lambda_{; i}); \\ 
C_{0; n} &\coloneqq   \sum_{k = 0}^{n-2} c_k \sigma_k(\lambda_{;n}) = C - \lambda_n C_{1; n}; \,
\text{ and    } \, 
&C_{2; i, j}&  \coloneqq \sum_{k = 2}^{n-2} c_k \sigma_{k-2}(\lambda_{; i, j}).
\end{align*}
\end{flemma}

\begin{proof}First, we have the following
\begin{align*}
0 = h \lambda_1 \cdots \lambda_n - \sum_{k=0}^{n-2} c_k \sigma_{k} (\lambda) = \lambda_i \Bigl(  h \lambda_1 \cdots \lambda_n /\lambda_i - \sum_{k=1}^{n-2} c_k \sigma_{k-1}(\lambda_{; i})    \Bigr) - \sum_{k = 0}^{n-2} c_k \sigma_k (\lambda_{; i}).
\end{align*}This implies that
\begin{align*}
\label{eq:3.7}
 C_{0; i} \coloneqq \sum_{k = 0}^{n-2} c_k \sigma_k (\lambda_{; i}) = h \lambda_1 \cdots \lambda_n  - \lambda_i \sum_{k=1}^{n-2} c_k \sigma_{k-1}(\lambda_{; i})  = C - \lambda_i C_{1; i}. \tag{3.7}
\end{align*}In addition, for $i \neq j$, we get
\begin{align*}
\label{eq:3.8}
C_{0; i} + C_{0; j} - C_{0; i, j} =  \sum_{k = 0}^{n-2} c_k \sigma_k (\lambda_{; i}) + \sum_{k = 0}^{n-2} c_k \sigma_k (\lambda_{; j}) - \sum_{k = 0}^{n-2} c_k \sigma_k (\lambda_{; i, j}) &= C - \lambda_i \lambda_j C_{2; i, j},    \tag{3.8}
\end{align*}where we denote $C_{0; i, j}  = \sum_{k = 0}^{n-2} c_k \sigma_k (\lambda_{; i, j})$. By (\ref{eq:3.8}), we obtain 
\begin{align*}
\label{eq:3.9}
C_{0; i, j} &=   C_{0; i}  +  C_{0; j} -  C  + \lambda_i \lambda_j C_{2; i, j}.    \tag{3.9}
\end{align*}For $i = j$, on $h = \sum_{k=0}^{n-2} c_k \sigma_k(\lambda) / \sigma_n(\lambda)$, we have 
\begin{align*}
\lambda_n = \frac{  \sum_{k = 0}^{n-2} c_k \sigma_k(\lambda_{; n})  }{h \lambda_1 \cdots \lambda_{n-1}  - \sum_{k=1}^{n-2}  c_k \sigma_{k-1}(\lambda_{; n})  }  =  \frac{C_{0; n}}{ h \lambda_1 \cdots \lambda_{n-1}  - C_{1; n}  }
\end{align*}this implies that
\begingroup
\allowdisplaybreaks
\begin{align*}
\label{eq:3.10}
\frac{\partial^2}{\partial \lambda_i^2} \lambda_n &= \frac{\partial}{\partial \lambda_i} \Bigl(     \frac{\sum_{k = 1}^{n-2} c_k \sigma_{k-1}(\lambda_{; i, n})      }{   h \lambda_1 \cdots \lambda_{n-1}  - \sum_{k=1}^{n-2}  c_k \sigma_{k-1}(\lambda_{; n})    }    \tag{3.10}  \\
&\kern6em-   \frac{  \sum_{k = 0}^{n-2} c_k \sigma_k(\lambda_{; n}) \bigl( h \lambda_1 \cdots \lambda_{n-1}/\lambda_i  - \sum_{k=2}^{n-2}  c_k \sigma_{k-2}(\lambda_{; i, n})  \bigr)  }{ \bigl( h \lambda_1 \cdots \lambda_{n-1}  - \sum_{k=1}^{n-2}  c_k \sigma_{k-1}(\lambda_{; n})  \bigr)^2 }    \Bigr) \\
&= 2   \frac{\sum_{k = 0}^{n-2} c_k \sigma_{k}(\lambda_{;  n})    \bigl( h \lambda_1 \cdots \lambda_{n-1}/\lambda_i  - \sum_{k=2}^{n-2}  c_k \sigma_{k-2}(\lambda_{; i, n})  \bigr)^2   }{ \bigl( h \lambda_1 \cdots \lambda_{n-1}  - \sum_{k=1}^{n-2}  c_k \sigma_{k-1}(\lambda_{; n})  \bigr)^3 }    \\
&\kern2em- 2   \frac{  \sum_{k = 1}^{n-2} c_k \sigma_{k-1}(\lambda_{; i, n}) \bigl( h \lambda_1 \cdots \lambda_{n-1}/\lambda_i  - \sum_{k=2}^{n-2}  c_k \sigma_{k-2}(\lambda_{; i, n})  \bigr)  }{ \bigl( h \lambda_1 \cdots \lambda_{n-1}  - \sum_{k=1}^{n-2}  c_k \sigma_{k-1}(\lambda_{; n})  \bigr)^2 }     \\
&= 2   \frac{ \lambda_n     \bigl( C  -  \lambda_i \lambda_n C_{2; i, n}  \bigr)^2   }{ \lambda_i^2  C_{0; n}^2 }   - 2   \frac{  \lambda_n C_{1; i, n}  \bigl( C  - \lambda_i \lambda_n C_{2; i, n}  \bigr)  }{ \lambda_i C_{0; n}^2 }      \\
&= 2   \frac{ \lambda_n     \bigl( C  -  \lambda_i \lambda_n C_{2; i, n}  \bigr)   }{ \lambda_i^2  C_{0; n}^2 } \Bigl(    C -  \lambda_i \lambda_n C_{2; i, n} - \lambda_i  C_{1; i, n} \Bigr) = 2   \frac{ \lambda_n C_{0; i}     \bigl( C  -  \lambda_i \lambda_n C_{2; i, n}  \bigr)   }{ \lambda_i^2  C_{0; n}^2 }. 
\end{align*}
\endgroup
For $i \neq j$, similarly we get
\begingroup
\allowdisplaybreaks
\begin{align*}
\label{eq:3.11}
\frac{\partial^2}{\partial \lambda_j  \partial \lambda_i } \lambda_n &=  \frac{\partial}{\partial \lambda_j} \Bigl(     \frac{\sum_{k = 1}^{n-2} c_k \sigma_{k-1}(\lambda_{; i, n})      }{   h \lambda_1 \cdots \lambda_{n-1}  - \sum_{k=1}^{n-2}  c_k \sigma_{k-1}(\lambda_{; n})    }  \tag{3.11}   \\
&\kern6em -   \frac{  \sum_{k = 0}^{n-2} c_k \sigma_k(\lambda_{; n}) \bigl( h \lambda_1 \cdots \lambda_{n-1}/\lambda_i  - \sum_{k=2}^{n-2}  c_k \sigma_{k-2}(\lambda_{; i, n})  \bigr)  }{ \bigl( h \lambda_1 \cdots \lambda_{n-1}  - \sum_{k=1}^{n-2}  c_k \sigma_{k-1}(\lambda_{; n})  \bigr)^2 }    \Bigr) \\
&=   \frac{\sum_{k = 2}^{n-2} c_k \sigma_{k-2}(\lambda_{; i, j, n})      }{   h \lambda_1 \cdots \lambda_{n-1}  - \sum_{k=1}^{n-2}  c_k \sigma_{k-1}(\lambda_{; n})    }  \\
&\kern2em -   \frac{  \sum_{k = 1}^{n-2} c_k \sigma_{k-1}(\lambda_{; i, n}) \bigl( h \lambda_1 \cdots \lambda_{n-1}/\lambda_j  - \sum_{k=2}^{n-2}  c_k \sigma_{k-2}(\lambda_{; j, n})  \bigr)  }{ \bigl( h \lambda_1 \cdots \lambda_{n-1}  - \sum_{k=1}^{n-2}  c_k \sigma_{k-1}(\lambda_{; n})  \bigr)^2 }     \\
&\kern2em -   \frac{  \sum_{k = 1}^{n-2} c_k \sigma_{k-1}(\lambda_{; j, n}) \bigl( h \lambda_1 \cdots \lambda_{n-1}/\lambda_i  - \sum_{k=2}^{n-2}  c_k \sigma_{k-2}(\lambda_{; i, n})  \bigr)  }{ \bigl( h \lambda_1 \cdots \lambda_{n-1}  - \sum_{k=1}^{n-2}  c_k \sigma_{k-1}(\lambda_{; n})  \bigr)^2 } \\
&\kern2em -   \frac{  \sum_{k = 0}^{n-2} c_k \sigma_k(\lambda_{; n}) \bigl( h \lambda_1 \cdots \lambda_{n-1}/\lambda_i \lambda_j - \sum_{k=3}^{n-2}  c_k \sigma_{k-2}(\lambda_{; i,  j, n})  \bigr)  }{ \bigl( h \lambda_1 \cdots \lambda_{n-1}  - \sum_{k=1}^{n-2}  c_k \sigma_{k-1}(\lambda_{; n})  \bigr)^2 }   \\
&\kern2em + 2  \frac{  \sum_{k = 0}^{n-2} c_k \sigma_k(\lambda_{; n}) \bigl( h \lambda_1 \cdots \lambda_{n-1}/\lambda_i  - \sum_{k=2}^{n-2}  c_k \sigma_{k-2}(\lambda_{; i, n})  \bigr)  }{ \bigl( h \lambda_1 \cdots \lambda_{n-1}  - \sum_{k=1}^{n-2}  c_k \sigma_{k-1}(\lambda_{; n})  \bigr)^3 }  \\
&\kern2em \times  \bigl( h \lambda_1 \cdots \lambda_{n-1}/\lambda_j  - \sum_{k=2}^{n-2}  c_k \sigma_{k-2}(\lambda_{; j, n})  \bigr)  \\
&= \frac{\lambda_n}{ \lambda_i \lambda_j C_{0; n}^2}  \Bigl(  \lambda_i \lambda_j C_{2; i, j, n} C_{0; n}   - C_{0; n}  (C- \lambda_i \lambda_j \lambda_n C_{3; i, j, n})  \\
&\kern6em - \lambda_i   C_{1; i,  n} (C - \lambda_j \lambda_n C_{2; j, n} )    - \lambda_j   C_{1; j,  n} (C - \lambda_i \lambda_n C_{2; i, n} )   \\
&\kern15em + 2  (C- \lambda_i \lambda_n C_{2; i, n}) (C- \lambda_j \lambda_n C_{2; j, n})   \Bigr )  \\
&= \frac{\lambda_n}{ \lambda_i \lambda_j C_{0; n}^2}  \Bigl( C_{0; j} (C - \lambda_i \lambda_n C_{2; i, n}) + C_{0; i} (C - \lambda_j \lambda_n C_{2; j, n})  \\
&\kern17em - C_{0; n} (C - \lambda_i \lambda_j C_{2; i, j}) \Bigr). 
\end{align*}
\endgroup

Here, for convenience, we denote 
\begin{align*}
C_{1; i, j} \coloneqq \sum_{k = 1}^{n-2} c_k \sigma_{k-1}(\lambda_{; i, j}); \, 
C_{2; i, j, n} \coloneqq  \sum_{k = 2}^{n-2} c_k \sigma_{k-2}(\lambda_{; i, j, n}); \,  \text{ and } C_{3; i, j, n} \coloneqq  \sum_{k = 3}^{n-2} c_k \sigma_{k-3}(\lambda_{; i, j, n}).
\end{align*} On the other hand, for $i = j$, we have

\begingroup
\allowdisplaybreaks
\begin{align*}
\label{eq:3.12}
&\kern-1em h_{ii} + h_{nn} \frac{h_i^2}{h_n^2} - 2h_{in} \frac{h_i}{h_n}    \tag{3.12}  \\
&= 2 \frac{\sum_{k =0}^{n-2} c_k \sigma_k(\lambda_{; i})}{\lambda_i^2 \sigma_n(\lambda)  }  + 2 \frac{\sum_{k= 0}^{n-2} c_k \sigma_k(\lambda_{;n})}{\lambda_n^2  \sigma_n(\lambda) }   \Bigl(   \frac{\lambda_n   \sum_{k=0}^{n-2} c_k \sigma_k(\lambda_{; i})}{\lambda_i   \sum_{k=0}^{n-2} c_k \sigma_k(\lambda_{; n})}  \Bigr)^2 \\
&\kern2em - 2 \frac{\sum_{k =0}^{n-2} c_k \sigma_k(\lambda_{; i, n})}{\lambda_i \lambda_n \sigma_n(\lambda) }   \frac{\lambda_n   \sum_{k=0}^{n-2} c_k \sigma_k(\lambda_{; i})}{\lambda_i    \sum_{k=0}^{n-2} c_k \sigma_k(\lambda_{; n})}   \\
&= \frac{2 \sum_{k = 0}^{n-2} c_k \sigma_k(\lambda_{; i}) }{ \lambda_i^2 \sigma_n(\lambda)  \sum_{k=0}^{n-2} c_k \sigma_k(\lambda_{; n}) } \Bigl (  \sum_{k = 0}^{n-2} c_k \sigma_k(\lambda_{; i})  + \sum_{k = 0}^{n-2} c_k \sigma_k(\lambda_{; n}) - \sum_{k = 0}^{n-2} c_k \sigma_k(\lambda_{; i, n})        \Bigr) \\
&= \frac{2 C_{0; i} }{  \lambda_i^2 \sigma_n(\lambda)  C_{0; n} } \Bigl ( C   - \lambda_i \lambda_n C_{2; i, n}       \Bigr) = \frac{C_{0; n} }{\lambda_n \sigma_n(\lambda)    }  \frac{\partial^2 }{\partial \lambda_i^2} \lambda_n.
\end{align*}
\endgroup

For $i \neq j$, we obtain
\begingroup
\allowdisplaybreaks
\begin{align*}
\label{eq:3.13}
&\kern-1em h_{ij} + h_{nn} \frac{h_i h_j}{h_n^2} - h_{in} \frac{h_j}{h_n} - h_{jn} \frac{h_i}{h_n}  \tag{3.13} \\
&= \frac{\sum_{k =0}^{n-2} c_k \sigma_k(\lambda_{; i, j})}{\lambda_i \lambda_j  \sigma_n(\lambda) }  + 2 \frac{\sum_{k= 0}^{n-2} c_k \sigma_k(\lambda_{;n})}{ \lambda_n^2 \sigma_n(\lambda) } \frac{\lambda_n   \sum_{k=0}^{n-2} c_k \sigma_k(\lambda_{; i})}{\lambda_i   \sum_{k=0}^{n-2} c_k \sigma_k(\lambda_{; n})}    \frac{\lambda_n   \sum_{k=0}^{n-2} c_k \sigma_k(\lambda_{; j})}{\lambda_j    \sum_{k=0}^{n-2} c_k \sigma_k(\lambda_{; n})} \\
&\kern2em - \frac{\sum_{k =0}^{n-2} c_k \sigma_k(\lambda_{; i, n})}{ \lambda_i \lambda_n \sigma_n(\lambda)}   \frac{\lambda_n   \sum_{k=0}^{n-2} c_k \sigma_k(\lambda_{; j})}{\lambda_j    \sum_{k=0}^{n-2} c_k \sigma_k(\lambda_{; n})}  - \frac{\sum_{k =0}^{n-2} c_k \sigma_k(\lambda_{; j, n})}{ \lambda_j \lambda_n  \sigma_n(\lambda) }   \frac{\lambda_n   \sum_{k=0}^{n-2} c_k \sigma_k(\lambda_{; i})}{\lambda_i    \sum_{k=0}^{n-2} c_k \sigma_k(\lambda_{; n})}  \\
&= \frac{1}{ \lambda_i \lambda_j \sigma_n(\lambda)   C_{0; n} } \Bigl(   C_{0; i, j} C_{0; n}  + 2  C_{0; i} C_{0; j}     - C_{0; i, n}  C_{0; j}  - C_{0; j, n}  C_{0; i}  \Bigr). 
\end{align*}
\endgroup

We can simplify the terms in the parentheses of quantity (\ref{eq:3.13}):
\begingroup
\allowdisplaybreaks
\begin{align*}
\label{eq:3.14}
&\kern-1em C_{0; i, j} C_{0; n}  + 2  C_{0; i} C_{0; j}     - C_{0; i, n}  C_{0; j}  - C_{0; j, n}  C_{0; i}  \tag{3.14}  \\
&= C_{0; n}  \Bigl( C_{0; i} + C_{0; j} -  C + \lambda_i \lambda_j C_{2; i, j} \Bigr) + 2 C_{0; i}  C_{0; j} \\
&\kern2em -  C_{0; j}  \Bigl(  C_{0; i} + C_{0; n} -  C  + \lambda_i \lambda_n C_{2; i, n} \Bigr)  \\
&\kern2em -  C_{0; i}  \Bigl(  C_{0; j} + C_{0; n} -  C  + \lambda_j \lambda_n C_{2; j, n} \Bigr)  \\
&=  C_{0; j}  \Bigl(  C  - \lambda_i \lambda_n C_{2; i, n} \Bigr) +   C_{0; i}  \Bigl(   C  - \lambda_j \lambda_n  C_{2; j, n} \Bigr)  -C_{0; n}   \Bigl(  C - \lambda_i \lambda_j  C_{2; i, j} \Bigr). 
\end{align*}
\endgroup

Hence, by (\ref{eq:3.11}), (\ref{eq:3.13}), and (\ref{eq:3.14}), we get
\begin{align*}
h_{ij} + h_{nn} \frac{h_i h_j}{h_n^2} - h_{in} \frac{h_j}{h_n} - h_{jn} \frac{h_i}{h_n}   = \frac{ C_{0; n}}{ \lambda_n \sigma_n(\lambda) }  \frac{\partial^2}{\partial \lambda_i \partial \lambda_j} \lambda_n. 
\end{align*} This finishes the proof. 
\end{proof}

\hypertarget{T:3.1}{\begin{fthm}[Convexity of the general inverse $\sigma_k$ equation]}
Let $f(\lambda) \coloneqq  \lambda_1 \cdots \lambda_n - \sum_{k = 0}^{n-1} c_k \sigma_k(\lambda)$ be a general inverse $\sigma_k$ type multilinear polynomial. If the diagonal restriction $r_f(x) = x^n - \sum_{k = 0}^{n-1} c_k \binom{n}{k} x^k$ is strictly right-Noetherian, then $\{f=0\}$ is strictly convex. 
\end{fthm}

\begin{proof}
We assume $c_{n-1} = 0$ for convenience. We prove the statement by mathematical induction on the degree $n$. When $n = 2$, $f(\lambda) =  {\lambda}_1  {\lambda}_2    - {c}_0$ with $c_0 > 0$. Let $(\mu_1, \mu_2), (\nu_1, \nu_2) \in   \{ f \geq 0 \}$. That is, $f(\mu_1, \mu_2) = \mu_1 \mu_2 - c_0 \geq 0$ and $f(\nu_1, \nu_2) = \nu_1 \nu_2 - c_0 \geq 0$. For any ${t} \in (0, 1)$, we have 
\begin{align*}
&\kern-2em f( (1- {t}) \mu_1+   {t}\nu_1,  (1- {t}) \mu_2 +  {t}\nu_2) \\
&= (1-t)^2 \mu_1 \mu_2 +  {t}(1- {t}) (\mu_1 \nu_2 + \mu_2 \nu_1) + t^2 \nu_1 \nu_2 - c_0 \\
&\geq  (1-t)^2 \mu_1 \mu_2 + 2  {t}(1- {t}) \sqrt{ \mu_1 \mu_2 \nu_1\nu_2 } + t^2 \nu_1 \nu_2 - c_0 \geq 0
\end{align*}
and equality holds when $(\mu_1, \mu_2) = (\nu_1, \nu_2)$ and $\mu_1 \mu_2 = c_0 = \nu_1 \nu_2$. Hence, the set $\{ f(\lambda) = \lambda_1 \lambda_2 - c_0 \geq 0 \}$ is strictly convex. \smallskip

Suppose the statement is true when $n = m-1$. When $n = m \geq 3$, let $\mu = (\mu_1, \cdots, \mu_m), \nu = (\nu_1, \cdots, \nu_m) \in  \{ f \geq 0 \}$. If $\mu_i = \nu_i$ for some $i \in \{1, \cdots, m\}$, then by fixing the $i$-th position and Proposition~\hyperlink{P:2.8}{2.8}, we may view $f(\bullet, \mu_i, \bullet)$ as a degree $m-1$ strictly $\Upsilon$-stable general inverse $\sigma_k$ equation. By mathematical induction, for any $t \in (0, 1)$, we have $f( (1-t)\mu + t \nu ) > 0$. So, we only need to consider the case that $\mu_i \neq \nu_i$ for all $i \in \{1, \cdots, m\}$. 
In addition, it suffices to show that $\{ f \geq 0 \}$ is convex. We show that if $\{ f \geq 0 \}$ is convex, then $\{ f \geq 0 \}$ is indeed strictly convex. If $\{ f \geq 0 \}$ is convex, let $\mu, \nu \in   \{ f \geq 0 \}$ with $\mu_i \neq \nu_i$ for all $i \in \{1, \cdots, m\}$. We can view $f( (1-t) \mu + t \nu)$ as a degree $m$ polynomial in terms of $t$, so there are finitely many roots by the fundamental theorem of algebra. Suppose that there exists $\tilde{t} \in (0, 1)$ such that $f( (1-\tilde{t}) \mu + \tilde{t} \nu ) = 0$. 
Let $\tilde{\mu}$ and $\tilde{\nu}$ denote the $(m-1)$-tuples $(\mu_2, \cdots, \mu_m)$ and $(\nu_2, \cdots, \nu_m)$, respectively. Since $f( (1-t) \mu  + t \nu )$ only has finitely many roots, there exists $0 < t_0 < \tilde{t}$ and $1 > t_1 > \tilde{t}$ such that $f( (1-t_0) \mu +  t_0 \nu) > 0$ and $f( (1-t_1) \mu + t_1 \nu) > 0$. For $\epsilon > 0$ sufficiently small, we get $f( (1-t_0) \mu_1 + t_0 \nu_1 - \epsilon, (1-t_0) \tilde{\mu} + t_0 \tilde{\nu} ) > 0$ and $f( (1-\tilde{t}) \mu_1 +  \tilde{t} \nu_1 - \frac{t_1 - \tilde{t}}{{t_1} - t_0}\epsilon, (1-\tilde{t}) \tilde{\mu} +   \tilde{t} \tilde{\nu} ) < 0$, which is a contradiction if $\{ f \geq 0\}$ is convex. \smallskip

Moreover, it suffices to show that $\{ f(\lambda) \geq 0 \}$ is convex if $c \in \tilde{\mathscr{C}}_m$ is in the generic strata. Let $c \in \tilde{\mathscr{C}}_m$ not in the generic strata and suppose $\{ f(\lambda) = \lambda_1 \cdots \lambda_m - \sum_{k=0}^{m-2} c_k \sigma_k(\lambda) \geq 0\}$ is not convex. If there exists 
$\mu, \nu \in \{ f(\lambda) = \lambda_1 \cdots \lambda_m - \sum_{k=0}^{m-2} c_k \sigma_k(\lambda) \geq 0\}$ and $\tilde{t} \in (0, 1)$ such that $f((1-\tilde{t})\mu + \tilde{t} \nu) < 0$. 
By choosing $\epsilon > 0$ sufficiently small, we get $f(\mu_1 + \epsilon, \tilde{\mu}) > 0$, $f(\nu_1 + \epsilon, \tilde{\nu}) > 0$, and $f((1-\tilde{t})\mu_1 + \tilde{t} \nu_1 + \epsilon, (1-\tilde{t})\tilde{\mu} + \tilde{t} \tilde{\nu}  ) < 0$. We \hypertarget{claim1}{\textbf{claim}} that there exists $\tilde{c} \in \tilde{\mathscr{C}}_m$ in the generic strata with $\{  \lambda_1 \cdots \lambda_m - \sum_{k=0}^{m-2} \tilde{c}_k \sigma_k(\lambda) \geq 0\} \subset \{   \lambda_1 \cdots \lambda_m - \sum_{k=0}^{m-2} c_k \sigma_k(\lambda) \geq 0\}$ such that $(\mu_1 + \epsilon, \tilde{\mu}), (\nu_1 + \epsilon, \tilde{\nu}) \in \{ \lambda_1 \cdots \lambda_m - \sum_{k=0}^{m-2} \tilde{c}_k \sigma_k(\lambda) \geq 0\}$. If \bhyperlink{claim1}{\textbf{claim 1}} holds and $\{ \lambda_1 \cdots \lambda_m - \sum_{k=0}^{m-2} \tilde{c}_k \sigma_k(\lambda) \geq 0\}$ is convex when $\tilde{c} \in \tilde{\mathscr{C}}_m$ is in the generic strata, then we get $((1-\tilde{t})\mu_1 + \tilde{t} \nu_1 + \epsilon, (1-\tilde{t})\tilde{\mu} + \tilde{t} \tilde{\nu}  ) \in \{  \lambda_1 \cdots \lambda_m - \sum_{k=0}^{m-2} \tilde{c}_k \sigma_k(\lambda) \geq 0\} \subset \{   \lambda_1 \cdots \lambda_m - \sum_{k=0}^{m-2} c_k \sigma_k(\lambda) \geq 0\}$, which leads to a contradiction.   

To justify \bhyperlink{claim1}{\textbf{claim 1}}, for any $\mu \in \{ f(\lambda) = \lambda_1 \cdots \lambda_m - \sum_{k=0}^{m-2} c_k \sigma_k(\lambda) > 0\}$, we consider the following continuous function, which is the composition of two continuous functions:
\begin{center} 
\begin{tikzcd}[row sep=-0.2em]
\tilde{\mathcal{X}}_m \ar[r,"\psi"] & \tilde{\mathscr{C}}_m \ar[r,"\mu"] & \mathbb{R} \\
(x_{m-2}, \cdots, x_1, x_0) \arrow[r,mapsto]  & (d_{m-2}, \cdots, d_1, d_0) \arrow[r,mapsto] & \mu_1 \cdots \mu_m - \sum_{k=0}^{m-2} d_k \sigma_k(\mu)
\end{tikzcd}
\end{center}
Here, $\psi \colon \tilde{\mathcal{X}}_m \rightarrow \tilde{\mathscr{C}}_m $ is the map in Lemma~\hyperlink{L:2.9}{2.9}. Let $x_l(c)$ be the largest real root of the $l$-th derivative of $x^m - \sum_{k=0}^{m-2} c_k \binom{m}{k} x^k$ for $l \in \{0, \cdots, m-2\}$.
Since $\mu \circ \psi (x_{m-2}(c), \cdots, x_{1}(c), x_{0}(c)) = \mu_1 \cdots \mu_m - \sum_{k=0}^{m-2} c_k \sigma_k(\mu) > 0$, by the continuity of the function $\mu \circ \psi$, there exists $N_{\mu} > 0$ sufficiently large such that for any $N \geq N_{\mu}$, we have
\begin{align*}
\mu \circ \psi \Bigl(x_{m-2}(c) + \frac{1}{(m-1)N}, \cdots, x_1(c) + \frac{1}{2N}, x_0(c) + \frac{1}{N}   \Bigr) > 0.
\end{align*}Here, $x_0(c) + \frac{1}{N} > x_1(c) + \frac{1}{2N} > \cdots > x_{m-2}(c) + \frac{1}{(m-1)N}$, so it is the generic strata. 
By picking $N \geq \max \{ N_{(\mu_1 + \epsilon, \tilde{\mu})}, N_{(\nu_1 + \epsilon, \tilde{\nu})} \}$ and by the \bhyperlink{T:2.4}{$\Upsilon$-dominance Theorem}, we get 
\begin{align*}
(\mu_1 + \epsilon, \tilde{\mu}), (\nu_1 + \epsilon, \tilde{\nu}) \in \Bigl\{ \lambda_1 \cdots \lambda_m - \sum_{k=0}^{m-2} \tilde{c}_k \sigma_k(\lambda) \geq 0 \Bigr\} \subset  \Bigl\{   \lambda_1 \cdots \lambda_m - \sum_{k=0}^{m-2} c_k \sigma_k(\lambda) \geq 0 \Bigr\},
\end{align*}where $(\tilde{c}_{m-2}, \cdots, \tilde{c}_1, \tilde{c}_0 ) = \psi (x_{m-2}(c) + \frac{1}{(m-1)N}, \cdots, x_1(c) + \frac{1}{2N}, x_0(c) + \frac{1}{N}  )$. Thus, we confirm \bhyperlink{claim1}{\textbf{claim 1}} and only need to show that $\{ f(\lambda) \geq 0 \}$ is convex if $c \in \tilde{\mathscr{C}}_m$ is in the generic strata.
 \smallskip

Now, to prove $\{ f \geq 0 \}$ is convex when $c \in \tilde{\mathscr{C}}_m$ is in the generic strata, by the intermediate value theorem, it suffices to consider $\mu, \nu   \in  \partial \{ f \geq 0 \}$ and show that $(1-t) \mu+ t \nu \in \{ f \geq 0\}$ for any $t \in (0, 1)$. 
First, for convenience, we assume $\mu_1 > \nu_1$. By fixing $\mu_1$, we may view $\tilde{\mu} = (\mu_2, \cdots, \mu_{m})$ is on the level set of the following degree $m-1$ strictly $\Upsilon$-stable general inverse $\sigma_k$ equation:
\begin{align*}
\label{eq:3.15}
\mu_1 \Bigl(\lambda_2 \cdots \lambda_{m} - \sum_{k=1}^{m-2} c_k \sigma_{k-1}(\lambda_{; 1}) \Bigr) - \sum_{k=0}^{m-2} c_k \sigma_{k}(\lambda_{; 1}) = 0. \tag{3.15}
\end{align*}Similarly, we view $\tilde{\nu} = (\nu_2, \cdots, \nu_{m})$ is on the level set of the following degree $m-1$ strictly $\Upsilon$-stable general inverse $\sigma_k$ equation:
\begin{align*}
\label{eq:3.16}
\nu_1 \Bigl(\lambda_2 \cdots \lambda_{m} - \sum_{k=1}^{m-2} c_k \sigma_{k-1}(\lambda_{; 1}) \Bigr) - \sum_{k=0}^{m-2} c_k \sigma_{k}(\lambda_{; 1}) = 0. \tag{3.16}
\end{align*}

Let $\tilde{x}_l(\mu_1)$ be the largest real root of the $l$-th derivative of the diagonal restriction $\mu_1 (x^{m-1} - \sum_{k=1}^{m-2} c_k \binom{m-1}{k-1} x^{k-1} ) - \sum_{k=0}^{m-2} c_k \binom{m-1}{k} x^k$ of quantity (\ref{eq:3.15}) for $l \in \{0, \cdots, m-2\}$. Similarly, let $\tilde{x}_l(\nu_1)$ be the largest real root of the $l$-th derivative of the diagonal restriction $\nu_1 (x^{m-1} - \sum_{k=1}^{m-2} c_k \binom{m-1}{k-1} x^{k-1} ) - \sum_{k=0}^{m-2} c_k \binom{m-1}{k} x^k$ of quantity (\ref{eq:3.16}) for $l \in \{0, \cdots, m-2\}$. Since $\mu_1 > \nu_1$, by Lemma~\hyperlink{L:2.12}{2.12}, we have $\tilde{x}_l(\mu_1) < \tilde{x}_l(\nu_1)$ for all $l \in \{0, 1, \cdots, m-2\}$. By the \bhyperlink{T:2.4}{$\Upsilon$-dominance Theorem}, we have the following set inclusion relation:  
\begin{align*}
&\kern-2em   \Bigl\{ \nu_1 \Bigl(\lambda_2 \cdots \lambda_{m} - \sum_{k=1}^{m-2} c_k \sigma_{k-1}(\lambda_{; 1}) \Bigr) - \sum_{k=0}^{m-2} c_k \sigma_{k}(\lambda_{; 1}) \geq 0 \Bigr \}   \\
& \bigsubset  \Bigl\{ \mu_1 \Bigl(\lambda_2 \cdots \lambda_{m} - \sum_{k=1}^{m-2} c_k \sigma_{k-1}(\lambda_{; 1}) \Bigr) - \sum_{k=0}^{m-2} c_k \sigma_{k}(\lambda_{; 1}) \geq 0 \Bigr \}.
\end{align*}
We show that for any $\tilde{\mu} \in \{f(\mu_1, \bullet) = 0\}$, $\tilde{\nu} \in \{f(\nu_1, \bullet) = 0\}$, $s \in (0, 1)$, and let $\tau_1 \coloneqq (1-s) \mu_1 + s \nu_1$, there exists a unique $\tilde{s} = \tilde{s}(s, \tilde{\mu}, \tilde{\nu}) \in (0, 1)$ such that $f(\tau_1, (1-\tilde{s})  \tilde{\mu} + \tilde{s} \tilde{\nu}) = 0$. We have
\begin{align*}
f(\tau_1, \tilde{\mu}) &= \tau_1 \Bigl(\mu_2 \cdots \mu_{m} - \sum_{k=1}^{m-2} c_k \sigma_{k-1}(\mu_{; 1}) \Bigr) - \sum_{k=0}^{m-2} c_k \sigma_{k}(\mu_{; 1}) \\
&=(\tau_1 - \mu_1) \Bigl(\mu_2 \cdots \mu_{m} - \sum_{k=1}^{m-2} c_k \sigma_{k-1}(\mu_{; 1}) \Bigr) < 0
\end{align*}and
\begingroup
\allowdisplaybreaks
\begin{align*}
f(\tau_1, \tilde{\nu}) &= \tau_1 \Bigl(\nu_2 \cdots \nu_{m} - \sum_{k=1}^{m-2} c_k \sigma_{k-1}(\nu_{; 1}) \Bigr) - \sum_{k=0}^{m-2} c_k \sigma_{k}(\nu_{; 1}) \\
&=(\tau_1 - \nu_1) \Bigl(\nu_2 \cdots \nu_{m} - \sum_{k=1}^{m-2} c_k \sigma_{k-1}(\nu_{; 1}) \Bigr) > 0.
\end{align*}
\endgroup
These imply that $\tilde{\nu} \in \{ f(\tau_1, \bullet) > 0\}$ and $\tilde{\mu} \notin \{ f(\tau_1, \bullet) \geq 0\}$. By the intermediate value theorem, there exists $\tilde{s} \in (0, 1)$ such that $f(\tau_1, (1-\tilde{s}) \tilde{\mu} + \tilde{s} \tilde{\nu}) = 0$. The uniqueness of $\tilde{s}$ is ensured because the set $\{f(\tau_1, \bullet) \geq 0\}$ is strictly convex, which is obtained by mathematical induction.  


We \hypertarget{claim2}{\textbf{claim}} that $\tilde{s} = \tilde{s}(s, \tilde{\mu}, \tilde{\nu}) \leq s$ for any $\tilde{\mu} \in \{f(\mu_1, \bullet) = 0\}$, $\tilde{\nu} \in \{f(\nu_1, \bullet) = 0\}$, and $s \in (0, 1)$. By the strict convexity of the set $\{f(\tau_1, \bullet) \geq 0\}$, for any $t \in [\tilde{s}, 1]$, we have $f(\tau_1, (1-t)\tilde{\mu} + t \tilde{\nu} ) \geq 0$.
So, if \bhyperlink{claim2}{\textbf{claim 2}} holds, then $f(\tau_1, (1-s)\tilde{\mu} + s \tilde{\nu}) = f((1-s)\mu + s \nu) \geq 0$. This implies that $\{f \geq 0\}$ is convex. Hence, we finish the proof if \bhyperlink{claim2}{\textbf{claim 2}} holds.\smallskip

First, to justify \bhyperlink{claim2}{\textbf{claim 2}}, we fix $\mu_1$ and treat $\mu_{m}$ as a function with variables $\{ \mu_2, \cdots, \mu_{m-1}\}$. Similarly, we fix $\nu_1$ and treat $\nu_{m}$ as a function with variables $\{ \nu_2, \cdots, \nu_{m-1}\}$. For $j \in \{2, \cdots, m-1\}$, we have
\begin{align*}
\label{eq:3.17}
\frac{\partial \mu_{m}}{\partial \mu_j} = -\frac{f_j(\mu_1, \cdots, \mu_m)}{f_{m}(\mu_1, \cdots, \mu_m)} = - \frac{f_j(\mu)}{f_m(\mu)} \quad \text{ and } \quad
\frac{\partial \nu_{m}}{\partial \nu_j} = -\frac{f_j(\nu_1, \cdots, \nu_m)}{f_{m}(\nu_1, \cdots, \nu_m)} = - \frac{f_j(\nu)}{f_m(\nu)}. \tag{3.17}
\end{align*}
Moreover, if we let $\tau_i = (1- \tilde{s}) \mu_i  + \tilde{s} \nu_i$ for any $i \in \{2, \cdots, m\}$ and denote the $(m-1)$-tuple $(\tau_2, \cdots, \tau_m)$ by $\tilde{\tau}$. For any $i, j \in \{2, \cdots, m-1\}$, we have
\begin{align*}
\frac{\partial \tau_i}{\partial \mu_j} &= \frac{\partial \tilde{s}}{\partial \mu_j} (\nu_i - \mu_i) + (1- \tilde{s}) \delta_{ij}; \quad &\frac{\partial \tau_{m}}{\partial \mu_j}& = \frac{\partial \tilde{s}}{\partial \mu_j} (\nu_{m} - \mu_{m}) + (1- \tilde{s}) \frac{\partial \mu_{m}}{\partial \mu_j}; \\
\frac{\partial \tau_i}{\partial \nu_j} &= \frac{\partial \tilde{s}}{\partial \nu_j} (\nu_i - \mu_i) +  \tilde{s} \delta_{ij}; \quad &\frac{\partial \tau_{m}}{\partial \nu_j}& = \frac{\partial \tilde{s} }{\partial \nu_j} (\nu_{m} - \mu_{m}) + \tilde{s} \frac{\partial \nu_{m}}{\partial \nu_j}.
\end{align*}
Second, on $\{f(\tau_1, \bullet) = 0\}$, by taking the partial derivative of $f(\tau_1, \bullet) = 0$ with respect to $\mu_j$ and $\nu_j$ for $j \in \{2, \cdots, m-1\}$, by above, we obtain
\begin{align*}
&\frac{\partial \tilde{s} }{\partial \mu_j} \sum_{k=2}^{m}f_k(\tau)(\nu_k - \mu_k) + (1- \tilde{s}) \Bigl(f_j(\tau)   + f_{m}(\tau) \frac{\partial \mu_{m}}{\partial \mu_j} \Bigr) = 0; \\
&\frac{\partial \tilde{s} }{\partial \nu_j} \sum_{k=2}^{m}f_k(\tau)(\nu_k - \mu_k)  + \tilde{s} \Bigl(f_j(\tau)   + f_{m}(\tau) \frac{\partial \nu_{m}}{\partial \nu_j} \Bigr) = 0.
\end{align*}
Since $\{f(\tau_1, \bullet) \geq 0\}$ is convex, by the supporting hyperplane theorem, we have $\sum_{k=2}^{m}f_k(\tau)(\nu_k - \mu_k) \geq 0$. Moreover, $\tilde{\nu} \in \{f(\tau_1, \bullet) > 0\}$, so $\tilde{\nu}$ cannot be on the tangent plane to $\{f(\tau_1, \bullet) = 0\}$ at the point $\tilde{\tau} \coloneqq (\tau_2, \cdots, \tau_m)$. Thus, we have $\sum_{k=2}^{m}f_k(\tau)(\nu_k - \mu_k) > 0$. By the implicit function theorem, $\frac{\partial \tilde{s}}{\partial \mu_j}$ and $\frac{\partial \tilde{s}}{\partial \nu_j}$ exist for all $j \in \{2, \cdots, m-1\}$.  
At the local extrema of $\tilde{s}$, we have $\frac{\partial \tilde{s}}{\partial \mu_j} = 0$ and $\frac{\partial \tilde{s}}{\partial \nu_j} = 0$ for all $j \in \{2, \cdots, m-1\}$. These imply that at the local extrema of $\tilde{s}$, we have
\begin{align*}
\frac{\partial \mu_{m}}{\partial \mu_j} = - \frac{f_j (\tau)}{f_{m}(\tau)} \quad  \text{ and } \quad \frac{\partial \nu_{m}}{\partial \nu_j} = - \frac{f_j (\tau)}{f_{m}(\tau)}.
\end{align*}So, at any local extremum of $\tilde{s}$, for any $j \in \{2, \cdots, m-1\}$, by (\ref{eq:3.17}), we have
\begin{align*}
\frac{f_j (\mu)}{f_{m}(\mu)} = \frac{f_j (\tau)}{f_{m}(\tau)} = \frac{f_j (\nu)}{f_{m}(\nu)} = \xi_j \in (0, \infty).
\end{align*}

If $\xi_j = 1$ for all $j \in \{2, \cdots, m-1\}$, then $\mu_2= \cdots = \mu_{m} = \tilde{x}_0(\mu_1)$, $\tau_2 = \cdots = \tau_{m} = \tilde{x}_0(\tau_1)$, and $\nu_2 = \cdots = \nu_{m} = \tilde{x}_0(\nu_1)$. Here, we denote $\tilde{x}_0(\tau_1)$ the largest real root of $\tau_1 (x^{m-1} - \sum_{k=1}^{m-2} c_k \binom{m-1}{k-1} x^{k-1} ) - \sum_{k=0}^{m-2} c_k \binom{m-1}{k} x^k$. We consider the diagonal restriction $r_f(x) \coloneqq x^m - \sum_{k=0}^{m-2} c_k \binom{m}{k} x^k$ of $f(\lambda)$, then $r_f(x)$ is strictly right-Noetherian because $c \in \tilde{\mathscr{C}}_m$. We show that the following rational function is convex on $\{ x > x_1\}$:
\begin{align*}
\label{eq:3.18}
\frac{\sum_{k=0}^{m-2} c_k \binom{m-1}{k} x^k }{x^{m-1} - \sum_{k=1}^{m-2} c_k \binom{m-1}{k-1}x^{k-1} }. \tag{3.18}
\end{align*}
Consider the second derivative of rational function (\ref{eq:3.18}) with respect to $x$, we get
\begin{align*}
&\kern-2em \frac{\partial^2}{\partial x^2} \Bigl(  \frac{\sum_{k=0}^{m-2} c_k \binom{m-1}{k} x^k }{x^{m-1} - \sum_{k=1}^{m-2} c_k \binom{m-1}{k-1}x^{k-1} } \Bigr) \\
&= \frac{\partial^2}{\partial x^2} \Bigl(  \frac{\sum_{k=0}^{m-2} c_k \binom{m-1}{k} x^k }{x^{m-1} - \sum_{k=1}^{m-2} c_k \binom{m-1}{k-1}x^{k-1} }  - x \Bigr) = -m \frac{\partial^2}{\partial x^2} \Bigl(  \frac{r_f(x)}{r_f'(x)} \Bigr) \\
& = -m \frac{\partial}{\partial x} \Bigl(  \frac{(r_f'(x))^2 - r_f(x) r_f''(x)}{(r_f'(x))^2} \Bigr) = m \frac{\partial}{\partial x} \Bigl(  \frac{  r_f(x) r_f''(x)}{(r_f'(x))^2} \Bigr) 
= m \frac{\partial}{\partial x} \alpha_f(x) >0.
\end{align*}Here, $\alpha_f(x) \coloneqq \frac{r_f(x) r_f''(x)}{(r_f'(x))^2}$ is the log-concavity ratio and the last inequality is due to Theorem~\hyperlink{T:2.1}{2.1}. So, rational function (\ref{eq:3.18}) is convex on $\{ x > x_1\}$. In addition, by Lemma~\hyperlink{L:2.11}{2.11} and Lemma~\hyperlink{L:2.12}{2.12}, we have $\tilde{x}_0(\nu_1) > \tilde{x}_0(\tau_1) > \tilde{x}_0(\mu_1) > x_1$. Thus, by the convexity of rational function (\ref{eq:3.18}), we obtain
\begin{align*}
\label{eq:3.19}
\frac{\mu_1(\tilde{x}_0(\nu_1) - \tilde{x}_0(\tau_1)) + \nu_1 (\tilde{x}_0(\tau_1) - \tilde{x}_0(\mu_1))      }{\tilde{x}_0(\nu_1) - \tilde{x}_0(\mu_1)} > \tau_1 = (1-s) \mu_1 +  s \nu_1. \tag{3.19}
\end{align*}By simplifying inequality (\ref{eq:3.19}), we get
\begin{align*}
(\mu_1 - \nu_1) \bigl(  (1-s) \tilde{x}_0(\mu_1) + s \tilde{x}_0 (\nu_1) - \tilde{x}_0(\tau_1) \bigr) > 0.
\end{align*}This implies that $(1-s) \tilde{x}_0(\mu_1) + s \tilde{x}_0 (\nu_1) > \tilde{x}_0(\tau_1)$ and $( (1-s) \tilde{x}_0(\mu) + s \tilde{x}_0(\nu), \cdots, (1-s) \tilde{x}_0(\mu) + s \tilde{x}_0(\nu)) \in \{ f(\tau_1, \bullet) \geq 0\} = \{ f( (1-s) \mu_1 + s \nu_1, \bullet) \geq 0\}$. 


If $\xi_j \neq 1$ for some $j \in \{2, \cdots, m-1\}$, for convenience, we say $\xi_{m-1} \in (0, 1)$. We consider the following set $\{f_{m-1} - \xi_{m-1} f_{m} \geq 0\}$ and we show that $\{f_{m-1} - \xi_{m-1} f_{m} \geq 0\}$ is convex. We have
\begin{align*}
\label{eq:3.20}
f_{m-1} &= \lambda_1   \cdots \lambda_{m-2} \lambda_m - \sum_{k=1}^{m-2} c_k \sigma_{k-1}(\lambda_{; m-1}) \tag{3.20} \\
&=    \lambda_{m} \Bigl(\lambda_1 \cdots \lambda_{m-2} - \sum_{k=2}^{m-2} c_k \sigma_{k-2}(\lambda_{; m-1, m}) \Bigr)   - \sum_{k=1}^{m-2} c_k \sigma_{k-1} (\lambda_{; m-1, m});           \\
\label{eq:3.21}
f_{m} &= \lambda_1 \cdots  \lambda_{m-1} - \sum_{k=1}^{m-2} c_k \sigma_{k-1}(\lambda_{; m}) \tag{3.21} \\
&=   \lambda_{m-1} \Bigl(\lambda_1 \cdots \lambda_{m-2}  - \sum_{k=2}^{m-2} c_k \sigma_{k-2}(\lambda_{; m-1, m}) \Bigr)   - \sum_{k=1}^{m-2} c_k \sigma_{k-1} (\lambda_{; m-1, m}).
\end{align*}
On the level set $\{ f_{m-1} - \xi_{m-1} f_{m} = 0\}$, by combining equations (\ref{eq:3.20}) and (\ref{eq:3.21}), we get
\begin{align*}
\lambda_{m} = \xi_{m-1} \lambda_{m-1} + (1- \xi_{m-1}) \frac{\sum_{k=1}^{m-2} c_k \sigma_{k-1} (\lambda_{; m-1, m})}{\lambda_1 \cdots  \lambda_{m-2} - \sum_{k=2}^{m-2} c_k \sigma_{k-2}(\lambda_{; m-1, m})}.
\end{align*}So, the level set $\{ f_{m-1} - \xi_{m-1} f_{m} = 0\}$ is actually a graph with domain $\mathbb{R}_{\lambda_{m-1}} \times \{ f_{m-1 m} > 0\}$.
Since $c  \in \tilde{\mathscr{C}}_m$ is in the generic strata and by mathematical induction, we have
\begin{align*}
\begin{pmatrix} 
\frac{\partial^2}{\partial \lambda_i \partial \lambda_j}  \Bigl (  \frac{\sum_{k=1}^{m-2} c_k \sigma_{k-1} (\lambda_{; m-1, m})}{\lambda_1 \cdots  \lambda_{m-2} - \sum_{k=2}^{m-2} c_k \sigma_{k-2}(\lambda_{; m-1, m})}   \Bigr)
\end{pmatrix}_{i, j \in \{1, \cdots, m-2\}} \geq 0.
\end{align*}Thus, the following $m-1 \times m-1$ matrix is also positive semi-definite: 
\begin{align*}
\begin{pmatrix}
\frac{\partial^2 \lambda_m}{\partial \lambda_i \partial \lambda_j} 
\end{pmatrix}_{i, j \in \{1, \cdots, m-1\}} &=
\begin{pmatrix}
\Bigl( \frac{\partial^2}{\partial \lambda_i \partial \lambda_j}  \Bigl (  \frac{\sum_{k=1}^{m-2} c_k \sigma_{k-1} (\lambda_{; m-1, m})}{\lambda_1 \cdots  \lambda_{m-2} - \sum_{k=2}^{m-2} c_k \sigma_{k-2}(\lambda_{; m-1, m})}   \Bigr)  \Bigr)_{i, j \in \{1, \cdots, m-2\}} & 0 \\
0 & 0
\end{pmatrix} \geq 0.
\end{align*}Hence, the graph $\{ f_{m-1} - \xi_{m-1} f_m = 0\}$ is convex.

Let $\mu, \nu \in   \{   f_{m-1} - \xi_{m-1} f_{m} \geq 0 \}$. We can view $(f_{m-1} - \xi_{m-1} f_{m})( (1-t) \mu + t \nu )$ as a polynomial in terms of $t$. If it is not a zero polynomial, then there are finitely many roots by the fundamental theorem of algebra. Suppose that there exists $\tilde{t} \in (0, 1)$ such that $(f_{m-1} - \xi_{m-1} f_{m})( (1-\tilde{t}) \mu + \tilde{t}\nu ) = 0$. Since $(f_{m-1} - \xi_{m-1} f_{m})( (1-t) \mu + t \nu )$ only has finitely many roots, there exists $t_0 < \tilde{t}$ and $t_1 > \tilde{t}$ such that $(f_{m-1} - \xi_{m-1} f_{m})( (1-t_0) \mu + t_0 \nu) > 0$ and $(f_{m-1} - \xi_{m-1} f_{m})(  (1-t_1) \mu + t_1 \nu) > 0$. For $\epsilon > 0$ sufficiently small, we get $(f_{m-1} - \xi_{m-1} f_{m})( (1-t_0)\mu_1 + t_0 \nu_1,   \cdots, (1-t_0) \mu_{m-1} + t_0 \nu_{m-1}, (1-t_0)\mu_m + t_0 \nu_m - \epsilon) > 0$ and $(f_{m-1} - \xi_{m-1} f_{m})( (1-\tilde{t})\mu_1 + \tilde{t} \nu_1,   \cdots, (1-\tilde{t}) \mu_{m-1} + \tilde{t} \nu_{m-1}, (1-\tilde{t}) \mu_m + \tilde{t} \nu_m - \frac{t_1 - \tilde{t}}{{t}_1 - t_0} \epsilon) < 0$. This leads to a contradiction because $\{ f_{m-1} - \xi_{m-1} f_m \geq 0\}$ is convex.

If $(f_{m-1} - \xi_{m-1} f_{m})( (1-t) \mu + t \nu)$ is a zero polynomial, then the line $\{ (1-t) \mu +  t \nu)   \colon t \in \mathbb{R}   \}$ lies in $\{ f_{m-1} - \xi_{m-1} f_{m} = 0 \}$. In particular, the line $\{ (1-t) \mu_1 +  t \nu_1, \cdots, (1-t) \mu_{m-2} +  t \nu_{m-2})   \colon t \in \mathbb{R}   \}$ lies in $\{ f_{m-1 \, m } > 0\}$, which implies that $\mu_i = \nu_i$ for all $i \in \{1, \cdots, m-2\}$. Hence, the line $\{ ((1-t) \mu_{m-1} +  t \nu_{m-1}, (1-t) \mu_{m} +  t \nu_{m})   \colon t \in \mathbb{R}   \}$ is actually the line
\begin{align*}
\lambda_m - \xi_{m-1} \lambda_{m-1} = (1- \xi_{m-1}) \frac{  \sum_{k=1}^{m-2} c_k \sigma_{k-1} (\mu_{; m-1, m})}{  \mu_1 \cdots \mu_{m-2}  - \sum_{k=2}^{m-2} c_k \sigma_{k-2}(\mu_{; m-1, m})}.
\end{align*}
In conclusion, if $(f_{m-1} - \xi_{m-1} f_{m})( (1-t) \mu + t \nu)$ is a zero polynomial, then $\mu_i = \nu_i$ for all $i \in \{1, \cdots, m-2\}$ and $(\nu_m - \mu_m) - \xi_{m-1} (\nu_{m-1} - \mu_{m-1}) = 0$.

At this local extremum, we have $\tilde{\tau} = (1-\tilde{s}) \tilde{\mu} + \tilde{s} \tilde{\nu}$. Since $\{ f_{m-1} - \xi_{m-1} f_m \geq 0\}$ is convex, $\mu, \nu \in \partial \{ f_{m-1} - \xi_{m-1} f_m \geq 0\}$ with $\mu_i \neq \nu_i$ for all $i \in \{1, \cdots, m\}$, and by above argument, we get $ (1- \tilde{s})  {\mu} + \tilde{s}  {\nu}  \in  \{ f_{m-1} - \xi_{m-1} f_m > 0\}$. Similar to before, we have $\{ f_{m-1} - \xi_{m-1} f_m \geq 0\} \subset   \{ f_{1 m-1} - \xi_{m-1} f_{1m} > 0\}$ and by the fact that $(  \tau_1, (1-\tilde{s}) \tilde{\mu} + \tilde{s} \tilde{\nu}  ) = ( (1-s)\mu_1 + s \nu_1, (1-\tilde{s}) \tilde{\mu} + \tilde{s} \tilde{\nu}  )  \in  \partial \{ f_{m-1} - \xi_{m-1} f_m \geq 0\}$.
Hence, $(1- \tilde{s})  {\mu}_1 + \tilde{s}  {\nu}_1 > (1-  {s})  {\mu}_1 +  {s}  {\nu}_1$, which implies that $s > \tilde{s}$. Thus, we confirm \bhyperlink{claim2}{\textbf{claim 2}} for all local extrema.\smallskip

Last, we study the asymptotic behavior, we show that if $(1-s)\mu_j + s \nu_j$ is sufficiently large for some $j \in \{2, \cdots, n\}$, then $ (1-s) \tilde{\mu} + s \tilde{\nu}  \in \{f(\tau_1, \bullet) \geq 0\} =  \{f( (1-s) \mu_1 + s \nu_1, \bullet) \geq 0\}$. We use mathematical induction on the degree $n$ to prove this. When $n = 2$, we consider $f(\lambda)= \lambda_1 \lambda_2 - c_0$ with $c_0 > 0$. By fixing $\mu_1 > \nu_1$ and let $\tau_1 = (1-s)\mu_1+ s \nu_1$, we have 
\begin{align*}
\tilde{s} = \frac{\tau_2 -  \mu_2}{\nu_2 - \mu_2} = \frac{\frac{c_0}{\tau_1} - \frac{c_0}{\mu_1}}{\frac{c_0}{\nu_1} - \frac{c_0}{\mu_1}} = \frac{ \mu_1 \nu_1 - \tau_1 \nu_1}{\mu_1 \tau_1 - \tau_1 \nu_1} = s \frac{\nu_1}{\tau_1} < s.
\end{align*}When $n = 3$, we consider $f(\lambda)= \lambda_1 \lambda_2 \lambda_3 - c_1(\lambda_1 + \lambda_2 + \lambda_3) - c_0$ with $(c_1, c_0)$ in the generic strata of $\tilde{\mathscr{C}}_3$. By fixing $\mu_1 > \nu_1$ and let $\tau_1 = (1-s)\mu_1+ s \nu_1$, we consider $(\mu_2, \mu_3) \in \{f(\mu_1, \bullet) = 0 \}$, $(\nu_2, \nu_3) \in \{f(\nu_1, \bullet) = 0 \}$, and $((1-s)\mu_2+s\nu_2, (1-s)\mu_3 + s\nu_3)$. By picking $N_1 > 0$ sufficiently large such that if $(1-s)\mu_2+ s \nu_2 \geq N_1$ and $(1-s) \mu_3 + s \nu_3 \geq N_1$, then
\begin{align*}
((1-s)\mu_2+ s \nu_2)( (1-s)\mu_3 + s \nu_3)- \frac{c_1}{\tau_1}  ( (1-s)(\mu_2+\mu_3)+ s(\nu_2+\nu_3)   ) - c_1 - \frac{c_0}{\tau_1}
&>   \frac{N_1^2}{2}   > 0.
\end{align*}Hence, $( (1-s)\mu_2+ s \nu_2, (1-s)\mu_3 + s \nu_3) \in \{ f( \tau_1, \bullet) = f( (1-s)\mu_1+ s \nu_1, \bullet) \geq 0\}$ in this case. Without loss of generality, we consider the case that $(1-s)\mu_2 + s \nu_2 \geq N_2 > N_1$ and $(1-s) \mu_3 + s \nu_3$ is bounded above by $N_1$, where $N_2$ will be determined later. In this case, by the fact that $\nu_3 > c_1/\nu_1$, $\mu_3 >  c_1 / \mu_1$, $ {(1-s)}/{\mu_1} +  {s}/{\nu_1}    -   {1}/{\tau_1}  > \epsilon_{s, \mu_1, \nu_1} > 0$, and $c_1 > 0$, we get
\begin{align*}
&\kern-2em  ((1-s)\mu_2+ s \nu_2)  ((1-s)\mu_3+s\nu_3) - \frac{c_1}{\tau_1}  ( (1-s)(\mu_2+\mu_3)+s(\nu_2+\nu_3)  ) - c_1 - \frac{c_0}{\tau_1}\\
&=   ((1-s)\mu_2 + s\nu_2) \Bigl(   ((1-s)\mu_3+s\nu_3) - \frac{c_1}{\tau_1} \Bigr) - \frac{c_1}{\tau_1}  ((1-s)\mu_3 + s\nu_3)    - c_1 - \frac{c_0}{\tau_1} \\
&>  c_1  ((1-s)\mu_2 + s \nu_2) \Bigl(    \frac{1-s}{\mu_1} + \frac{s}{\nu_1}    -  \frac{1}{\tau_1}   \Bigr)   - \frac{c_1}{\tau_1}   N_1   - c_1 - \frac{c_0}{\tau_1}  \geq    \frac{c_1 N_2 \epsilon_{s, \mu_1, \nu_1} }{2}       > 0,
\end{align*}provided that $N_2$ is sufficiently large. 
Hence, if $(1-s)\mu_j + s \nu_j \geq N_2$ for some $j \in \{2, 3\}$, then $( (1-s)\mu_2 + s\nu_2,  (1-s)\mu_3 + s \nu_3 ) \in \{ f( \tau_1, \bullet) = f((1-s)\mu_1 + s\nu_1, \bullet) \geq 0\}$ in this case. \smallskip

Suppose the statement is true when $n = m-1 \geq 3$. When $n = m \geq 4$, we consider the point $ (1-s) \tilde{\mu} + s \tilde{\nu}$. Similar to the above, we can pick $N_1 > 0$ sufficiently large such that if $(1-s)\mu_j + s \nu_j \geq N_1$ for all $j \in \{2, \cdots, m\}$, then $(1-s) \tilde{\mu} + s \tilde{\nu} \in \{ f( \tau_1, \bullet)  \geq 0\}$. Here, we iteratively show that for $k$ from $2$ to $m-1$, there exists $N_k > N_{k-1}$ sufficiently large such that if $(1-s)\mu_j+s\nu_j \geq N_k$ for all $j \in \{2, \cdots, m -k + 1 \}$ and $N_{k-1} \geq (1-s)\mu_j + s \nu_j$ for all $j \in \{m-k+2, \cdots, m\}$, then $(1-s) \tilde{\mu} + s \tilde{\nu}  \in \{ f( \tau_1, \bullet)   \geq 0\}$. When $k = 2$, by the previous argument, we consider the leading term of $f(\tau_1, (1-s)\tilde{\mu} + s \tilde{\nu} )$, that is, $f_{2 \cdots m-1}( \tau_1, (1-s)\tilde{\mu} + s \tilde{\nu} ) = \tau_1 ( (1-s) \mu_m+ s \nu_m) - c_{m-2}$. Since we assume that $c = (c_{m-2}, \cdots, c_1, c_0) \in \tilde{\mathscr{C}}_m$ is in the generic strata, so $c_{m-2} > 0$ and $f_{2 \cdots m-1}(\lambda) = \lambda_1 \lambda_m - c_{m-1}$ is strictly $\Upsilon$-stable. Similar to above, there exists $\epsilon_{s, \mu_1, \nu_1} > 0$ such that
\begin{align*}
&\kern-2em f_{2 \cdots m-1}( \tau_1, (1-s)\tilde{\mu} + s \tilde{\nu} ) \\
&= \tau_1 \Bigl(  (1-s) \mu_m + s \nu_m - \frac{c_{m-2}}{\tau_1} \Bigr) > c_{m-2} \tau_1 \Bigl(  \frac{1-s}{\mu_1} + \frac{s}{\nu_1}  - \frac{1}{\tau_1} \Bigr) > c_{m-2} \tau_1 \epsilon_{s, \mu_1, \nu_1} > 0.
\end{align*}
Hence, $f_{2 \cdots m-1}( \tau_1, (1-s)\tilde{\mu} + s \tilde{\nu} ) = \tau_1 ( (1-s)\mu_m + s \nu_m) - c_{m-2}$ has a uniform positive lower bound. This implies that there exists $N_2 > N_{1}$ sufficiently large such that if $(1-s)\mu_j+s\nu_j \geq N_2$ for all $j \in \{2, \cdots, m - 1 \}$ and $N_{1} \geq (1-s)\mu_m + s \nu_m$, then $(1-s) \tilde{\mu} + s \tilde{\nu}  \in \{ f( \tau_1, \bullet)   \geq 0\}$.

Suppose that when $k = l - 1 \geq 2$, the leading term $f_{2 \cdots m-l+2}( \tau_1, (1-s) \tilde{\mu} + s \tilde{\nu} )$ has a uniform positive lower bound. When $k = l$, we consider the leading term $f_{2 \cdots m-l+1}( \tau_1, (1-s) \tilde{\mu} + s \tilde{\nu} )$. Since $c  \in \tilde{\mathscr{C}}_m$ is in the generic strata, by mathematical induction, the set $\{f_{2 \cdots m-l+1}(\lambda) \geq 0\}$ is strictly convex. We consider the cross sections when $\lambda_1 = \mu_1$ and $\lambda_1 = \nu_1$, that is, $\{f_{2 \cdots m-l+1}(\mu_1, \bullet) \geq 0\}$ and $\{f_{2 \cdots m-l+1}(\nu_1, \bullet) \geq 0\}$. For any 
\begin{align*}
(\mu_{m-l+2}, \cdots, \mu_m ) \in \{f_{2 \cdots m-l+1}(\mu_1, \bullet) \geq 0\} \cap \Bigl[\frac{c_{m-2}}{\mu_1}, \frac{N_{l-1}}{1-s} \Bigr]^{l-1}
\end{align*}
and
\begin{align*}
(\nu_{m-l+2}, \cdots, \nu_m ) \in \{f_{2 \cdots m-l+1}(\nu_1, \bullet) \geq 0\} \cap \Bigl[\frac{c_{m-2}}{\nu_1}, \frac{N_{l-1}}{s} \Bigr]^{l-1}, 
\end{align*}
because the set $\{f_{2 \cdots m-l+1}(\lambda) \geq 0\}$ is strictly convex, we get 
\begin{align*}
f_{2 \cdots m-l+1}( (1-s) \mu_1+ s \nu_1, (1-s) \mu_{m-l+2}+ s \nu_{m-l+2}, \cdots, (1-s) \mu_m + s\nu_m) > 0. 
\end{align*}
In addition, since $\{f_{2 \cdots m-l+1}(\mu_1, \bullet) \geq 0\} \cap [c_{m-2}/\mu_1, N_{l-1}/(1-s)]^{l-1}$ and $\{f_{2 \cdots m-l+1}(\nu_1, \bullet) \geq 0\} \cap [c_{m-2}/\nu_1, N_{l-1}/s]^{l-1}$ are both compact subsets, we obtain
\begin{align*}
\inf \Set{ f_{2 \cdots m-l+1} ((1-s)\mu + s \nu)  | \scalebox{0.7}{ $  \begin{array}{l}
    (\mu_{m-l+2}, \cdots, \mu_m ) \in \{f_{2 \cdots m-l+1}(\mu_1, \bullet) \geq 0\} \cap \bigl[ \frac{c_{m-2}}{\mu_1}, \frac{N_{l-1}}{1-s} \bigr]^{l-1}  \\
 \text{ and }   (\nu_{m-l+2}, \cdots, \nu_m )  \in \{f_{2 \cdots m-l+1}(\nu_1, \bullet) \geq 0\} \cap \bigl [ \frac{c_{m-2}}{\nu_1}, \frac{N_{l-1}}{s} \bigr]^{l-1}
  \end{array} $} } > 0.
\end{align*}

This implies that the leading term $f_{2 \cdots m-l+1}( \tau_1, (1-s)\mu_{m-l+2}+s\nu_{m-l+2}, \cdots, (1-s)\mu_m+s\nu_m)$ has a uniform positive lower bound. By letting $N_l > N_{l-1}$ sufficiently large, for $(1-s)\mu_j+s\nu_j \geq N_l$ for all $j \in \{2, \cdots, m -l + 1 \}$ and $N_{l-1} \geq (1-s)\mu_j+s\nu_j$ for all $j \in \{m-l+2, \cdots, m\}$, we have $f((1-s)\mu +s\nu ) > 0$. Hence, $(1-s) \tilde{\mu} + s \tilde{\nu}  \in  \{f((1-s)\mu_1+s\nu_1, \bullet) \geq 0\}$ provided that $N_l$ is sufficiently large. By doing this process iteratively, if $(1-s)\mu_j+s\nu_j \geq N_{m-1}$ for some $j \in \{2, \cdots, m\}$, then $(1- s)\tilde{\mu} + s \tilde{\nu}  \in  \{f((1-s)\mu_1  + s  \nu_1, \bullet) \geq 0\}$, which finishes the proof.
\end{proof}

\begin{figure}
\begin{tikzpicture}[scale = 0.2]
\draw[->] (-0.5,0) -- (20,0) node[right] {$\lambda_2$};
\draw[->] (0,-0.5) -- (0,20) node[above] {$\lambda_3$};
\draw[color=red,domain=1.5:18,samples=100]    plot ( \x , {  (4*\x + 3)/(3*\x - 4)   })  node[right] { \tiny $\mu_1 =3$ }       ; 
\draw[color=red,domain=2.34:18,samples=100]    plot ( \x , {  (4*\x-1.4)/(1.9*\x - 4)   })  node[right] { \tiny $s = 1/2$ }       ; 
\draw[color=red,domain=6.4:18,samples=100]    plot ( \x , {  (5*\x-7.25)/(\x - 5)   })  node[right] { \tiny $\nu_1 = 0.8$ }       ; 
\draw[color=red,fill=red,thick] (9.21307, 9.21307) circle[radius= 0.5em]; 
\draw[color=red,fill=red,thick] ( {4.02758}, { 4.02758}) circle[radius= 0.5em];  
\draw[color=red,fill=red,thick] ( {3}, {3)}) circle[radius= 0.5em];  
\draw[<->,color=blue] (3,3) -- (4.02758,4.02758);
\draw[<->,color=blue] (9.21307,9.21307) -- (4.02758,4.02758);
\node at (3.2,4.1) {$\tilde{s}$};
\node at (5.8,7.8) {$1-\tilde{s}$};
\end{tikzpicture}
\caption{$\lambda_1 \lambda_2 \lambda_3 - 4(\lambda_1 + \lambda_2 + \lambda_3) + 9 = 0$ with $\nu_1 =0.8$, $\tau_1 = 1.9$, and $\mu_1 =3$.}
\label{fig:3.1}
\end{figure}

\hypertarget{R:3.1}{\begin{frmk}}
In Figure~\ref{fig:3.1}, we plot the level set $\lambda_1 \lambda_2 \lambda_3 - 4(\lambda_1 + \lambda_2 + \lambda_3) + 9 = 0$ with $\nu_1 =0.8$, $\tau_1 = 1.9$, and $\mu_1 =3$. $\tau_1 = 1.9$ when $s = 1/2$. The ratio $\tilde{s}:1-\tilde{s}$ is the ratio of the distance $\dist( (\mu_2, \mu_3), (\tau_2, \tau_3))$ to the distance $\dist( (\tau_2, \tau_3), (\nu_2, \nu_3))$.
\end{frmk}

We prove the following Lemma to end this section, which shows that Theorem~\hyperlink{T:2.1}{2.1} is equivalent to the positive definiteness of the Hessian matrix of $\lambda_n$ on the curve $\{ \lambda_1 = \cdots = \lambda_{n-1}\}$ on $\{f = 0\}$. The author believes that the following result gives a promising hope that Conjecture~\hyperlink{C:1.1}{1.1} is true.

\hypertarget{L:3.3}{\begin{flemma}}
Let $f(\lambda) \coloneqq  \lambda_1 \cdots \lambda_n - \sum_{k = 0}^{n-1} c_k \sigma_k(\lambda)$ be a general inverse $\sigma_k$ type multilinear polynomial. If the diagonal restriction $r_f(x)$ of $f$ is strictly right-Noetherian, then on the curve $\{\lambda_1 = \cdots = \lambda_{n-1}\}$ of the level set $\{ f  =  0 \}$ with $\lambda_1  > x_1$, the positive definiteness of the following $n-1 \times n-1$ Hessian matrix 
\begin{align*}
\Bigl(  \frac{\partial^2 \lambda_n}{\partial \lambda_i \partial \lambda_j}  \Bigr)_{i, j \in \{1, \cdots, n-1\}}
\end{align*}is equivalent to the monotonicity of log-concavity ratio of $r_f(x) = x^n - \sum_{k=0}^{n-1} c_k  \binom{n}{k} x^k$. Here, $x_1$ is the largest real root of $r'_f$.
\end{flemma}

\begin{proof}
For convenience, we assume that $c_{n-1} = 0$. By Lemma~\hyperlink{L:3.2}{3.2}, we show that the following matrix is positive-definite at every point on the curve $\{\lambda_1 = \cdots = \lambda_{n-1} = x \}$ with $x > x_1$:
\begingroup
\allowdisplaybreaks
\begin{align*}
\begin{pmatrix}
h_{ij} + h_{nn} \frac{h_i h_j}{h_n^2} - h_{in} \frac{h_j}{h_n} - h_{jn} \frac{h_i}{h_n}  
\end{pmatrix}_{i, j \in \{1, \cdots, n-1\}}.
\end{align*}
\endgroup
By Lemma~\hyperlink{L:3.2}{3.2}, we have
\begin{align*}
&\kern-1em h_{ij} + h_{nn} \frac{h_i h_j}{h_n^2} - h_{in} \frac{h_j}{h_n} - h_{jn} \frac{h_i}{h_n}  \\   
&= \frac{ ( C - \lambda_i C_{1; i}      )    \bigl(   C  - \lambda_j \lambda_n  C_{2; j, n}  \bigr) + ( C - \lambda_j C_{1; j}      )   \bigl(   C - \lambda_i \lambda_n   C_{2; i, n}  \bigr)   }{  \lambda_i \lambda_j \sigma_n(\lambda)   (  C -   \lambda_n C_{1;  n}      )  } \\
&\kern2em - \frac{     ( C - \lambda_n C_{1; n}      )   \bigl(  C - \lambda_i \lambda_j  C_{2; i, j}  \bigr)  }{  \lambda_i \lambda_j  \sigma_n(\lambda)  (  C -   \lambda_n C_{1;  n}      )  }(1- \delta_{ij}). \end{align*}

Now, it suffices to show that the following $n-1 \times n-1$ matrix is positive semi-definite
\begin{align*}
\label{eq:3.22}
\tilde{C} &\coloneqq
2 (x^{n-2}\lambda_n -  C_{1; 1})(x^{n-2} -   C_{2; 1,n}) \mathds{1}_{n-1 \times n-1} \tag{3.22} \\ 
&\kern2em-  (x^{n-1} -    C_{1; n})(x^{n-3} \lambda_n - C_{2; 1, 2}) 
\begin{pNiceMatrix}[nullify-dots,xdots/line-style = loosely dotted]
0 & 1 &    \Cdots & 1  \\
1 & 0 &   \Ddots &  \Vdots \\
\Vdots  & \Ddots &  \Ddots &1  \\
1 &    \Cdots  & 1& 0
\end{pNiceMatrix},
\end{align*}where $\mathds{1}_{n-1 \times n-1} $ is the $n-1 \times n-1$ all-ones matrix and we have
\begin{align*}
C_{1; 1} &= \sum_{k = 1}^{n-2} c_k \sigma_{k-1}(\lambda_{; 1}) = \sum_{k = 1}^{n-2} c_k \biggl( \mybinom[0.8]{n-2}{k-1}  x^{k-1}  + \mybinom[0.8]{n-2}{k-2}  x^{k-2} \lambda_n   \biggr); \\
C_{1; n} &= \sum_{k = 1}^{n-2} c_k \sigma_{k-1}(\lambda_{; n}) = \sum_{k = 1}^{n-2} c_k   \mybinom[0.8]{n-1}{k-1}  x^{k-1} ; \\
C_{2; 1, 2} &=  \sum_{k = 2}^{n-2} c_k \sigma_{k-2}(\lambda_{; 1, 2}) =  \sum_{k = 2}^{n-2} c_k  \biggl( \mybinom[0.8]{n-3}{k-2}  x^{k-2}  + \mybinom[0.8]{n-3}{k-3}  x^{k-3} \lambda_n   \biggr); \\
C_{2; 1, n} &=  \sum_{k = 2}^{n-2} c_k \sigma_{k-2}(\lambda_{; 1, n}) =  \sum_{k = 2}^{n-2} c_k    \mybinom[0.8]{n-2}{k-2}  x^{k-2}.
\end{align*}

By change of basis, to show $\tilde{C}$ is positive-definite, it is equivalent to showing that the following matrix is positive-definite
\begin{align*}
O^* \tilde{C} O,
\end{align*}where 
\begin{align*}
O \coloneqq  \begin{pNiceMatrix}[nullify-dots,xdots/line-style = loosely dotted]
1 & 1 & 1 & \Cdots  & 1 & 1 \\
 -1 & 0 & 0  & \Cdots & 0 & 1 \\
 0 & -1 & 0 & \Cdots & 0 &1 \\
  0 & 0& -1 & \Ddots & \Vdots & \\
\Vdots & \Vdots & \Ddots & \Ddots & 0 &  \Vdots \\
0 & 0 & \Cdots & 0& -1& 1
\end{pNiceMatrix}.
\end{align*}The column vectors of matrix $O$ are in fact the eigenvectors of $\mathds{1}_{n-1 \times n-1} $, which makes the new matrix $O^* \tilde{C} O$ and the computations simpler. We have
\begingroup
\allowdisplaybreaks
\begin{align*}
\label{eq:3.23}
 O^* \tilde{C} O  \tag{3.23}
&=  2 (x^{n-2}\lambda_n -  C_{1; 1})(x^{n-2} -   C_{2; 1,n}) \begin{pNiceMatrix}[nullify-dots,xdots/line-style = loosely dotted]
0_{n-2 \times n-2} &\vec{0} \\
\vec{0}^* & (n-1)^2
\end{pNiceMatrix} \\ 
&\kern2em - (x^{n-1} -    C_{1; n})(x^{n-3} \lambda_n - C_{2; 1, 2}) 
\begin{pNiceMatrix}[nullify-dots,xdots/line-style = loosely dotted]
-2 & -1 &  \Cdots & -1 & 0  \\
-1 & -2 &  \Ddots & \Vdots & 0 \\
\Vdots & \Ddots &   \Ddots & -1 & \Vdots  \\
-1 & \Cdots & -1 &   -2 &0 \\
0 & 0 &  \Cdots & 0 & (n-1)(n-2) 
\end{pNiceMatrix} \\
&= (x^{n-1} -    C_{1; n})(x^{n-3} \lambda_n - C_{2; 1, 2})    \begin{pNiceMatrix}[nullify-dots,xdots/line-style = loosely dotted]
2 & 1 &  \Cdots & 1 & 0  \\
1 & 2 &  \Ddots & \Vdots & 0 \\
\Vdots & \Ddots &   \Ddots & 1 & \Vdots  \\
1 & \Cdots & 1 &   2 &0 \\
0 & 0 &  \Cdots & 0 & 0
\end{pNiceMatrix} \\
&\kern2em + (n-1) \begin{pNiceMatrix}[nullify-dots,xdots/line-style = loosely dotted]
0 & 0 &  \Cdots & 0 & 0  \\
0 & 0 &  \Ddots & \Vdots & 0 \\
\Vdots & \Ddots &   \Ddots & 0 & \Vdots  \\
0 & \Cdots & 0 &   0 &0 \\
0 & 0 &  \Cdots & 0 & \begin{smallmatrix} 2(n-1) (x^{n-2}\lambda_n -  C_{1; 1})(x^{n-2} -   C_{2; 1,n}) \\ \kern2em- (n-2)  (x^{n-1} -    C_{1; n})(x^{n-3} \lambda_n - C_{2; 1, 2}) \end{smallmatrix}
\end{pNiceMatrix}.
\end{align*}
\endgroup
The $n-2 \times n-2$ matrix 
\begin{align*}
\begin{pNiceMatrix}[nullify-dots,xdots/line-style = loosely dotted]
2 & 1 &  \Cdots & 1  \\
1 & 2 &  \Cdots & 1  \\
\Vdots & \Vdots &   \Ddots &\Vdots  \\
1 & 1 &  \Cdots & 2
\end{pNiceMatrix}
\end{align*}is a positive-definite matrix with eigenvalues $\{   1, \cdots, 1, n-1  \}$. So to prove whether $O^* \tilde{C} O$ is positive-definite, it is equivalent to showing whether the following quantity is positive.
\begin{align*}
\label{eq:3.24}
 2(n-1) (x^{n-2}\lambda_n -  C_{1; 1})(x^{n-2} -   C_{2; 1,n})  - (n-2)  (x^{n-1} -    C_{1; n})(x^{n-3} \lambda_n - C_{2; 1, 2}). \tag{3.24}
\end{align*}Now, we use the equation itself, that is, 
\begin{align*}
\label{eq:3.25}
\lambda_n = \frac{\sum_{k = 0}^{n-2} c_k  \binom{n-1}{k}  x^{k}}{x^{n-1}  - \sum_{k = 1}^{n-2} c_k \binom{n-1}{k-1} x^{k-1}}. \tag{3.25}
\end{align*}Then by (\ref{eq:3.24}) and (\ref{eq:3.25}), we obtain the following 
\begingroup
\allowdisplaybreaks
\begin{align*}
\label{eq:3.26}
&\kern-1em  2(n-1) (x^{n-2}\lambda_n -  C_{1; 1})(x^{n-2} -   C_{2; 1,n})  - (n-2)  (x^{n-1} -    C_{1; n})(x^{n-3} \lambda_n - C_{2; 1, 2})   \tag{3.26} \\
&= 2(n-1) \biggl (  \lambda_n \biggl( x^{n-2} - \sum_{k = 2}^{n-2} c_k  \mybinom[0.8]{n-2}{k-2}  x^{k-2}  \biggr) -  \sum_{k = 1}^{n-2} c_k  \mybinom[0.8]{n-2}{k-1}  x^{k-1}      \biggr ) \\  
&\kern2em \times  \biggl(x^{n-2}   - \sum_{k = 2}^{n-2} c_k    \mybinom[0.8]{n-2}{k-2}  x^{k-2} \biggr  )  \\
&\kern2em - (n-2) \biggl ( \lambda_n \biggl(  x^{n-3} -   \sum_{k = 3}^{n-2} c_k \mybinom[0.8]{n-3}{k-3}  x^{k-3}     \biggr)   -   \sum_{k = 2}^{n-2} c_k   \mybinom[0.8]{n-3}{k-2}  x^{k-2}    \biggr ) \\
&\kern2em \times \biggl(x^{n-1} -  \sum_{k = 1}^{n-2} c_k   \mybinom[0.8]{n-1}{k-1}  x^{k-1} \biggr ) \\
&=  \Biggl(      2(n-1) \biggl( x^{n-2} - \sum_{k = 2}^{n-2} c_k  \mybinom[0.8]{n-2}{k-2}  x^{k-2}  \biggr)^2   \times \sum_{k = 0}^{n-2} c_k  \mybinom[0.8]{n-1}{k}  x^{k}   \\     
&\kern3em   -2(n-1)    \biggl( x^{n-1} - \sum_{k = 1}^{n-2} c_k  \mybinom[0.8]{n-1}{k-1}  x^{k-1}  \biggr) \biggl( x^{n-2} - \sum_{k = 2}^{n-2} c_k  \mybinom[0.8]{n-2}{k-2}  x^{k-2}  \biggr)   \\
 &\kern3em \times \sum_{k = 1}^{n-2} c_k  \mybinom[0.8]{n-2}{k-1}  x^{k-1}         \\      
 &\kern3em + (n-2)  \biggl( x^{n-1} - \sum_{k = 1}^{n-2} c_k  \mybinom[0.8]{n-1}{k-1}  x^{k-1}  \biggr)^2  \times \sum_{k = 2}^{n-2} c_k  \mybinom[0.8]{n-3}{k-2}  x^{k-2}   \\
 &\kern3em -(n-2)   \biggl( x^{n-1} - \sum_{k = 1}^{n-2} c_k  \mybinom[0.8]{n-1}{k-1}  x^{k-1}  \biggr)   \biggl( x^{n-3} - \sum_{k = 3}^{n-2} c_k  \mybinom[0.8]{n-3}{k-3}  x^{k-3}  \biggr) \\
  &\kern3em \times \sum_{k = 0}^{n-2} c_k  \mybinom[0.8]{n-1}{k}  x^{k}   \Biggr) \Bigg/ \biggl(    x^{n-1} - \sum_{k = 1}^{n-2} c_k \mybinom[0.8]{n-1}{k-1} x^{k-1}   \biggr).
\end{align*}
\endgroup

To simplify equation (\ref{eq:3.26}), if we consider the diagonal restriction $r_f(x) = x^n - \sum_{k = 0}^{n-2} c_k \binom{n}{k}x^k$, where $f = \lambda_1 \cdots \lambda_n - \sum_{k=0}^{n-2} c_k \sigma_k(\lambda)$, then we have
\begin{align*}
r_f^{(l)}(x) &= l!  \mybinom[0.8]{n}{l} \biggl (  x^{n-l} - \sum_{k = 0}^{n-2} c_k \mybinom[0.8]{n-l}{k-l}x^{k-l} \biggr)
\end{align*}and
\begin{align*}
\sum_{k = 0}^{n-2} c_k  \mybinom[0.8]{n- l -1}{k-l}  x^{k-l} &= \frac{x r_f^{(l+1)}(x)}{  (l+1)! \binom{n}{l+1} } - \frac{r_f^{(l)}(x)}{l!\binom{n}{l}}.
\end{align*}

So, we can rewrite equation (\ref{eq:3.26}) as
\begingroup
\allowdisplaybreaks
\begin{align*}
\label{eq:3.27}
&\kern-1em 2(n-1) (x^{n-2}\lambda_n -  C_{1; 1})(x^{n-2} -   C_{2; 1,n})  - (n-2)  (x^{n-1} -    C_{1; n})(x^{n-3} \lambda_n - C_{2; 1, 2}) \tag{3.27}  \\
&= \frac{n}{r_f'(x)} \Bigl(  2(n-1) \frac{r_f''(x)^2}{n^2(n-1)^2}   \Bigl(  \frac{x r_f'(x)}{n} - r_f(x)    \Bigr)  \\
&\kern5em - 2(n-1) \frac{r_f'(x)}{n}  \frac{r_f''(x)}{n(n-1)}  \Bigl(  \frac{x r_f''(x)}{n(n-1)} - \frac{r_f'(x)}{n}    \Bigr)  \\
&\kern5em - (n-2) \frac{r_f'(x)}{n} \frac{r_f'''(x)}{n(n-1)(n-2)}    \Bigl(  \frac{x r_f'(x)}{n} - r_f(x)    \Bigr)  \\ 
&\kern7em + (n-2) \frac{r_f'(x)^2}{n^2}    \Bigl(  \frac{x r_f'''(x)}{n(n-1)(n-2)} - \frac{r_f''(x)}{n(n-1)}    \Bigr) \Bigr) \\
&= \frac{1}{n(n-1) r_f'(x)} \bigl(  r_f(x) r_f'(x) r_f'''(x) +  r_f'(x)^2 r_f''(x)  -2 r_f(x) r_f''(x)^2      \bigr) \\
&= \frac{r_f'(x)^2}{n(n-1)  } \frac{\partial}{\partial x} \Bigl (   \frac{r_f(x) r_f''(x)}{r_f'(x)^2}      \Bigr ) = \frac{r_f'(x)^2}{n(n-1)  } \frac{\partial}{\partial x} \alpha_f(x).
\end{align*}
\endgroup 
Here, for notational convention, we denote $\alpha_{r_f}(x)$ as $\alpha_{f}(x)$. In conclusion, we have shown that on the curve $\{\lambda_1 = \cdots = \lambda_{n-1}\}$, the positive definiteness of the Hessian matrix of $\lambda_n$ is equivalent to the monotonicity of log-concavity ratio $\alpha_f$. 
\end{proof}

\section{Some Applications}
\label{sec:4}

In this section, we use our \bhyperlink{T:3.1}{convexity Theorem} to verify some examples. First, when the degree is low, the \bhyperlink{T:2.3}{Positivstellensatz Theorem} can be verified using the resultants and the discriminant. Here, we give a different proof of the Positivstellensatz results in \cite{lin2022}.

\hypertarget{D:4.1}{\begin{fdefi}[Resultant]}
The \textbf{resultant} of two univariate polynomials $p_1(x)$ and $p_2(x)$ is defined as the determinant of their Sylvester matrix. To be more precise, if we write
\begin{align*}
p_1(x) &= a_d x^d + a_{d-1} x^{d-1} + \cdots + a_0; \\
p_2(x) &= b_e x^e + b_{e-1} x^{e-1} + \cdots + b_0,
\end{align*}then the resultant of $p_1$ and $p_2$ is defined by the following.
\begin{align*}
\label{eq:4.1}
\res(p_1, p_2) \coloneqq \det  \begin{pNiceMatrix}[nullify-dots,xdots/line-style = loosely dotted]
a_d & 0 & \Cdots & 0  & b_e & 0 & \Cdots & 0 \\
a_{d-1} & a_d & \Cdots  & 0 & b_{e-1} & b_e & \Cdots & 0 \\
a_{d-2} & a_{d-1}  & \Ddots & 0 & b_{e-2} & b_{e-1} & \Ddots &  0 \\
\Vdots & \Vdots & \Ddots & a_d & \Vdots & \Vdots &  \Ddots & b_e \\
a_0 & a_{1}&  \Cdots & a_{d-1} &  b_0  & b_{1} & \Cdots & \Vdots \\
0 & a_0 & \Ddots & \Vdots  &  0  &  b_0  & \Ddots  & \Vdots \\
\Vdots & \Vdots & \Ddots & a_{1} & \Vdots & \Vdots &  \Ddots & b_{1} \\
0 & 0 & \Cdots & a_0 & 0 & 0 & \Cdots & b_0
\end{pNiceMatrix}. \tag{4.1}
\end{align*}
\end{fdefi}

\hypertarget{D:4.2}{\begin{fdefi}[Discriminant]}
Let $p(x) = a_n x^n + a_{n-1} x^{n-1} + \cdots + a_1 x + a_0$ be a polynomial of degree $n$ and the coefficient $a_0, \cdots, a_n$ are real numbers. The \textbf{discriminant} of $p$ is defined by
\begin{align*}
\label{eq:4.2}
\disc(p) \coloneqq \frac{(-1)^{n(n-1)/2}}{a_n} \res(p, p'). \tag{4.2}
\end{align*}
\end{fdefi}

\hypertarget{P:4.1}{\begin{fprop}}
The level set of the following general inverse $\sigma_k$ equations are all convex. 
\begin{align*}
\lambda_1 \lambda_2 - c_0 = 0, 
\end{align*} where $c_0 > 0$;
\begin{align*}
\lambda_1 \lambda_2 \lambda_3 - c_1(\lambda_1 + \lambda_2 + \lambda_3) - c_0 = 0, 
\end{align*} where $c_1 \geq 0$ and $c_0 > -2 c_1^{3/2}$;
\begin{align*} 
\lambda_1 \lambda_2 \lambda_3 \lambda_4 - c_2 \sigma_2(\lambda) - c_1 \sigma_1(\lambda) - c_0 = 0, 
\end{align*}where $c_2 \geq 0$, $c_1 \geq -2 c_2^{3/2}$, and $c_0 > -3c_2 x_1^2 -3c_1 x_1 $, where
\begin{align*}
x_1 &= 
\begin{dcases}
\sqrt[3]{c_1}&, \text{ when } c_2 = 0; \\
2 \sqrt{c_2} \cos \Bigl [ \frac{1}{3}   \arccos \Bigl(  \frac{c_1}{ 2c_2^{3/2}  }   \Bigr)         \Bigr ]&, \text{ when } c_2 > 0 \text{ and } 4c_2^3 - c_1^2 \geq 0; \\
2 \sqrt{c_2} \cosh \Bigl [ \frac{1}{3}   \arccosh \Bigl(  \frac{c_1}{ 2c_2^{3/2}  }   \Bigr)         \Bigr ]&, \text{ when } c_1, c_2 > 0 \text{ and } 4c_2^3 - c_1^2 \leq 0. \\
\end{dcases}
\end{align*}  
\end{fprop}

\begin{proof}
Here, we only prove the degree four case:
\begin{align*}
\lambda_1 \lambda_2 \lambda_3 \lambda_4 - c_2 \sigma_2(\lambda) - c_1 \sigma_1(\lambda) - c_0 = 0. 
\end{align*}First, the diagonal restriction and its derivatives (after dividing by the leading coefficient) will be
\begin{align*}
\{ x^4 - 6c_2 x^2 - 4c_1 x - c_0, \ x^3 - 3c_2 x - c_1, \ x^2 - c_2, \ x\}.
\end{align*}Second, for the largest real roots, if we denote the largest real root of $k$-th derivative by $x_k$, then we have $x_2 = \sqrt{c_2}$, $x_3 = 0$. Then, for the depressed cubic polynomial $x^3 - 3c_2 x - c_1$, we want
\begin{align*} 
x_2^3 - 3c_2 x_2 - c_1 = -2c_2^{3/2} - c_1 \leq  0.
\end{align*}That is, $c_1 \geq -2c_2^{3/2}$. We compute the discriminant of the cubic polynomial $x^3 - 3c_2 x - c_1$, by (\ref{eq:4.1}) and (\ref{eq:4.2}), we have
\begingroup
\allowdisplaybreaks
\begin{align*}
\disc (x^3 - 3c_2 x - c_1)  \tag{4.3}  &= \res\bigl( x^3 - 3c_2 x - c_1, 3x^2 - 3c_2 \bigr) \\
 &=  \det \begin{pNiceMatrix}[nullify-dots,xdots/line-style = loosely dotted]
1 & 0    & 3 & 0 & 0  \\
0 & 1    & 0 & 3 & 0  \\
-3c_2 & 0    & -3c_2 & 0 & 3 \\
-c_1 & -3c_2   & 0 & -3c_2 &  0  \\
0 & -c_1   &  0  & 0 & -3c_2 \\
\end{pNiceMatrix} \\ 
&= 27(4c_2^3 - c_1^2).
\end{align*}
\endgroup
When $4c_2^3 - c_1^2 \geq 0$, then this is the case \textit{casus irreducibilis}. When $4c_2^3 - c_1^2 \leq 0$, then the root can be represented using hyperbolic functions. So the largest real root $x_1$ will be
\begingroup
\allowdisplaybreaks
\begin{align*}
x_1 &= \begin{dcases}
\sqrt[3]{c_1}&, \text{ when } c_2 = 0; \\
2 \sqrt{c_2} \cos \Bigl [ \frac{1}{3}   \arccos \Bigl(  \frac{c_1}{ 2c_2^{3/2}  }   \Bigr)         \Bigr ]&, \text{ when } c_2 > 0 \text{ and } 4c_2^3 - c_1^2 \geq 0; \\
2 \sqrt{c_2} \cosh \Bigl [ \frac{1}{3}   \arccosh \Bigl(  \frac{c_1}{ 2c_2^{3/2}  }   \Bigr)         \Bigr ]&, \text{ when } c_1, c_2 > 0 \text{ and } 4c_2^3 - c_1^2 \leq 0. \\
\end{dcases}
\end{align*}
\endgroup
Here, we take the branch $\arccos(\bullet) \in [0, \pi]$ and $\arccosh$ is the inverse hyperbolic cosine. Last, we plug $x_1$ in to the quartic polynomial $x^4 - 6c_2 x^2 - 4c_1 x - c_0$. Because we want $x_0 > x_1$, so we want the following to be true. 
\begin{align*}
x_1^4 - 6c_2 x_1^2 - 4c_1 x_1 - c_0 = -3c_2 x_1^2 -3c_1 x_1 - c_0 < 0.
\end{align*}That is, $c_0 > -3c_2 x_1^2 -3c_1 x_1$.
\end{proof}

In \cite{guan2019class}, Guan--Zhang studied the solvability of a general class of curvature equations. These curvature equations can be viewed as generalizations of the equations for Christoffel--Minkowski problem in convex geometry. Guan--Zhang considered the following class of equations
\begin{align*}
\sigma_m (\lambda) + c_{m-1} \sigma_{m-1}(\lambda) =  \sum_{k = 0}^{m-2} c_k \sigma_k(\lambda),
\end{align*}where $n$ is the dimension of the space, $n \geq m \geq 2$, $c_k \geq 0$ for $k \in \{0, \cdots, m-2\}$, and $c_{m-1} \in \mathbb{R}$. They obtained a priori estimates for the admissible solutions. Here, for the special case $m = n$, without assuming the admissible condition, we can show that the level set is convex. The following result can also be applied to the general inverse $\sigma_k$ equations with non-negative coefficients considered by Collins--Székelyhidi \cite{collins2017convergence} and Fang--Lai--Ma \cite{fang2011class}.

\begin{flemma}
The level set of the following general inverse $\sigma_k$ equation 
\begin{align*}
\label{eq:4.4}
f(\lambda) \coloneqq \lambda_1 \cdots \lambda_n + c_{n-1} \sigma_{n-1}(\lambda)  - \sum_{k=0}^{n-2} c_k \sigma_k(\lambda) = 0 \tag{4.4}
\end{align*}is convex if $c_k \geq 0$ for $k \in \{0, \cdots, n-2\}$ with $\sum_{k=0}^{n-2} c_k > 0$ and $c_{n-1} \in \mathbb{R}$.
\end{flemma}

\begin{proof}
Consider the following diagonal restriction $r_f(x)$ of equation (\ref{eq:4.4}), that is,
\begin{align*}
\label{eq:4.5}
r_f(x) &= x^n + n c_{n-1} x^{n-1} - \sum_{k=0}^{n-2} c_k \mybinom[0.8]{n}{k} x^k. \tag{4.5}
\end{align*}By Theorem~\hyperlink{T:3.1}{3.1}, if $r_f$ is strictly right-Noetherian,  then we are done. We prove this by mathematical induction on the degree $n$. We also \hypertarget{L:4.1 claim}{\textbf{claim}} that when $n \geq 2$, if the coefficients $c_k$ satisfy the hypothesis, then $x_0 > 0$. When $n = 1$, we have $r_f(x) = x + c_0$, which is strictly right-Noetherian. When $n = 2$, we have $r_f(x) = x^2 + 2c_1 x - c_0;\  r_f'(x) = 2x + 2c_1$. If we write $x_1 = -c_1$ the largest real root of $r_f'$, then by the hypothesis $\sum_{k=0}^{2-2}c_k = c_0 > 0$, we get
\begin{align*}
r_f(-c_1) &= -c_1^2 - c_0 \leq -c_0 < 0.
\end{align*}This implies that $r_f$ is strictly right-Noetherian. Moreover, we have $r_f(0) = -c_0 <0$, which implies that $x_0 > 0$. So the \bhyperlink{L:4.1 claim}{claim} is true when $n = 2$.
When $n = 3$, we obtain
\begin{align*}
r_f(x) = x^3 + 3c_2 x^2 - 3c_1 x - c_0;\quad r_f'(x) = 3x^2 + 6c_2 x - 3c_1;\quad r_f''(x) = 6x + 6c_2.
\end{align*}We have $x_2 = -c_2$. If $c_1 > 0$, then $x_1 > \max\{x_2, 0\}$. In addition, we obtain
\begin{align*}
 r_f(x_1) 
&= x_1^3 + 3c_{2}x_1^{2} - 3c_1 x_1 - c_0  = - 0.5 x_1^3  - 1.5c_1 x_1 - c_0 < 0.
\end{align*}Thus, the largest real root $x_0$ of $r_f$ is greater than $x_1$, $r_f$ is strictly right-Noetherian. If $c_1 = 0$, then $x_1 = \max\{0, -2c_2\}$. If $c_2  < 0$, then similar to above, we get $ r_f(x_1)  < 0$. This implies that $x_0 > x_1$,  $r_f$ is strictly right-Noetherian. Otherwise, if $c_2  \geq 0$, then $x_1 = 0$. For this case, by the hypothesis, we have $\sum_{k=0}^{3-2}c_k = c_0 + c_1 = c_0 > 0$. This implies that 
\begin{align*}
r_f(x_1) = r_f(0) = - c_0 < 0.
\end{align*}Thus, $x_0 > x_1 = 0$, $r_f$ is again strictly right-Noetherian. No matter which case, the \bhyperlink{L:4.1 claim}{claim} is true. Suppose the statement and the \bhyperlink{L:4.1 claim}{claim} is true when $n = m-1$. When $n = m$, by equation (\ref{eq:4.5}), we have
\begin{align*}
r_f(x) = x^m + m c_{m-1} x^{m-1} - \sum_{k = 0}^{m-2} c_k \mybinom[0.8]{m}{k} x^k.
\end{align*}If we consider the first derivative of $r_f(x)$ with respect to $x$, then we obtain
\begin{align*}
r_f'(x) = m \Bigl (    x^{m-1} + (m-1) c_{m-1} x^{m-2} - \sum_{k = 1}^{m-2} c_k \mybinom[0.8]{m-1}{k-1} x^{k-1}       \Bigr).
\end{align*}There are two cases to be considered. First, if $\sum_{k=1}^{m-2} c_k > 0$, then $r_f'$ satisfies the hypothesis, so  $r_f'$ is strictly right-Noetherian. Moreover, the largest real root $x_1$ of $r_f'$ will be positive. Also,
\begin{align*}
 r_f(x_1) 
&= x_1^m + mc_{m-1}x_1^{m-1} - \sum_{k=0}^{m-2} c_k \mybinom[0.8]{m}{k}x_1^k \\
&= x_1^m  - \sum_{k=0}^{m-2} c_k \mybinom[0.8]{m}{k}x_1^k - \frac{m}{m-1}x_1^m + \sum_{k=1}^{m-2}c_k \frac{m}{m-1} \mybinom[0.8]{m-1}{k-1} x_1^k \\
&= -\frac{1}{m-1} x_1^m  - \sum_{k=0}^{m-2}c_k \Bigl (1 - \frac{k}{m-1} \Bigr)  \mybinom[0.8]{m}{k} x_1^k < 0.
\end{align*}In this case, $x_0 > x_1 > 0$, $r_f$ is strictly right-Noetherian. Second, if $\sum_{k=1}^{m-2} c_k = 0$, then $c_k = 0$ for all $k \in \{1, \cdots, m-2\}$. By hypothesis, we have $c_0 > 0$, so $r_f(x) = x^m + mc_{m-1} x^{m-1} -c_0$. For $k \in \{1, \cdots, m-1\}$, we have $x_k = \max \{0, -(m-k)c_{m-1}\}$. We are done if $c_{m-1} < 0$. Otherwise, we have $x_1 = \cdots = x_{m-1} =  0$ and $x_0 > x_1 = 0$. Hence, no matter which case, $r_f$ is strictly right-Noetherian, and the \bhyperlink{L:4.1 claim}{claim} is true. This finishes the proof.
\end{proof}

\begin{flemma}
The level set of the deformed Hermitian--Yang--Mills equation
\begin{align*}
\label{eq:4.6}
\Im \bigl ( \omega + \sqrt{-1} \chi \bigr )^n &= \tan \bigl (   {\theta}   \bigr ) \cdot \Re \bigl ( \omega + \sqrt{-1} \chi \bigr )^n \tag{4.6}
\end{align*}is convex if $\theta$ is in the supercritical phase, that is, $\theta \in \bigl((n-2)\pi/2, n \pi/2 \bigr)$. In addition, the level set is also convex if $\theta  \in \bigl( -n \pi/2, -(n-2)\pi/2 \bigr)$.
\end{flemma}

\begin{proof}
First, it is well-known that the dHYM equation (\ref{eq:4.6}) can be rewritten as the following equation
\begin{align*}
\sum_{i = 1}^n \arctan(\lambda_i ) = \theta.
\end{align*}Since $\theta \in \bigl((n-2)\pi/2, n \pi/2 \bigr)$, we have 
\begin{align*}
\frac{\theta - k\pi/2}{n-k} \in  \Bigl( \frac{(n-2-k)\pi}{2(n-k)},  \frac{(n-k)\pi}{2(n-k)}  \Bigr) =  \Bigl(   \frac{\pi}{2} - \frac{\pi}{n-k},  \frac{\pi}{2}  \Bigr) \mbox{\LARGE$\subset$}  \Bigl(   - \frac{\pi}{2} ,  \frac{\pi}{2}  \Bigr)
\end{align*}for $k \in \{0, 1, \cdots, n-1\}$. Second, by Theorem~\hyperlink{T:3.1}{3.1}, we consider the diagonal restriction, we get
\begin{align*}
(n-k) \arctan(x_k) = \theta - \frac{k \pi}{2}
\end{align*}for $k \in \{0, 1, \cdots, n-1\}$, where $x_k$ is the largest real root of the $k$-th derivative of the diagonal restriction. We claim that:
\begin{align*}
\label{eq:4.7}
x_0 = \tan \Bigl ( \frac{\theta}{n}  \Bigr) >  \cdots > x_k =  \tan \Bigl ( \frac{\theta - k \pi/2}{n-k}  \Bigr) > \cdots > x_{n-1} =  \tan \Bigl( \theta - (n-1) \frac{\pi}{2} \Bigr).  \tag{4.7}
\end{align*}

Since $\tan(x)$ is increasing on $(-\pi/2, \pi/2)$, to check (\ref{eq:4.7}), it suffices to check whether 
\begin{align*}
  \frac{\theta}{n}    >  \frac{\theta - \pi/2}{n-1}   > \cdots >    \frac{\theta - k \pi/2}{n-k}    > \cdots >  \theta - (n-1) \frac{\pi}{2}. 
\end{align*}This is true because the function $f(x) = (\theta - x\pi/2)/(n-x)$ is decreasing on $(-\infty, n)$. By Theorem~\hyperlink{T:3.1}{3.1}, the level set is convex. \smallskip

In addition, we see that the dHYM equation is real-rooted, by Proposition~\hyperlink{P:2.2}{2.2}, the dHYM equation is both strictly right-Noetherian and strictly left-Noetherian. So if $\theta \in \bigl( -n \pi/2, -(n-2)\pi/2 \bigr)$, then the level set will also be convex.
\end{proof}

\vspace{2cm}
\Address

\end{document}